\documentclass[12pt,oneside,english]{amsart}
\usepackage{lmodern}
\usepackage{helvet}

\usepackage[T1]{fontenc}
\usepackage[latin9]{inputenc}
\usepackage{amsthm}
\usepackage{amstext}
\usepackage{amssymb}
\usepackage{graphicx}
\usepackage{verbatim}
\makeatletter

\newcommand{\noun}[1]{\textsc{#1}}

\numberwithin{equation}{section}
\numberwithin{figure}{section}
  \theoremstyle{plain}
  \newtheorem*{thm*}{\protect\theoremname}
  \theoremstyle{plain}
  \newtheorem*{prop*}{\protect\propositionname}
\theoremstyle{plain}
\newtheorem{thm}{\protect\theoremname}[section]
  \theoremstyle{plain}
  \newtheorem{lem}[thm]{\protect\lemmaname}
  \theoremstyle{remark}
  \newtheorem{rem}[thm]{\protect\remarkname}
  \theoremstyle{definition}
  \newtheorem{defn}[thm]{\protect\definitionname}
  \theoremstyle{plain}
  \newtheorem{prop}[thm]{\protect\propositionname}
  \theoremstyle{plain}
  \newtheorem{cor}[thm]{\protect\corollaryname}
  \theoremstyle{definition}
  \newtheorem{example}[thm]{\protect\examplename}

\usepackage{amsfonts}
\usepackage{amssymb}
\usepackage[all]{xy}
\usepackage{array}
\usepackage{lmodern}
\usepackage[T1]{fontenc}

\def\QQ{\mathbb{Q}}
\def\RR{\mathbb{R}}
\def\CC{\mathbb{C}}
\def\UU{\mathbb{U}}
\def\FF{\mathbb{F}}
\def\ZZ{\mathbb{Z}}
\def\PP{\mathbb{P}}
\def\sl{\mathfrak{sl}_2}
\def\vf{\varphi}
\def\vft{\tilde{\vf}}
\def\lm{\mathfrak{m}}
\def\ad{\text{ad}}
\def\Ad{\text{Ad}}
\def\lz{\mathfrak{z}}
\def\lq{\mathfrak{q}}
\def\lp{\mathfrak{p}}
\def\lg{\mathfrak{g}}
\def\sur{\twoheadrightarrow}

\theoremstyle{definition}

\theoremstyle{definition}

\theoremstyle{theorem}
\newtheorem*{jm}{Jacobson-Morosov}
\theoremstyle{theorem}
\newtheorem*{kostant}{Kostant}
\theoremstyle{theorem}
\newtheorem*{schmid}{Schmid}
\theoremstyle{definition}
\newtheorem*{thx}{Acknowledgments}

\newcommand*{\@old@slash}{}\let\@old@slash\slash
\def\slash{\relax\ifmmode\delimiter"502F30E\mathopen{}\else\@old@slash\fi}

\makeatother
\def\backslash{\delimiter"526E30F\mathopen{}}
\usepackage{babel}
  \providecommand{\corollaryname}{Corollary}
  \providecommand{\definitionname}{Definition}
  \providecommand{\examplename}{Example}
  \providecommand{\lemmaname}{Lemma}
  \providecommand{\propositionname}{Proposition}
  \providecommand{\remarkname}{Remark}
  \providecommand{\theoremname}{Theorem}
\providecommand{\theoremname}{Theorem}

\begin{document}

\title{Boundary components of Mumford-Tate domains}

\author{Matt Kerr and Gregory Pearlstein}

\begin{abstract}
We study certain spaces of nilpotent orbits in the Hodge domains introduced
by \cite{GGK}, and treat a number of examples. More precisely, we
compute the Mumford-Tate group of the limit mixed Hodge structure
of a generic such orbit. The result is used to present these spaces
as iteratively fibered algebraic-group orbits in a minimal way. 
\end{abstract}
\maketitle

\section{Introduction}

Mumford-Tate groups are the natural $\QQ$-algebraic symmetry groups
of Hodge structures, in the sense of stabilizing the Hodge substructures
of all tensor powers of a Hodge structure and its dual. The modern
definition, for Hodge structures of arbitrary weight, is due to Deligne
\cite{De3}, with a generalization to mixed Hodge structures by Deligne (unpublished) and Andr\'e
\cite{An}. They were originally introduced by Mumford and Tate \cite{Mu}
to provide a Hodge-theoretic characterization of certain families
of abelian varieties studied by Kuga and Shimura. In that ``classical''
context of weight one Hodge structures, they led to Deligne's reformulation
of Shimura varieties in terms of reductive groups (cf. \cite{De5,Mi,Ke}) and his subsequent
proof that Hodge cycles on abelian varieties are absolute (cf. \cite{De}). 
Hence they already have deep and longstanding ties
with Hodge theory and algebraic geometry on the one hand, and with
the theory of automorphic representations on the other. Amplifying
the latter point, it turns out that (up to isogeny) the semisimple
real Lie groups admitting discrete series representations are precisely
those which arise as the group of real points $G(\mathbb{R})$ of
a Mumford-Tate group of a pure Hodge structure.

By the work of Griffiths \cite{Gr}, the cohomology of a family of
smooth complex projective varieties $\mathcal{X}\to\mathcal{S}$ determines
a period map 
\[
\Phi:\mathcal{S}\to\Gamma\backslash D
\]
to a discrete quotient of a \emph{period domain} $D$ -- that is,
a classifying space for pure polarized Hodge structures with given
Hodge numbers. As in the theory of Shimura varieties, it is more natural
to allow $D$ to be a general \emph{Mumford-Tate domain}, parametrizing
only those Hodge structures with the same ``symmetries'' as those
in $\Phi$. (This is a homogeneous space for the corresponding $G(\mathbb{R})$.)
Whenever $D$ is Hermitian symmetric, the quotient $\Gamma\backslash D$
admits the toroidal compactifications of \cite{AMRT}. In the special
case of weight 1 period domains (Siegel spaces), this compactifaction
was reinterpreted by \cite{CCK} in terms of the adjunction of Hodge-theoretic
boundary components $B(\sigma)$. The latter parametrize the possible
limiting mixed Hodge structures (LMHS) arising from the degeneration
of $\Phi$ at a boundary point $x\in\bar{\mathcal{S}}\backslash\mathcal{S}$,
where the family of varieties undergoes a degeneration. Kato and Usui
recently extended this approach in \cite{KU} to obtain partial compactifications
$\Gamma\backslash D_{\Sigma}$ for higher weight period domains.

The purpose of the present paper is to study degenerations of Hodge
structure in the more general context of arbitrary Mumford-Tate domains,
with an emphasis on the higher weight setting. More precisely, we
shall introduce boundary components classifying the possible nilpotent
orbits, and present them group-theoretically (as double-coset spaces)
in a ``minimal'' way. In marked contrast to the Shimura case, when
$D$ is non-Hermitian the quotients $\Gamma\backslash D$ are usually
not algebraic varieties \cite{GRT}. Nevertheless, it has become clear
in recent years that these quotients support a rich geometry -- including
a dense honeycomb of Shimura subvarieties and CM points -- which has
begun to be explored in \cite{C,C1,C2}, \cite{FL}, \cite{FHW},
\cite{GGK}, and \cite{Ro1,Ro2}.

There are several compelling reasons to introduce and study their
boundary components. One is to make predictions about the form (and
arithmetic nature) of the LMHS for variations of Hodge structure with
given symmetries, and use this information in turn to classify the
possible period maps. Indeed, a variation with fixed underlying local
system is determined by a single LMHS, and in keeping with the general
simplifying effect of degenerating in algebraic geometry, the latter
is often more accessible to computation. Furthermore, while $\Gamma\backslash D$
may be non-algebraic, the corresponding discrete quotient \emph{of
a boundary component} may not only be algebraic, but may be \emph{as
simple as a CM abelian variety or elliptic modular surface}.

From the standpoint of automorphic representation theory, this is
significant: it raises the possibility of putting a $\bar{\QQ}$-algebraic
structure on the coherent cohomology of a non-classical $\Gamma\backslash D$
(see the remarks on Fourier coefficients below). One case which surely
warrants further investigation is the split-$G_{2}$-domain parametrizing
Hodge structures of weight $6$ with Hodge numbers $h^{6,0}=h^{5,1}=\cdots=h^{0,6}=1$.
It has three types of boundary component quotients, two of which --
isomorphic to $\CC^{*}$ resp. a (noncompact) elliptic modular surface
-- are algebraic. Moreover, the totally degenerate limit of discrete
series for $G_{2}$ contributes to its automorphic cohomology.

Beyond clarifying the relationship between LMHS and Hodge representations,
our construction of boundary components in this paper has several
applications to Hodge theory and complex geometry, some of which 
are treated in follow-up works (based on this paper). In particular, in \cite{KP2,KR} we
use it to interpret Hodge-theoretically ``most'' of the $G(\RR)$-orbits
in the topological boundary of $D$ in its compact dual $\check{D}$.
Meanwhile, \cite{GGK4} provides a general introductory discussion
of $G(\RR)$-orbits in $\check{D}$ and the cycle space $\check{\mathcal{U}}$,
combining the root-weight analysis of such in \cite{FHW} with the
Hodge-theoretic interpretation.

\section*{Summary of Results}

Let $D=G(\RR)/H$ be a Mumford-Tate domain and $\Gamma\leq G(\QQ)$
be a torsion-free, non-co-compact congruence subgroup. Then the boundary
components $B(\sigma)$ we consider parametrize all nilpotent orbits
subordinate to a given nilpotent cone $\sigma$. Computing the Mumford-Tate
group\footnote{Here our conventions differ with some of the literature: see the
``Concluding Remarks'' in this Introduction.} of the generic limit mixed Hodge structure parametrized by a
boundary component, in particular, is a natural and unexpectedly subtle
problem which arose during the writing of \cite{GGK} (cf. $\S$I.C).
Our solution is given in Theorem \ref{thm1} for $\dim(\sigma)=1$
($\sigma=\mathbb{R}_{\geq0}\langle N\rangle$, $N\in\mathfrak{g}_{\mathbb{Q}}$),
and may be interpreted as follows:
\begin{thm*}
Let $D=G(\RR)/H$ be a M-T domain, and assume the set $B(N)$ of $N$-nilpotent
orbits $e^{\tau N}F^{\bullet}$ in $\check{D}$ is nonempty. Then
the M-T group $M_{B(N)}$ of the LMHS $(F^{\bullet},W(N)_{\bullet})$
of a general such orbit is a subgroup of the centralizer of $N$ in
$G$, admitting a filtration by normal subgroups $M_{B(N)}=W_{0}M_{B(N)}\trianglerighteq W_{-1}M_{B(N)}\trianglerighteq\cdots$
with\\
(i) $Gr_{0}^{W}M_{B(N)}=:G_{B(N)}$ reductive, and equal to the M-T
group of a general $\varphi_{split}:=\oplus_{i}Gr_{i}^{W(N)}F^{\bullet}$;\\
(ii) $Gr_{k}^{W}M_{B(N)}$ abelian for $k<0$; and\\
(iii) $W_{-1}M_{B(N)}=M_{N}:=\exp\left\{ \text{im}(\text{ad}N)\cap\ker(\text{ad}N)\right\} $
its unipotent radical.\\
Moreover, $M_{N}(\mathbb{C})\rtimes G_{B(N)}(\mathbb{R})$ acts transitively
on the corresponding set $\tilde{B}(N)$ of LMHS.
\end{thm*}
Theorem \ref{mainthm} generalizes the above statements beyond the rank one
case.  When $D$ is Hermitian symmetric, these results are related to work
of Milne and Pink on the mixed Shimura varieties involved in the toroidal 
compactifications of $\Gamma\backslash D$ \cite{Mi,Pi}.  Indeed, the quotients
$\left( \Gamma\cap\text{stab}\langle\sigma\rangle\right)\backslash B(\sigma)$
recover these classical boundary components, in effect extending and completing
the work of \cite{CCK} (cf. Remark \ref{HSD remark}).  So the boundary objects
considered in this paper are a common generalization of the Kato-Usui components
and the toroidal ones.

One motivation for studying them (and pursuing Theorem \ref{mainthm}) was to give a precise relation between
the $B(\sigma)$'s and classifying spaces for polarized MHS of the
type considered by Hertling \cite{He},\footnote{These spaces are also homogeneous manifolds with respect to the semi-direct product of a real reductive group and a complex unipotent group, a phenomenon first observed in the earlier studies \cite{Car,Us,SSU}.} and use this to understand
their structure. This is carried out below in the first half of $\S$7,
and leads to the following for $\sigma=\mathbb{R}_{\geq0}\langle N\rangle$:
\begin{thm*}
If $\Gamma$ is neat, then $\overline{B(N)}:=\left(\Gamma\cap\text{stab}\langle N\rangle\right)\backslash B(N)$
admits a $C^{\infty}$ fibration tower
\[
\overline{B(N)}\twoheadrightarrow\cdots\twoheadrightarrow\overline{B(N)}_{(k)}\overset{\overline{\rho}_{(k)}}{\twoheadrightarrow}\cdots\twoheadrightarrow\overline{B(N)}_{(1)}\overset{\overline{\rho}_{(1)}}{\twoheadrightarrow}\overline{D(N)}
\]
with (for $k>1$, generalized) intermediate Jacobian fibers at each
stage, and base $\overline{D(N)}$ a discrete quotient of the M-T
domain $G_{B(N)}(\mathbb{R}).\varphi_{split}$.
\end{thm*}
Another piece of motivation comes from the study of automorphic cohomology
in \cite{C} and \cite{GGKauto}. Let $D$ be a nonclassical Mumford-Tate
domain, where ``nonclassical'' is meant in the strong sense of possessing
no nonconstant holomorphic or antiholomorphic fibration over a Hermitian
symmetric domain. Then with $\Gamma$ as above, the quotient $\Gamma\backslash D$
is non-algebraic \cite{GRT}.%
\footnote{In the case where $\Gamma$ is instead co-compact, this has been known
in certain cases for some time \cite{CT}.%
} In particular, one cannot use sections of nontrivial holomorphic
homogeneous line bundles $\mathcal{L}$ (such as $K_{D}^{\otimes m}$)
to embed $\Gamma\backslash D$ in projective space; in fact, it is
strongly suspected that all spaces $H^{0}(D,\mathcal{O}(\mathcal{L}))^{\Gamma}$
of holomorphic automorphic forms are zero. Instead, one can consider
the higher automorphic cohomology groups $H^{p}(D,\mathcal{O}(\mathcal{L}))^{\Gamma}$
introduced in \cite{GS}. In the special case studied by Carayol \cite{C}
(also see Example \ref{exa:(Carayol's-example-)} below), in spite
of the non-algebraicity of $\Gamma\backslash G(\RR)/H$, the boundary
component quotients turn out to be isomorphic either to $\mathbb{C}^{*}$,
a CM elliptic curve $E$, or its conjugate $\bar{E}$. For certain
spaces of automorphic cohomology classes, he defines \emph{generalized
Fourier coefficients} in $\{H^{1}(E,\mathcal{O}(-n))\}_{n\in\mathbb{N}}$,
and proves that the space is spanned over $\mathbb{C}$ by classes
for which all of these coefficients are defined over $\bar{\QQ}$.
Such an arithmetic structure is potentially very useful from the standpoint
of proving algebraicity of the Hecke eigenvalues of automorphic representations
not appearing in the coherent cohomology of any Shimura variety.

Consequently one desires, as a first step, some simple conditions
under which $\Gamma_{\sigma}\backslash B(\sigma)$ ($\Gamma_{\sigma}
:=\Gamma\cap \text{stab}(\sigma)$) has a canonical
model over $\bar{\QQ}$. Such conditions are laid out in Proposition
\ref{prop hatch} and Theorem \ref{thm structure}, whose corollaries
include the
\begin{prop*}
Let $M$ be an adjoint $\mathbb{Q}$-algebraic group, $D=M(\mathbb{R})/H$
a Hodge domain parametrizing level $2\ell$ Hodge structures on $\mathfrak{m}=Lie(M)$,
and $B(N)\neq\emptyset$. If\\
(i) $\ker(\text{ad}N)\cap\text{im}\left\{ (\text{ad}N)^{2}\right\} =\langle N\rangle$
and\\
(ii) $[\ker(\text{ad}N),\ker(\text{ad}N)]\subseteq\text{im}(\text{ad}N)$\\
hold, then $\overline{B(N)}$ is a CM abelian variety\footnote{We need not have
$Gr^{W(N)}_{-1}\mathfrak{m}\cong H^1(\overline{B(N)})(1)$, since it may
be of type (say) $(-2,1)+(1,-2)$ as in Carayol's example.} of dimension
$\frac{1}{2}\dim\left(Gr_{-1}^{W(N)}\mathfrak{m}\right)$ over $\CC$.
\end{prop*}
At least for $G$ of Hermitian type, we expect that (as in Carayol's
example) one can use such boundary components to layer a notion of
arithmeticity on to a subspace of automorphic cohomology by asking
for those classes whose generalized Fourier coefficients are defined
over $\bar{\QQ}$. Though we won't pursue this application here, it
has influenced our decision to focus on CM abelian varieties, cf.
Remark \ref{rem7.12}.

Our results on the structure of $\overline{B(N)}$ are put to work
on examples in $\S$8. There, we first consider cases where $D$ is
a period domain for Hodge structures of weight 1, 2, and 3 (with $G=Sp_{4}$
or $SO(4,1)$), the last of which was the focus of \cite{GGK2}. Turning
to subdomains, we then treat Carayol's weight 3 example (with $G=U(2,1)$
and Hodge numbers $(1,2,2,1)$) and an ``exceptional'' domain for
weight 2 Hodge structures with Mumford-Tate group $G_{2}$ and Hodge
numbers $(2,3,2)$. The boundary components for two other $G_{2}$-related
examples, including the one mentioned earlier, will be worked out
from an alternative viewpoint in \cite{KP2}.

To carry out this program one also needs to know how the $\{\Gamma_{\sigma}
\backslash B(\sigma)\}$ are adjoined to $\Gamma \backslash D$.  Theorem \ref{thm Ku}
and its proof show how to generalize the results of \cite{KU} on logarithmic partial 
compactifications to the Mumford-Tate setting.

It is natural to expect that the new technology of Mumford-Tate domains
and their boundary components should have implications for variations
of Hodge structures, say over a fixed curve $\mathcal{S}$. In $\S10$
we obtain some refined rigidity and Arakelov-type finiteness results
for VHS on $\mathcal{S}$ with fixed local system, Hodge numbers and
M-T group. We make the elementary observation that they might be classified
by their values in a single boundary component, which already in the
rigid case should be of arithmetic interest. To a Hodge theorist,
this application should legitimize our construction all on its own!
A key example is the geometric $G_{2}$-VHS (with Hodge numbers $(1,1,1,1,1,1,1)$)
arising in the work of Dettweiler and Reiter \cite{DR} (partly based
on \cite{Ka}), which is the subject of $\S9$.

\section*{Hodge Domains}

Let $G$ be an algebraic group defined over $\QQ$ and $k\subseteq\CC$
a field; then the group of $k$-rational points will be denoted $G(k)$.%
\footnote{A concise review of algebraic groups may be found in \cite[sec. I.A]{Ke}.%
} When $k$ is $\RR$ (resp. $\CC$), this will be regarded as a real
(resp. complex) Lie group with Lie algebra $\lg$ (resp. $\lg_{\CC}$).%
\footnote{We will write $\lg_{\QQ}$ for the underlying $\QQ$-vector space,
and $\lg_{k}:=\lg_{\QQ}\otimes_{\QQ}k$.%
} Two examples which play a distinguished role below are the algebraic
tori $\UU\subset\mathbb{S}$ inside $GL_{2}$ with $k$-rational points
\[
\UU(k)=\left\{ \left.\left(\begin{array}{cc}
a & b\\
-b & a
\end{array}\right)\right|\begin{array}{c}
a^{2}+b^{2}=1\\
a,b\in k
\end{array}\right\} 
\]
\[
\mathbb{S}(k)=\left\{ \left.\left(\begin{array}{cc}
a & b\\
-b & a
\end{array}\right)\right|\begin{array}{c}
a^{2}+b^{2}\neq0\\
a,b\in k
\end{array}\right\} .
\]

Next recall where a Mumford-Tate domain comes from: let
\begin{itemize}
\item $V$ be a $\QQ$-vector space;
\item $Q$ be a $(-1)^{n}$-symmetric bilinear form on $V$;
\item $'\vf:\UU\to Aut(V,Q)$ be a morphism of algebraic groups defined
over $\RR$, with $'\vf(-1)=(-1)^{n}\text{id}_{V}$ and $Q(\cdot,{}'\vf(i)\overline{\cdot})>0$;
and
\item $'G\leq Aut(V,Q)$ be the $\QQ$-algebraic group closure of $'\vf(\UU(\RR))$.
\end{itemize}
Equivalently, $'\vf$ is a weight $n$ Hodge structure on $V$ polarized by $Q$,
and $'G$ is its Mumford-Tate group. The corresponding Mumford-Tate
domain is defined to be the orbit
\[
'D:={}'G(\RR).{}'\vf,
\]
under the action by conjugation. The \emph{compact dual} $'\check{D}$
of $'D$ is the (left-translation) $'G(\CC)$-orbit of the Hodge flag
$'F^{\bullet}$ corresponding to $'\vf$. This is always a complex
projective variety defined over a number field, and contains $'D$
as a real-analytic open subset.

Now, $'G$ is reductive, so we write
\begin{itemize}
\item $A:={}'G/'G^{\text{der}}$ for the maximal abelian quotient,%
\footnote{It follows from \cite[IV.A.2]{GGK} that $A(\RR)$ is compact.%
}
\item $M:={}'G^{\Ad}={}'G/Z_{'G}$ for the ($\QQ$-algebraic) adjoint group,
\end{itemize}
and consider the composition
\[
\vf:\UU\overset{'\vf}{\rightarrow}{}'G\overset{\tiny\begin{array}{c}
\text{finite}\\
\text{quotient}
\end{array}}{\twoheadrightarrow}{}'G/\{Z_{{}'G}\cap{}'G^{\text{der}}\}=A\times M=:G.
\]
Applying the obvious projections, we have $\vf_{A}:\UU\to A$ and
\[
\vf_{M}:\UU\overset{\vf}{\rightarrow}G\twoheadrightarrow M\subset Aut(\lm,B),
\]
which gives a polarized Hodge structure of weight zero on $\lm:=Lie(M)$.
(The polarizing form $B$ is restricted from $Q\otimes Q^{\vee}$
on $End(V)$; it need not be proportional to the Killing form if $M$
is non-simple, but we shall denote it by $B$ anyway.) This allows
us to present $'D$ as a \emph{Hodge domain}
\[
D:=M(\RR).\vf_{M}\cong G(\RR).\vf\cong{}'G(\RR).{}'\vf\,(={}'D),
\]
cf. \cite[sec. IV.F]{GGK}. The reason for the different notations
$D,\,{}'D$ is that we think of the first as a set of (weight zero)
polarized Hodge structures on $\lm$ and the second as polarized Hodge
structures on $V$ (usually of positive weight). This also illustrates
the fact that the same complex manifold may appear in many different
ways as a Mumford-Tate domain $G(\RR)/H$ (with different groups $G(\RR)$
and $H$).

\section*{$\mathfrak{sl}_{2}$-triples}

Representations of the Lie algebra $\sl$, spanned by
\[
n_{+}=\left(\begin{array}{cc}
0 & 1\\
0 & 0
\end{array}\right)\,,\;\; n_{-}=\left(\begin{array}{cc}
0 & 0\\
1 & 0
\end{array}\right)\,,\;\; y=\left(\begin{array}{cc}
1 & 0\\
0 & -1
\end{array}\right),
\]
will play a key role in computing the Mumford-Tate groups. Let $\FF$
be a subfield of $\CC$. Any triple of elements $e_{+},e_{-},x\in\lm_{\FF}$
satisfying 
\[
[x,e_{+}]=2e_{+},\;\;[x,e_{-}]=-2e_{-},\;\;[e_{+},e_{-}]=x
\]
is called an $\FF$-$\sl$-triple.%
\footnote{If $\FF=\CC$, we shall just call this an $\sl$-triple.%
} It is equivalent to a homomorphism
\[
\rho:SL_{2}\to M
\]
with $d\rho:\sl\to\lm$ defined over $\FF$ (and $x=d\rho(y)$, $e_{+}=d\rho(n_{+})$,
and $e_{-}=d\rho(n_{-})$).

Let a nilpotent element $e_{-}\in\lm_{\CC}$ be given. Define 
\[
\lm_{e_{-}}:=\text{im}(\ad e_{-})\cap\ker(\ad e_{-})
\]
and let $M_{e_{-}}$ denote the corresponding algebraic subgroup of
$M_{\CC}$. We have the two famous theorems:

\begin{jm}\cite{J-M,Mo} There exists an $\sl$-triple (in $\lm_{\CC}$)
with $e_{-}=d\rho(n_{-})$, i.e. as ``nil-negative element''. \end{jm}

\begin{kostant}\cite{Ko} \emph{(a)} Given additionally $x\in\lm_{\CC}$
satisfying
\[
\text{(i)}\,[x,e_{-}]=-2e_{-}\;\;\text{and}\;\;\text{(ii)}\, x\in\text{im}(\mathrm{ad}e_{-}),
\]
there is an unique choice of $e_{+}\in\lm_{\CC}$ completing $x,e_{-}$
to an $\sl$-triple.

\emph{(b)} The set of all $\sl$-triples having $e_{-}$ as nil-negative
element is the $M_{e_{-}}$-orbit
\[
\left\{ \left(\mathrm{Ad}g.x,\mathrm{Ad}g.e_{+},e_{-}\right)\left|g\in M_{e_{-}}(\CC)\right.\right\} =
\]
\[
\left\{ \left(e^{\mathrm{ad}\gamma}.x,e^{\mathrm{ad}\gamma}.e_{+},e_{-}\right)\left|\gamma\in\lm_{e_{-},\CC}\right.\right\} .
\]

\emph{(c)} The set of all $x$ occurring in this orbit (or equivalently,
all $x$ satisfying the assumption in (a)) is
\[
x+\lm_{e_{-},\CC}.
\]

\end{kostant}

In the entire statement of Kostant, one can replace $\CC$ by $\RR$
(when $e_{-}\in\lm_{\RR}$). According to the following observation,
we can also refine parts of the above for smaller ground fields:

\begin{schmid}\cite{Sc} If $e_{-}\in\lm_{\FF}$, then there exists
an $\FF$-$\sl$-triple containing it as nil-negative element, say
$(x,e_{+},e_{-})$. The set of all such is in $1$-to-$1$ correspondence
with $x+\lm_{e_{-},\FF}$. Hence, if $x'\in\lm_{\CC}$ satisfies (i)
and (ii) above, then the affine subset $x'+\lm_{e_{-},\CC}$ consisting
of all such is defined over $\FF$, and contains elements in $\lm_{\FF}$
(with corresponding ``$e_{+}$'' in $\lm_{\FF}$).

\end{schmid}

\section*{Concluding remarks}

We wrap up this introduction with a few additional comments on
notation and terminology. For an element $g$ in a Lie group, $\langle g\rangle$
denotes the subgroup it generates; for $\gamma$ an element of a Lie
algebra, the line it generates is denoted by $\langle\gamma\rangle$,
with the field understood from context (or a subscript). For example,
given a Hodge domain $D$ as above and a nilpotent element $N\in\lm_{\QQ}$,
a \emph{nilpotent orbit} is a subset $e^{\langle N\rangle_{\CC}}.F^{\bullet}\subseteq\check{D}$
where $N(F^{\bullet})\subseteq F^{\bullet-1}$ and $e^{\tau N}.F^{\bullet}\in D$
for $\Im(\tau)\gg0$.

A far more important point is our intentional (and consistent
throughout this paper) \emph{notational disregard for issues of disconnectedness}.
For instance, by $M(\RR)$ we shall always mean the identity connected
component in the sense of real Lie groups, usually written $M(\RR)^{+}$;
likewise, $D_{M}$ is always connected (and really an orbit of $M(\RR)^{+}$).
In addition, there are three further kinds of disconnectedness pertaining
specifically to boundary components:%
\footnote{These are stated for the rank-one case, cf. $\S$2 (where $B(N)$
and the $\{I^{p,q}\}$ are defined); they have obvious generalizations
replacing $N$ with $\sigma$.%
}
\begin{itemize}
\item In $M(\RR)^{+}$, $Z(N)$ can fail to be connected. (In fact, the
centralizer of $N$ in $M(\CC)$ can even fail to be connected. On
the other hand, $M_{N}$ -- real or complex -- is always connected:
indeed, since it is unipotent, one can connect each element to the
identity by a one-parameter subgroup.)
\item $B(N)$ can break up into different components with different $I^{p,q}$-dimensions.
These cannot be mapped one to another, even by $Z(N)(\CC)$ (which
respects $W(N)_{\bullet}$ hence must preserve $I^{p,q}$-dimensions).
\item Even within the subset of $B(N)$ comprising LMHS with fixed $I^{p,q}$-dimensions,
we can have different connected components since $Gr_{0}^{W}M_{B(N)}(\RR)$
may not be connected.
\end{itemize}
Instead of adding a ``$+$'' to every real Lie group and a ``connected
component of'' to every boundary component (and M-T domain), we have
simply left these things tacit throughout. There are, beyond readability,
two very good reasons for focusing on connected components: the first
is that it eases passing back and forth between the Lie-algebra/tangent-space
point of view, and the Lie-group/orbit perspective. The second is
that invariants (Mumford-Tate groups, fibration structure, etc.) of
the connected components in ``one'' boundary component are likely
to be different.%
\footnote{In \cite{GGK} (cf. Example VI.B.15), it was discovered that (the
finitely many) connected components of the subset of a period domain
comprising Hodge structures with Mumford-Tate group contained in a
given (Mumford-Tate) subgroup of $Aut(V,Q)$, need not have the same
generic Mumford-Tate group.%
}

Finally, we alert the reader that in this paper, a ``Hodge tensor'' means
a $\mathbb{Q}$-rational class of type $(p,p)$ (for any $p$) in the tensor 
algebra $T(V):=\oplus_{k,\ell}V^{\otimes k}\otimes\check{V}^{\otimes \ell}$
of the HS $V$.  The Mumford-Tate group, for us, is the group fixing the Hodge
tensors. While consistent with \cite{GGK}, this differs from the terminology
in \cite{An} [resp. \cite{De}] (for whom a ``Hodge tensor'' is type $(0,0)$),
whose ``Mumford-Tate group'' is equal [resp. isogenous] to our ``full MTG'' (cf.
Definition \ref{def24}(ii)).

\begin{thx}

The authors are grateful to P. Griffiths and M. Green for stimulating
correspondence, and P. Brosnan for a helpful discussion. This work
was partially supported by NSF Standard Grants DMS-1068974/1361147 (Kerr)
and DMS-1002625/1361120 (Pearlstein). Part of this paper was written while M.K. was a member at the Institute for Advanced Study, and he thanks the IAS for ideal working conditions and the Fund for Mathematics for partial support. Finally, special thanks go to the four referees,
whose careful reading and thorough comments have substantially
improved this paper.

\end{thx}

\section{Rank-one boundary components}

We continue with the setup of $\S$1. Postponing a fuller discussion
of Kato-Usui compactifications to $\S6$ so as to get straight to
a statement of the core result, write
\begin{itemize}
\item $M=$ adjoint Mumford-Tate group,
\item $D_{M}=$ Hodge domain for $M$
\end{itemize}
as above, and also:
\begin{itemize}
\item $N\in\lm_{\QQ}$ a nilpotent element;
\item $W(N)_{\bullet}:=$ the corresponding ``weight'' filtration on $\lm$,
centered about $0$; and
\item $\lm_{N}:=\text{im}(\ad N)\cap\ker(\ad N)=Lie(M_{N})$.
\end{itemize}
The boundary component associated to $N$ is defined by
\[
B_{M}(N):=e^{\langle N\rangle_{\CC}}\backslash \tilde{B}_{M}(N)\,\text{(}=\text{set of nilpotent orbits)}
\]
where
\[
\tilde{B}_{M}(N):=\left\{ F^{\bullet}\in\check{D}_{M}\left|\Ad e^{\tau N}F^{\bullet}\text{ is a nilpotent orbit}\right.\right\} .
\]
For any $F^{\bullet}\in\tilde{B}_{M}(N)$, we can consider the (weight 0) VHS
\[
\Phi_{F^{\bullet}}:\Delta^{*}\to\langle T\rangle\backslash D_{M}
\]
given by
\begin{itemize}
\item the $\QQ$-local system $\underline{\lm}\to\Delta^{*}$ given by the
monodromy operator $T:=e^{\ad N}$ on $\lm_{\QQ}$, and
\item the family of filtrations given with respect to the multivalued basis
of $\underline{\lm}_{\CC}$ by $e^{\frac{1}{2\pi i}\log(q)N}F^{\bullet}$.
\end{itemize}
This has LMHS $(\lm,F^{\bullet},W(N)_{\bullet})$ (henceforth abbreviated
$(F^{\bullet},W(N)_{\bullet})$), which means in particular that \emph{we
can think of $\tilde{B}_{M}(N)$ as a set of LMHS}.  Accordingly, define
\[
\tilde{B}_{M}^{\RR}(N):=\left\{ F^{\bullet}\in\tilde{B}_{M}(N)\left|(F^{\bullet},W(N)_{\bullet})\text{ is an }\RR\text{-split MHS}\right.\right\} ,
\]
and
\[
B_{M}^{\RR}(N):=e^{\langle N\rangle_{\RR}}\backslash \tilde{B}_{M}^{\RR}(N) \subset B_{M}(N).
\]

To understand the relationship between $M_{N}$ and the weight filtration
$W(N)_{\bullet}$, introduce:
\begin{itemize}
\item the centralizer
\[
Z_{M}(N):=\left\{ g\in M\left|\Ad(g)N=N\right.\right\} \supseteq M_{N},
\]
 with Lie algebra $\lz_{M}(N):=\ker(\ad N)$; and
\item the ``weight filtration''
\[
W_{\bullet}^{N}M:=\left\{ g\in M\left|(\Ad(g)-\text{id})W(N)_{i}\lm\subset W(N)_{i+\bullet}\lm\;(\forall i)\right.\right\} 
\]
on $M$ (and its subgroups).
\end{itemize}
We shall prove the following in $\S$\ref{sec3}:
\begin{lem}
\label{lem21}(i) $Lie(W_{k}^{N}M)=W(N)_{k}\lm$,

(ii) $Z_{M}(N)\subset W_{0}^{N}M$, and

(iii) $M_{N}=Z_{M}(N)\cap W_{-1}^{N}M$.\end{lem}

Henceforth, we drop subscript $M$'s except where needed for clarity.

\begin{defn}
\label{def22}Let $(V,F^{\bullet},W_{\bullet})$ be a MHS ($V$ a $\mathbb{Q}$-vector space).

(i) \cite{De4,CKS} The \emph{Deligne bigrading} $I^{p,q}(V_{\CC})$
associated to this MHS is the unique bigrading of $V_{\CC}$ such
that $\oplus_{p+q\leq n}I^{p,q}=W_{n}$, $\oplus_{p\geq p_{0};q}I^{p,q}=F^{p_{0}}$,
and $\overline{I^{q_{0},p_{0}}}\equiv I^{p_{0},q_{0}}$ mod $\oplus_{p<p_{0};q<q_{0}}I^{p,q}$.
(One has $\overline{I^{q,p}}=I^{p,q}$ $(\forall p,q)$ $\iff$ $(V,F,W)$
is $\RR$-split.) It is compatible with tensors and duals in the obvious
way.

(ii) The MHS's \emph{Hodge tensors} are the elements of $\oplus_p I^{p,p}(T(V_{\mathbb{C}}))\cap T(V)$.

\end{defn}

\begin{rem}
\label{rem ?}(a) By the construction in $\S1$ of Hodge
structures on $\lm\subset End(V)$ (from HS on $V$), it is clear that the Lie bracket
$[\, , \,]:\,\mathfrak{m}\times \mathfrak{m}\to \mathfrak{m}$ is a Hodge tensor
of type $(0,0)$ everywhere on $D$.  In view of the compatibility
of taking limit MHS with tensor and dual operations, it is clear that it remains
so on $\tilde{B}(N)$; that is,
\[
(\ad I^{p,q}(\lm_{\CC}))I^{p',q'}(\lm_{\CC})\subset I^{p+p',q+q'}(\lm_{\CC}).
\]
Though \emph{(i)} in Lemma \ref{lem21}
is clear from this, we shall also give a more instructive proof in
$\S3$.

(b) Hodge tensors in a nilpotent orbit%
\footnote{or, more generally, in a polarized VHS, cf. \cite[I.C.1]{GGK}%
} remain Hodge in the limit; in particular, the polarizing form $B\in(\lm^{\vee})^{\otimes2}$
gives perfect pairings
\[
\left\{ \begin{array}{c}
Gr_{k}^{W(N)}\lm\times Gr_{-k}^{W(N)}\lm\to\QQ,\\
I^{p,q}(\lm_{\CC})\times I^{-p,-q}(\lm_{\CC})\to\CC.
\end{array}\right.
\]

(c) Furthermore, it is well-known (cf. \cite{Ca2}) that the formulas
\[
Q_{k}(v,w):=B(v,N^{k}w)
\]
yield polarizations on the primitive subspaces $P_k :=\ker(N^{k+1})\subseteq 
Gr_{k}^{W(N)}\lm$ for $k\geq0$. These naturally extend to polarizations
$B_k$ on $Gr^{W(N)}_k \mathfrak{m} = \oplus_{j\geq 0} N^j (P_{k+2j}).$
(If $k<0$, one defines $B_{k}$ via $B_{-k}$ and the Hard Lefschetz isomorphism
$N^{k}:Gr_{-k}^{W(N)}\lm\overset{\cong}{\to}Gr_{k}^{W(N)}\lm$.)  Conversely,
$(V,F^{\bullet},W_{\bullet})$ is \emph{polarized by} $N\in\mathfrak{m}$ if $W_{\bullet}
=W(N)_{\bullet}$, $N(F^{\bullet})\subseteq F^{\bullet -1}$, and $Q_k$ polarizes $P_k$
($\forall k\geq 0$).
\end{rem}

Next, we recall some material from \cite{An} and \cite{GGK}. 
\begin{defn}\label{def24}
Again let $(V,F^{\bullet},W_{\bullet})$ be a MHS.

(i) \cite[I.C.3]{GGK} Define $\tilde{\vf}:\mathbb{S}(\CC)\,(\cong\CC^{*}\times\CC^{*})\to Aut(V_{\CC})$
by $\CC$-linear extension of the rule ``$\tilde{\vf}(z,w)|_{I^{p,q}}=$
multiplication by $z^{p}w^{q}$''.

(ii) \cite{An} The \emph{Mumford-Tate group} $M_{(F,W)}$ {[}resp.
\emph{full MTG} $\tilde{M}_{(F,W)}${]} \emph{of $(V,F^{\bullet},W_{\bullet})$}
is the largest $\QQ$-algebraic subgroup of $GL(V)$ fixing {[}resp.
scaling{]} all Hodge tensors.
\end{defn}
We remind the reader that for $(F^{\bullet},W_{\bullet})\in\tilde{B}(N)$,
$N$ belongs to $I^{-1,-1}(\lm_{\CC})$ (cf. Remark \ref{rem ?}(a)). To generalize (ii) to a
\emph{set} of MHS, require that the group fix {[}resp. scale{]} Hodge
tensors common to the whole set.
\begin{lem}
\label{lemma2}(a) \cite[I.C.6]{GGK} In the situation of Definition
\ref{def24}, the $\QQ$-closure of the image of $\tilde{\vf}$ is $\tilde{M}_{(F,W)}$.
(There is no corresponding result for $M_{(F,W)}$.)

(b) For a \emph{set} of MHS, the full MTG equals the $\QQ$-closure
of the corresponding set of $\tilde{\vf}(\mathbb{S}(\mathbb{C}))$'s.

(c) For any $F^{\bullet}\in \tilde{B}(N)$, $M_{(F,W(N))}\leq M$.
\end{lem}

\begin{proof}
(b) follows from (a) and Chevalley's theorem.  For (c),
any tensor fixed by $M$ is Hodge for $\Phi_{F^{\bullet}}$, 
hence remains so in the limit.
\end{proof}

As above, one defines a weight filtration $W_{\bullet}$ on $M_{(F,W)}$.
The following properties are standard (cf. \cite{A-K},\cite{An}):
\begin{enumerate}
\item $W_{0}M_{(F,W)}=M_{(F,W)}$;
\item $W_{-1}M_{(F,W)}$ is its unipotent radical;
\item Writing $M_{\text{split}}$ for the (reductive) MTG of the Hodge structure
$\oplus_{i}(Gr_{i}^{W}V,Gr_{i}^{W}F^{\bullet})$, one has a split
short-exact sequence
\[
0\to W_{-1}M_{(F,W)}\to M_{(F,W)}\to M_{\text{split}}\to0.
\]
This also holds with the two right-hand groups replaced by the respective
full Mumford-Tate groups. 
\end{enumerate}
We denote the representation of $\mathbb{U}(\CC)$ {[}resp. $\mathbb{S}(\CC)${]}
corresponding to $\oplus_{i}\left(Gr_{i}^{W}V,Gr_{i}^{W}F^{\bullet}\right)$
by $\varphi_{\text{split}}$ {[}resp. $\tilde{\varphi}_{\text{split}}${]}.

Now for the central object of study:
\begin{defn}
The \emph{Mumford-Tate group $M_{B(N)}$ of the boundary component
$B(N)$} is the largest algebraic subgroup of $GL(\mathfrak{m})$ fixing the Hodge
tensors common to all $(F^{\bullet},W(N)_{\bullet})\in\tilde{B}(N)$.
(That is, it is the MTG of $\tilde{B}(N)$ viewed as a set of MHS.
Similarly one defines $\tilde{M}_{B(N)}$, $M_{B^{\RR}(N)}$, $\tilde{M}_{B^{\RR}(N)}$.)
\end{defn}
By Lemma \ref{lemma2}(c), we have $M_{B(N)}\leq M$. 
Since $M_{B(N)}$ is the MTG of a set of LMHS having $N$ as a Hodge
tensor, it is contained in $Z(N)$; likewise, $\tilde{M}_{B(N)}$
is contained in 
\[
\tilde{Z}(N):=\left\{ g\in M\left|\Ad(g)\langle N\rangle\subseteq\langle N\rangle\right.\right\} .
\]
So for any (limit) MHS $\tilde{\vf}\in\tilde{B}(N)$ we have a diagram
\begin{equation} \label{eqn*1}\xymatrix{
\mathbb{S}  \ar [r]^{\tilde{\vf}\mspace{30mu}} \ar @{=} [d] & \tilde{Z}(N) \ar @{^(->} [r]^{\Ad\mspace{50mu}} \ar @{->>} [d]  & Aut(\lm,B,W(N)_{\bullet}) \ar @{->>} [d]
\\
\mathbb{S} \ar [r]^{\tilde{\vf}_{\text{split}}\mspace{50mu}} & \tilde{Z}(N)/M_N \ar @{^(->} [r]^{\overline{\Ad}\mspace{50mu}} & \times_i GL(Gr^{W(N)}_i\lm )
\\
\UU \ar @{^(->} [u] \ar [r]^{\vf_{\text{split}}\mspace{50mu}} & Z(N)/M_N \ar @{^(->} [u] \ar @{^(->} [r]^{\overline{\Ad}\mspace{60mu}} & \times_i Aut(Gr^{W(N)}_i\lm,B_i). \ar @{^(->} [u]
}\end{equation} It will be convenient to write $Gr_{i}^{W(N)}\vf$
for the composition of $\overline{\Ad}\circ\vf_{\text{split}}$ with
the projection to $Aut(Gr_{i}^{W(N)}\lm,B_{i})$.%
\footnote{In \cite{GGK}, $\varphi$ denotes the restriction of $\tilde{\varphi}$
to $\mathbb{U}(\mathbb{C})\cong\mathbb{C}^{*}$ (i.e. $\vf(z):=\vft(z,z^{-1})$).
Since the $\QQ$-closure of its image is not $M_{(F,W)}$ in the mixed
case, we will use only ``$Gr_{i}^{W(N)}\vf$'', and not ``$\vf$''
as such, in this paper.%
}
\begin{rem}
\label{rem ??}The three groups

(a) $Aut(Gr_{0}^{W(N)}\lm,B_{0})\times Aut(Gr_{-1}^{W(N)}\mathfrak{m},B_{-1})$,

(b) $\times_{k\geq0}Aut(P_{k},Q_{k})$, and

(c) $\times_{k\geq0}Aut(Gr_{-k}^{W(N)}\lz(N),B_{-k}|_{\cdots})$\\
are embedded ``diagonally'' in $\times_{i}Aut(Gr_{i}^{W(N)}\lm,B_{i})$,
using the maps
\[
\left\{ \begin{array}{c}
Gr_{j}^{W(N)}\lm\overset{N^{j/2}}{\hookrightarrow}Gr_{0}^{W(N)}\lm\overset{N^{j/2}}{\twoheadrightarrow}Gr_{-j}^{W(N)}\lm\;\;(j\text{ even})\\
\\
Gr_{j}^{W(N)}\lm\overset{N^{(j+1)/2}}{\hookrightarrow}Gr_{-1}^{W(N)}\lm\overset{N^{(j-1)/2}}{\twoheadrightarrow}Gr_{-j}^{W(N)}\lm\;\;(j\text{ odd})
\end{array}\right.
\]
and the isomorphisms
\[
\left\{ \begin{array}{c}
Gr_{j}^{W(N)}\lm\cong\oplus_{l\geq\max\{0,-j\}}N^{\ell}P_{j+2\ell},\\
\\
P_{k}\overset{\cong}{\underset{N^{k}}{\longrightarrow}}Gr_{-k}^{W(N)}\lz(N).
\end{array}\right.
\]
Since $Z(N)$ commutes with $N$, one easily verifies that the bottom
$\overline{\Ad}$ in \eqref{eqn*1} factors through each of (a)-(c).
We conclude that the groups
\[
Gr_{0}^{W}Z(N)=Z(N)/M_{N}\,,\;\;\;\; Gr_{0}^{W}\tilde{Z}(N)=\tilde{Z}(N)/M_{N}
\]
act faithfully on the vector spaces

(a') $Gr_{0}^{W(N)}\lm\oplus Gr_{-1}^{W(N)}\lm$,

(b') $\oplus_{k\geq0}P_{k}$, and (perhaps most naturally)

(c') $\oplus_{k\geq0}Gr_{-k}^{W(N)}\lz(N)$.

This reveals that the Hodge structure \emph{$\vf_{\text{split}}$
is completely determined by its restriction to these spaces}.
\end{rem}
An important related point, which establishes $Gr_{0}^{W}Z(N)$ as
a possible Mumford-Tate group of $\vf_{\text{split}}$, is the
\begin{prop}
(i) $Gr_{0}^{W}Z(N)$ and $Gr_{0}^{W}\tilde{Z}(N)$ are reductive;
and\\
(ii) $M_{N}$ is the unipotent radical of $Z(N)$ and $\tilde{Z}(N)$.\end{prop}
\begin{proof}
Given $\vft\in\tilde{B}(N)$, set $C:=\vft(i,-i)\in Z(N)(\CC)$. We
may consider $\Psi_{C}$ (conjugation by $C$) as an automorphism
of $Z(N)/M_{N}$. Since $Z(N)/M_{N}$ acts faithfully in blocks on
$\oplus_{\ell\geq0}P_{\ell}$ (or $\oplus_{\ell\geq0}Gr_{-\ell}^{W}\lz(N)$),
and $C^{2}=\vft(-1,-1)$ acts in scalar blocks on $\oplus_{\ell}P_{\ell}$,
they commute. Hence $\Psi_{C}$ is an involution. Considered as Hodge
structures, the $P_{\ell}$ are polarized by $Q_{\ell}$; in particular,
the Hermitian forms defined by $Q_{\ell}(Cv,\bar{w})$ ($v,w\in P_{\ell}$)
are positive-definite. Since the real form
\[
\mathcal{G}:=\left\{ g\in\left(Z(N)/M_{N}\right)(\CC)\left|\Psi_{C}(\bar{g})=g\right.\right\} 
\]
of $Z(N)/M_{N}$ evidently preserves these forms, it is compact. We
conclude (by the criterion in the proof of Prop. 3.6 in \cite{De})
that $Gr_{0}^{W}Z(N)$ is reductive, and (ii) follows from this at
once.
\end{proof}
We are now ready to state the main result in the rank-one case:
\begin{thm}
\label{thm1}Assume $B(N)\neq\emptyset$. Then we have $M_{N}\subseteq M_{B(N)}\subseteq Z(N)$.
More precisely,\\
(A) $W_{-1}M_{B(N)}=M_{N}$; and\\
(B) $Gr_{0}^{W}M_{B(N)}\subseteq Gr_{0}^{W}Z(N)$ is the Mumford-Tate
group of the set $Gr_{0}^{W}Z(N)(\RR).\vf_{\text{split}}$ of Hodge
structures on (b') {[}resp. (a'),(c'){]} above, or equivalently the
$\QQ$-algebraic group closure of a generic $\vf_{\text{split}}$.
\end{thm}
A similar statement holds for $\tilde{M}_{B(N)}$.

\section{Some lemmas}\label{sec3}
\begin{defn}
Given a MHS $(V,F^{\bullet},W_{\bullet})$, we define its \emph{Deligne
splitting} (a splitting of $W_{\bullet}$) to be the element $Y\in End(V_{\CC})$
whose $k$-eigenspace is $\oplus_{p+q=k}I^{p,q}(V_{\CC})$ for each $k$.
Note that $Y\in I^{0,0}$.  The MHS is \emph{$\mathbb{Q}$-split} if it is
a direct sum of pure HS, or equivalently if $Y\in End(V)$.
\end{defn}
We reiterate that the mixed Hodge structures parametrized by boundary
components $\tilde{B}(N)$ are limits of \emph{weight zero} $N$-nilpotent
orbits with $V=\mathfrak{m}$, $W_{\bullet}=W(N)_{\bullet}$. Hence
the eigenvalues of $Y$ are automatically centered about $0$.

We shall work for now under the assumption that $\tilde{B}^{\RR}(N)\neq\emptyset$.
For simplicity, the notation $\Ad$ for $M$ acting on $\lm$ will
be dropped (except for occasional emphasis); so nilpotent orbits are
now just $e^{\tau N}.F^{\bullet}$.
\begin{lem}
\label{lemma3}$\tilde{B}^{\RR}(N)$ contains a $\QQ$-split MHS $(F^{\bullet},W(N)_{\bullet})$.\end{lem}
\begin{proof}
By assumption, we have a nilpotent orbit $F_{\tau}^{\bullet}\mathfrak{m}_{\mathbb{C}}:=e^{\tau N}.F_{0}^{\bullet}\lm_{\CC}$
with $\RR$-split LMHS $(F_{0}^{\bullet},W(N)_{\bullet})$. The Deligne
splitting of the latter kills its Hodge $(0,0)$ tensors, in particular
any which survive to Hodge tensors in the nilpotent orbit. Hence this
splitting is of the form $\ad Y_{0}$ for some $Y_{0}\in\lm_{\RR}$. Clearly $[Y_0,N]=-2N$.

Furthermore, since the LMHS is $\RR$-split we have $F_{\tau}^{\bullet}\in D$
for $\Im(\tau)>0$, cf. \cite[Lemma 3.12]{CKS}. (Though the result
as stated in {[}op. cit.{]} says \emph{only} that $F_{\tau}^{\bullet}$
is in the period domain containing $D$, our definition of $\tilde{B}^{\RR}(N)$
ensures that $F_{\tau}^{\bullet}\in D$ for $\Im(\tau)\gg0$, which
combined with {[}op. cit.{]} is clearly sufficient.) In particular,
$F_{i}^{\bullet}\in D$. Decomposing $N=N^{(1,-1)}+N^{(0,0)}+N^{(-1,1)}$
with respect to the pure HS on $\mathfrak{m}$ induced by $F_{i}^{\bullet}$,
the Hodge components $N^{(-j,j)}$ therefore belong to $\mathfrak{m}_{\mathbb{C}}$.

By \cite[Thm. 7.5.13]{Ca3}, there exists a Lie algebra homomorphism $\rho:\, \mathfrak{sl}_{2,\mathbb{R}}
\to End(\mathfrak{m}_{\mathbb{R}},B)$ with $\rho(n_-)=N$ and $\rho(y)=Y_0$, which is \emph{Hodge at $t=i$}.
That is, putting $\mathfrak{F}^1 \mathfrak{sl}_{2,\mathbb{C}} := \langle -iy+n_- + n_+ \rangle$, $\mathfrak{F}^0 \mathfrak{sl}_{2,\mathbb{C}} :=\langle -iy+n_- + n_+ , n_+ - n_-\rangle$, and $\mathfrak{F}^{-1}\mathfrak{sl}_{2,\mathbb{C}}=\mathfrak{sl}_{2,\mathbb{C}}$, we have $\rho_{\mathbb{C}}(\mathfrak{F}^{\bullet})\subseteq F^{\bullet}_i$.  In particular, $N^{(0,0)}=\rho_{\mathbb{C}}(n_-^{(0,0)})=\rho_{\mathbb{C}}(\frac{1}{2}(n_- - n_+))$, and so $Y_0 = \rho (y)= \rho(\text{ad}(n_-)(n_- - n_+))=\text{ad}(N)(2N^{(0,0)})$ belongs to the image of $\text{ad}(N)$ \emph{acting on $\mathfrak{m}_{\mathbb{R}}$}.

Taking $e_- = N\,(\in\mathfrak{m}_{\mathbb{Q}})$, $x'=Y\,(\in\mathfrak{m}_{\mathbb{R}})$ in Schmid's result from the end of $\S 1$, \emph{(i)} and \emph{(ii)} there are satisfied, and there exists a $Y\in (Y_0 + \mathfrak{m}_{N,\mathbb{C}})\cap \mathfrak{m}_{\mathbb{Q}} \subseteq Y_0 + \mathfrak{m}_{N,\mathbb{R}}$.  Kostant's result (c) now provides $g\in M_N(\mathbb{R})$ such that $Y=g.Y_0$; and then
\[
g.e^{\tau N}.F_{0}^{\bullet}=e^{\tau N}.g.F_{0}^{\bullet}=:e^{\tau N}.F^{\bullet}
\]
 is a nilpotent orbit with $\ad Y$ the Deligne splitting of its limiting MHS
$(F^{\bullet},W(N)_{\bullet})$.
\end{proof}

Now fix (by Lemma \ref{lemma3} and $\S1$)
\begin{itemize}
\item $\left(F^{\bullet},W(N)_{\bullet}\right)\in\tilde{B}^{\RR}(N)$ $\QQ$-split
(L)MHS,
\item $Y\in\lm_{\QQ}$ associated Deligne grading, and
\item $N_{+}\in\lm_{\QQ}$ such that $(Y,N_{+},N)$ is a $\QQ$-$\sl$-triple.
\end{itemize}
Consider the decomposition
\[
\lm=\oplus_{\ell\in\ZZ_{\geq0}}\lm(\ell)=\oplus_{\ell\in\ZZ_{\geq0}}V(\ell)^{\oplus m_{\ell}}
\]
into isotypical components, where $V(\ell)$ is a copy of the $(\ell+1)$-dimensional
irrep of $\sl$. More precisely, writing $V(\ell)=\text{span}\left\{ v_{-\ell}^{(\ell)},v_{-\ell+2}^{(\ell)},\ldots,v_{\ell}^{(\ell)}\right\} $
we have on each copy
\begin{itemize}
\item $(\ad Y)v_{k}^{(\ell)}=kv_{k}^{(\ell)},$
\item $(\ad N_{+})v_{\ell - 2k}^{(\ell)}=\left\{ \begin{array}{cc}
0, & k=0 \\
(\ell - k+1) v_{\ell-2k+2}^{(\ell)}, & \text{otherwise}
\end{array}\right.,$
\item $(\ad N)v_{\ell-2k}^{(\ell)}=\left\{ \begin{array}{cc}
0, & k=\ell\\
(k+1)v_{\ell-2k-2}^{(\ell)}, & \text{otherwise}
\end{array}\right..$
\end{itemize}
Note that the $\sl$-triple itself is a $V(2)\subset\lm$.

Setting 
\[
E(k):=\left\{ v\in\lm\left|(\ad Y)v=kv\right.\right\} \subset\lm
\]
the definition of $W(N)_{\bullet}$ is
\[
W(N)_{m}\lm:=\oplus_{k\leq m}E(k).
\]
(This is independent of the choice of $Y$, $N_{+}$ though the individual
$E(k)$ are not.)

\begin{proof}[Proof of Lemma 2.1]

(i) Let $v\in E(k)$ be given. For $w\in E(j)$, the Jacobi identity
gives
\[
[Y,[v,w]]=[v,[Y,w]]+[[Y,v],w]=(j+k)[v,w]
\]
\[
\implies(\ad v)w\in E(j+k).
\]
Moreover, there is \emph{some} $j\in\ZZ$ and $w\in E(j)$ such that
$(\ad v)w\neq0$: namely, $j=0$ and $w=Y$. Since
\[
\text{Lie}\left(W_{k}^{N}M\right)=\left\{ \gamma\in\lm\left|(\ad\gamma)W(N)_{i}\subset W(N)_{i+k}\;(\forall i)\right.\right\} ,
\]
the statement follows.

(ii) From the above characterization of isotypical components, any
$\gamma\in\ker(\ad N)$ is a sum of $v_{-\ell}^{(\ell)}$'s. That
is,
\[
\lz(N)=\oplus_{\ell\geq0}\lm(\ell)\cap E(-\ell)\subseteq W(N)_{0}\lm.
\]

(iii) We have 
\[
\begin{array}{ccc}
\lz(N)\cap\text{im}(\ad N) & = & \oplus_{\ell\geq0}\left((\ad N)\lm(\ell)\right)\cap E(-\ell)\\
\\
 & = & \oplus_{\ell\geq1}\lm(\ell)\cap E(-\ell)\;\;\;\;\;\;\;\;\\
\\
 & = & \lz(N)\cap W(N)_{-1}\lm.\;\;\;\;\;\;\;\;
\end{array}
\]

\end{proof} Emphasizing that we are working with a $\QQ$-split MHS,
we turn to
\begin{lem}
The $\lm(\ell)\subset\lm$ (a) are sub-MHS, and (b) admit a further
decomposition
\[
\lm(\ell)_{\CC}=\oplus_{\alpha\in\ZZ}\lm(\ell,\alpha)
\]
where 
\[
\lm(\ell,\alpha):=\lm(\ell)\cap\left\{ \oplus_{p-q=\alpha}I^{p,q}(\lm_{\CC})\right\} 
\]
are themselves sums of copies of $V(\ell)_{\CC}$.\end{lem}
\begin{proof}
For (a), simply write
\[
\lm(\ell)=\bigoplus_{k=0}^{\ell}\left\{ \begin{array}{c}
\ker\{(\ad N)^{k+1}\}\cap\text{im}\{(\ad N_{+})^{k}\}\\
\\
\cap\ker\{(\ad N_{+})^{\ell-k+1}\}\cap\text{im}\{(\ad N)^{\ell-k}\}
\end{array}\right\} 
\]
and note that $\ad N$, $\ad N_{+}$ are morphisms of MHS of respective
types $(-1,-1)$, $(+1,+1)$. (Of course, for $N_{+}$ this is \emph{only}
by virtue of $(F^{\bullet},W(N)_{\bullet})$ being $\QQ$-split.)
To produce (b), one takes the $(\ad N)$-orbit of the Hodge decomposition
of the pure Hodge structure $\lm(\ell)\cap E(\ell)=\lm(\ell)\cap\ker(\ad N_{+}).$
\end{proof}
Consequently we can read off the full MHS on $\lm$ from pure Hodge
structures $\tilde{P}_{\ell}$ of weight $\ell$
\[
\tilde{P}_{\ell}^{\frac{\ell+\alpha}{2},\frac{\ell-\alpha}{2}}:=\lm(\ell,\alpha)\cap E(\ell).
\]
(Note that $E(\ell)\subset W(N)_{\ell}\lm\twoheadrightarrow Gr_{\ell}^{W(N)}\lm$
sends $\tilde{P}_{\ell}\overset{\cong}{\to}P_{\ell}$.) Writing
\[
V(\ell):=\bigoplus_{k=0}^{\ell}\QQ(k)
\]
and endowing $\tilde{P}_{\ell}$ with the trivial $\sl$-representation,
\begin{equation}\label{eqn sl2-repn}\lm \cong \bigoplus_{\ell\geq 0} \tilde{P}_{\ell} \otimes V(\ell)\end{equation}
as MHS and $\sl$-representations.

We now turn to the last result of this section, which will be a key to the proof of Theorem \ref{thm1}:

\begin{lem}
\label{lemma4}$Z(N)(\RR)$ acts transitively on $\tilde{B}^{\RR}(N)$.\end{lem}
\begin{proof}
Clearly $Z(N)(\RR)$ \emph{acts} on $\tilde{B}^{\RR}(N)$, since for
$g\in Z(N)(\RR)$
\[
g.e^{\tau N}.F^{\bullet}=e^{\tau N}.g.F^{\bullet}
\]
 and $ge^{\tau N}F^{\bullet}$ is still in $D$ for $\Im(\tau)\gg0$
(as $g$ is real).

Next we determine the tangent space to $\tilde{B}(N)$ at some point
$x=(F^{\bullet},W(N)_{\bullet})$, in terms of whose Deligne bigrading
we consider the subspaces ($\lm^{p,q}:=I^{p,q}(\lm_{\CC})$)
\[
\lq:=\bigoplus_{\Tiny\begin{array}{c}
p,q\\
p<0
\end{array}}\lm^{p,q}\,,\;\;\;\lp:=\bigoplus_{\Tiny\begin{array}{c}
p,q\\
p\geq-1
\end{array}}\lm^{p,q}
\]
of $\lm_{\CC}$. There is a natural identification of $T_{x}\check{D}$
with $\lq$. Let $\{\alpha(t)\}_{t\in(-\epsilon,\epsilon)}$ be a
curve in $\lq$ with $e^{\alpha(t)}.F^{\bullet}\in\tilde{B}(N)$ (and
$\alpha(0)=0)$. Then $e^{\tau N}.e^{\alpha(t)}.F^{\bullet}$ must
be ($\forall t$) a nilpotent orbit, hence
\[
Ne^{\alpha(t)}F^{p}\;\subset\; e^{\alpha(t)}F^{p-1}
\]
\[
\implies\;\begin{array}[t]{c}
\underbrace{e^{-\alpha(t)}Ne^{\alpha(t)}}\\
e^{-\ad\alpha(t)}N
\end{array}F^{p}\;\subset\; F^{p-1}.
\]
Differentiating and setting $t=0$ gives
\[
[\alpha'(0),N]F^{p}\subset F^{p-1}
\]
\[
\implies[\alpha'(0),N]\in\lp\cap(\ad N)\lq.
\]
But from the picture $$\includegraphics[scale=0.7]{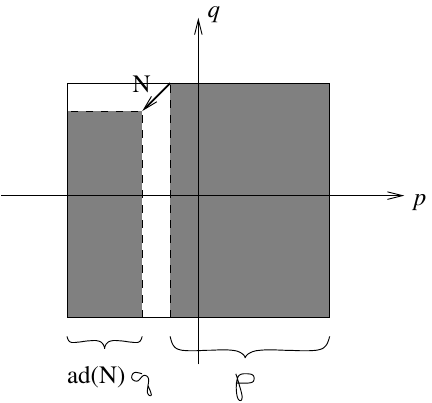}$$ we have
$\lp\cap(\ad N)\lq=\{0\}$
\[
\implies\alpha'(0)\in\ker(\ad N)=\lz(N)_{\CC},
\]
so that $\lz(N)_{\CC}\cap\lq\cong\lz(N)_{\CC}/F^{0}(\lz(N))$ is identified
with $T_{x}\tilde{B}(N)$.

Now suppose that $\lm^{p,q}=\overline{\lm^{q,p}}$ ($\forall p,q$),
i.e. $x\in\tilde{B}^{\RR}(N)$, and let $\beta(t)$ be a curve in
$\lz(N)_{\CC}$ such that $(e^{\beta(t)}F^{\bullet},W(N)_{\bullet})$
stays in $\tilde{B}^{\RR}(N)$ (and $\beta(0)=0)$. Then $\overline{e^{\beta(t)}F^{q}}\cap e^{\beta(t)}F^{p}\cap W_{p+q}$,
or (since $\mathfrak{z}(N)_{\mathbb{C}} \subseteq W_0 \mathfrak{m}_{\mathbb{C}}$) equivalently $\overline{F^{q}}\cap e^{-\overline{\beta(t)}}e^{\beta(t)}F^{p}\cap W_{p+q}$,
must remain of the same dimension as at $t=0$ (for all $p,q$). By
a short calculation, this implies that $\Im(\beta'(0))\in(F^{0}+\overline{F^{0}})\cap\lz(N)_{\RR}$,
so that
\[
\beta'(0)\in\left((F^{0}+\overline{F^{0}})\cap\lz(N)\right)_{\CC}\oplus\left((\lq\cap\overline{\lq})\cap\lz(N)\right)_{\RR}.
\]
Hence under the projection $\lm_{\CC}\underset{\rho}{\twoheadrightarrow}\lm_{\CC}/F^{0}\cong\lq$,
\[
\rho\left(\beta'(0)\right)\in\rho\left(\lz(N)_{\RR}\right)\cong\lz(N)_{\RR}/\{\lm^{0,0}\cap\lz(N)_{\RR}\}.
\]

It follows that $Z(N)(\RR)$ acts \emph{locally} transitively on $\tilde{B}^{\RR}(N)$,
i.e. that the orbit of a point yields a real-analytic open subset;
and a connected set cannot be a (disjoint) union of such orbits.
\end{proof}

\section{proof of Theorem \ref{thm1}}

Express the $\QQ$-split mixed Hodge structure $(F^{\bullet},W(N)_{\bullet})$ and $\sl$-representation
$(Y,N_{+},N)$ of $\S3$ by $(\vft,d\rho)$. The idea of the proof
is to let $Z(N)(\RR)$ act on this pair to yield (by Lemma \ref{lemma4})
all of $\tilde{B}^{\RR}(N)$. Then take (as per Lemma \ref{lemma2}(b))
the $\QQ$-closure of the union of the images of all $\vft'\in Z(N)(\RR).\vft$,
to yield $\tilde{M}_{B^{\RR}(N)}$. (At the end of the proof we shall
check that this coincides with $\tilde{M}_{B(N)}$.) In fact, all
we really need to determine is $\tilde{\lm}_{B^{\RR}(N)}$, which
we can do by acting infinitesimally (to arbitrary order) on the tangent
space to $\vft$ -- that is, by repeated application of $\ad\lz(N)_{\RR}$.
\begin{proof}
We start with an explicit formula for $\vft$, using Definition \ref{def24}.
The splitting of $F^{\bullet}$ induced by the $\left\{ \oplus_{q}I^{p,q}(\lm_{\CC})\right\} _{p\in\ZZ}$
kills Hodge $(0,0)$-tensors, hence any Hodge tensors in the (weight 0) nilpotent
orbit, and scales $N\,\left(\in I^{-1,-1}(\lm_{\CC})\right)$. Thus
it may be written $\ad\xi$ for some $\xi\in\tilde{\lz}(N)_{\CC}=(\lz(N)+\langle Y\rangle)_{\CC}\subseteq\lm_{\CC}$.
Since $\vft$ is $\RR$-split, $\ad\bar{\xi}$ not only splits $\overline{F^{\bullet}}$
but has eigenspaces $\left\{ \oplus_{p}I^{p,q}(\lm_{\CC})\right\} _{q\in\ZZ}$.
In fact, from $N\in I^{-1,-1}(\lm_{\CC})$ it is clear that $[N,\xi]=N=[N,\bar{\xi}]$.
Further, $\xi$ and $\bar{\xi}$ commute with each other and with
$Y$ since the three are simultaneously diagonalizable. By Definition
\ref{def24}(i), we have immediately that 
\[
\vft(z,w)=\exp(\log(z)\ad\xi+\log(w)\ad\bar{\xi})
\]

Now $\vft(z^{2},w^{2})/\exp(\log(zw)\ad Y)$ is clearly of pure weight
zero, so is of the form $\exp(\log(z/w)\text{ad}(i\phi))$ where we
have 
\[
Y=\xi+\bar{\xi}\;,\;\; i\phi=\xi-\bar{\xi}.
\]
In particular, $\phi$ commutes with $N$ and $Y$ and so 
\[
\phi\in(\lz(N)\cap E(0))_{\RR}=\lm(0)_{\RR}.
\]
Recalling \eqref{eqn sl2-repn}, we see that $\phi$ acts on $\lm$
through the $\{\tilde{P}_{\ell}\}$ while $Y$ acts through the $\{V(\ell)\}.$
The formula becomes \begin{equation}\label{eqn becomes}\vft (z,w)=\exp \left( \frac{1}{2} \log(zw)\ad Y + \frac{i}{2} \log(z/w)\ad \phi \right) ,\end{equation}
and $\phi$ determines $\vf_{\text{split}}$ through
\[
\left(Gr_{k}^{W(N)}\vf\right)(z)=\left.\vft(z,z^{-1})\right|_{E(k)}=\exp\left(\left.i\log(z)(\ad\phi)\right|_{E(k)}\right).
\]
For later use, we note that $\lm(0)$ is defined over $\QQ$ and (in view of Remark \ref{rem ?}(a)) closed
under the Lie bracket.

We wish to take the $\QQ$-algebraic group closure of the set of $Z(N)(\RR)$-conjugates
of 
\[
\vft(\CC^{*}\times\CC^{*})=\left\{ \exp(\alpha Y+\beta\phi)\;|\;\alpha,\beta\in\CC\right\} 
\]
in $\tilde{Z}(N)(\CC)$. As mentioned before, to compute $\tilde{\lm}_{B^{\RR}(N)}$
it suffices to take the $\QQ$-Lie-algebra closure of $Z(N)(\RR)$-conjugates
of the tangent space \begin{equation}\label{eqn space}T_{\mathbf{e}}\vft (\CC^*\times\CC^*)=\left\{\alpha Y+\beta \phi \; | \; \alpha,\beta\in\CC \right\}\end{equation}
in $\tilde{\lz}(N)_{\CC}$. We first consider the effect of an infinitesimal
$M_{N}(\RR)$-action on \eqref{eqn space}.

Let $\{\gamma_{i}\}\subset\lm_{N}$ be a basis consisting of $Y$-eigenspaces;
that is, $\gamma_{i}\in E(k_{i})\cap\lm_{N}$ with $k_{i}<0$. Then
\[
(\ad\gamma_{i})Y=-(\ad Y)\gamma_{i}=-k_{i}\gamma_{i},
\]
while $\phi\in(\lz(N)\cap E(0))_{\RR}$, $\gamma_{i}\in\lz(N)\cap E(k_{i})$
gives 
\[
(\ad\gamma_{i})\phi\in\left(\lz(N)\cap E(k_{i}+0)\right)_{\RR}\subset(\lm_{N})_{\RR}.
\]
To first order in $\epsilon$, the tangent plane to $e^{\epsilon\sum\mu_{i}\gamma_{i}}.\vft(z,w)$
($\mu_{i}\in\RR$; action by conjugation) is then given by 
\[
e^{\epsilon\sum\mu_{i}(\ad\gamma_{i})}(\alpha Y+\beta\phi)\;\;=
\]
\[
\alpha Y+\beta\phi-\epsilon\alpha\sum_{i}\mu_{i}k_{i}\gamma_{i}+\epsilon\beta\sum_{i}\mu_{i}(\ad\gamma_{i})\phi
\]
(where $\alpha,\beta\in\CC$). By taking $\beta=0$ \emph{we already
get all of $(\lm_{N}+\langle Y\rangle)_{\CC}$} from the effect on
$Y$ alone. (In fact, this essentially reproves part (c) of Kostant.)

To go from $M_{N}(\RR)$-action to $Z(N)(\RR)$-action, let $\gamma\in(\lm(0))_{\RR}$;
then
\[
\left\{ \begin{array}{c}
(\ad\gamma)Y=-(\ad Y)\gamma=0\\
(\ad\gamma)\phi\in(\lm(0))_{\RR}\mspace{45mu},
\end{array}\right.
\]
 and so
\[
e^{\epsilon(\ad\gamma+\sum_{i}\mu_{i}(\ad\gamma_{i}))}(\alpha Y+\beta\phi)\mspace{250mu}
\]
\[
\mspace{80mu}\equiv\;\beta\left(\phi+\epsilon(\ad\gamma)\phi+\frac{\epsilon^{2}}{2}(\ad\gamma)^{2}\phi+\cdots\right)\;\;\text{mod}\;(\lm_{N}+\langle Y\rangle)_{\CC}.
\]
Denote by $\lg_{B^{\RR}(N)}$ the $\QQ$-Lie-algebra closure in $\lm(0)$
($\cong\lz(N)/\lm_{N}$) of\footnote{Note that the vector space sum in \eqref{eqn !!!} stabilizes after finitely many terms, hence is well-defined (and easy to compute in this form).} \begin{equation}\label{eqn !!!}\langle \phi \rangle_{\CC} + \sum_{\gamma\in \lm(0)_{\RR}}(ad \gamma)\langle \phi\rangle_{\CC} + \sum_{\gamma\in\lm(0)_{\RR}}(\ad \gamma)^2\langle \phi \rangle_{\CC} + \cdots .\end{equation}
Then $\langle Y\rangle$ and 
\[
\tilde{\lg}_{B^{\RR}(N)}:=\lg_{B^{\RR}(N)}+\langle Y\rangle
\]
are $\QQ$-Lie algebras in $\tilde{\lz}(N)$ mapping isomorphically
to their images in $\tilde{\lz}(N)/\lm_{N}$, and we have proved the
short-exact sequence
\[
0\to\lm_{N}\to\tilde{\lm}_{B^{\RR}(N)}\to\tilde{\lg}_{B^{\RR}(N)}\to0.
\]
Note that we have crucially used the facts that $\lm_{N}+\langle Y\rangle$
is a $\QQ$-Lie algebra and that $\lm(0)$ and $Y$ commute, so that
the only $\QQ$-closure which needs to be taken is that of \eqref{eqn !!!}.
Intersecting with $\lz(N)$ now gives
\[
0\to\lm_{N}\to\lm_{B^{\RR}(N)}\to\lg_{B^{\RR}(N)}\to0.
\]

To finish off the proof we must show that none of these groups ``increases''
upon passing from $B^{\RR}(N)$ to $B(N)$. Given $\vft'\in\tilde{B}(N)$,
there exists a natural element%
\footnote{This is by a result of Deligne, cf. \cite[sec. 4]{KP}. What is not
addressed there is the fact that (in the present setting) $\delta$
belongs to the Mumford-Tate Lie algebra $\lm$. This follows at once
from its functoriality with respect to Tannakian operations and property
(3) below. %
}
\[
\delta\in\left(\bigoplus_{p,q\leq-1}I^{p,q}(\lm)\right)_{\RR}
\]
such that:

(1) $e^{-i\delta}.\vft'$ is $\RR$-split;

(2) $\delta\in W_{-2}\lm$; and

(3) $\delta$ commutes with all $(r,r)$-morphisms of MHS.\\
Clearly (2) gives $e^{-i\delta}\in(W_{-2}M)(\CC)\subseteq(W_{-1}M)(\CC)$,
while (3) $\implies$ $[\delta,N]=0$ 
\[
\implies\;\delta\in\lz(N)_{\RR}\;\implies\; e^{-i\delta}\in Z(N)(\CC).
\]
By Lemma \ref{lem21}, we now have $\delta\in M_{N}(\CC)$. Now \emph{a
priori}, $e^{-i\delta}.\vft'$ is just a MHS polarized by $N$ (cf. Remark \ref{rem ?}(c)). 
By \cite[Cor 3.13]{CKS}, this suffices to make
$e^{\tau N}e^{-i\delta}F_{\vft'}^{\bullet}$ a nilpotent orbit; hence
$e^{-i\delta}.\vft'\in B^{\RR}(N)$. This shows that 
\[
\tilde{B}(N)\neq\emptyset\;\implies\;\tilde{B}^{\RR}(N)\neq\emptyset.
\]

Provided we act on the Hodge filtration $F^{\bullet}$ (instead of
$\vft$), [loc. cit.] also implies that
$\tilde{B}(N)$ is closed under the action of $M_{N}(\CC)$. So by
the last paragraph, the $M_{N}(\CC)$-orbit of $\tilde{B}^{\RR}(N)$
is all of $\tilde{B}(N)$. Together with the proof of Lemma \ref{lemma4},
this means that on an infinitesimal level the $\ad(\lq\cap\bar{\lq}\cap\lm_{N,\CC})_{i\RR}$-orbit
of our given $\QQ$-split MHS $p=(F^{\bullet},W(N)_{\bullet})$ surjects
onto $T_{p}\tilde{B}(N)/T_{p}\tilde{B}^{\RR}(N)$. This is consistent
with the action of $\ad(\lq\cap\bar{\lq}\cap\lm_{N,\CC})$ on $\vft$.
The upshot is that we only need to examine the effect on our original
$\alpha Y+\beta\phi$ by a larger subset of $\ad(\lm_{N,\CC})$ replacing
$\ad(\lm_{N,\RR})$ in the above argument, which obviously does not
change the $\QQ$-closure. The proof of Theorem \ref{thm1} is now
complete.
\end{proof}
Two new results come out of this proof. For the first, we need
\begin{defn}
\label{def a}$\lg_{N}:=\lz(N)\cap E(0)$, $\tilde{\lg}_{N}:=\tilde{\lz}(N)\cap E(0)\,(=\lg_{N}+\langle Y\rangle)$,
$\lg_{B(N)}:=\lg_{B^{\RR}(N)}$; with corresponding Lie groups $G_{N},\, G_{B(N)}\le M_{B(N)}$
and $\tilde{G}_{N},\,\tilde{G}_{B(N)}\leq\tilde{M}_{B(N)}$.\end{defn}
\begin{prop}
\label{prop a}We have semi-direct product decompositions
\[
\tilde{M}_{B(N)}=M_{N}\rtimes\tilde{G}_{B(N)}\subseteq M_{N}\rtimes\tilde{G}_{N} = \tilde{Z}(N)
\]
\[
M_{B(N)}=M_{N}\rtimes G_{B(N)}\subseteq M_{N}\rtimes G_{N} = Z(N)
\]
with all groups $\QQ$-algebraic, $\tilde{G}_{B(N)}$, $\tilde{G}_{N}$,
$G_{B(N)}$, $G_{N}$ reductive, and $M_{N}$ the unipotent radical
in each case.
\end{prop}
While $G_{B(N)}$ depends on the choice of $\QQ$-split MHS in Lemma
\ref{lemma3}, its (isomorphic) image under the quotient by $M_{N}$
is $Gr_{0}^{W}M_{B(N)}$ (which does not).

The second result refines Lemma \ref{lemma4}:
\begin{prop}
\label{prop b}(i) $M_{B(N)}(\RR)$ acts transitively on $\tilde{B}^{\RR}(N)$;\\
(ii) $M_{N}(\CC)\rtimes G_{B(N)}(\RR)$ acts transitively on $\tilde{B}(N)$.
\end{prop}

\begin{rem}
As was pointed out by a referee (whose argument we shall paraphrase below), there is a simpler proof of Theorem \ref{thm1}(A).  The value of the above approach lies in the formula \eqref{eqn !!!} for $Gr^W_0$ of $\lm_{B(N)}$ (used for example in \cite[$\S\S$6-7]{KR}) and the use of deformations of $\tilde{\varphi}$ to parametrize points of $\tilde{B}(N)$ (used later in this paper and in \cite{KP2}).

The nontrivial inclusion $M_N\leq M_{B(N)}$ is equivalent to the statement that $M_N$ fix all ($\QQ$-)Hodge tensors common to all $(F^{\bullet},W(N)_{\bullet})\in \tilde{B}(N)$.  Let $\mathfrak{t}\in T(\lm)$ be such a tensor (of type $(p,p)$), and $\mu \in \lm_{N,\RR}$.  It will suffice to show that $\text{ad}(\mu)\mathfrak{t}=0$.  (We shall still write $\text{ad}$ for the action of $\lm$ on $T(\lm)$.)

Now $e^{\lambda \mu}\in M(\RR)$ ($\lambda\in\RR$ arbitrary) preserves $D$ and commutes with $e^N$, hence takes nilpotent orbits to nilpotent orbits.  So it preserves $\tilde{B}(N)$.  In particular, $(e^{-\lambda\mu}F^{\bullet},W(N)_{\bullet})$ belongs to $\tilde{B}(N)$ and we have $\mathfrak{t}\in e^{-\lambda\mu}F^{p}\cap W(N)_{2p} T(\lm)$.  These are sent (by $e^{\lambda\mu}$) to $(F^{\bullet},W(N)_{\bullet})$ resp. $e^{\lambda\mu}\mathfrak{t}\in F^p\cap W(N)_{2p}T(\lm_{\RR})$.  So $\left. \left(\frac{d}{d\lambda}e^{\lambda\mu}\mathfrak{t}\right)\right|_{\lambda=0}\in F^p\cap W(N)_{2p}T(\lm_{\RR})$, which implies that $\text{ad}(\mu)\mathfrak{t}\in F^p\cap W(N)_{2p} T(\lm_{\RR}).$ On the other hand, we also have $\mu\in W(N)_{-1}\lm_{\RR}$, whence $\text{ad}(\mu)\mathfrak{t}\in W(N)_{2p-1} T(\lm_{\RR})$.  Since $F^p\cap W(N)_{2p-1}T(\lm_{\RR})={0}$, $\text{ad}(\mu)\mathfrak{t}=0$.  

As the referee pointed out, taking $\mu=N$ yields $\text{ad}(N)\mathfrak{t}=0$, hence the interesting conclusion that $p\leq 0$.
\end{rem}

\section{Arbitrary rank}

We shall now retell the story of $\S\S2-4$ in the more general context
of a boundary component associated to a (finitely-generated) \emph{rational
nilpotent cone}
\[
\sigma:=\QQ_{\geq0}\langle N_{1},\ldots,N_{s}\rangle\subset\lm_{\QQ},
\]
where $N_{1},\ldots,N_{s}\in\lm_{\QQ}$ are commuting nilpotent elements.
(Without loss of generality these may be assumed to be strongly \emph{convex},
i.e. $\sigma\cap(-\sigma)=\{0\}$.) Write 
\[
\sigma^{\circ}:=\QQ_{>0}\langle N_{1},\ldots,N_{s}\rangle
\]
for its interior and $\langle\sigma\rangle$ for its $\QQ$-vector-space
closure, with $\sigma_{\RR}^{\circ} \subset \sigma_{\mathbb{R}} \subseteq \langle\sigma\rangle_{\RR} \subseteq \langle\sigma\rangle_{\CC}$ all having the obvious meaning.

A \emph{$\sigma$-nilpotent orbit} is a subset of $\check{D}_{M}$
of the form $e^{\langle\sigma\rangle_{\CC}}.F^{\bullet}$ where $F^{\bullet}$
satisfies \begin{equation}\label{eqn hatch}\begin{array}{cc}\text{(i) }N_{j}F^{p}\subset F^{p-1}\;(\forall j)\\\text{(ii) }e^{\sum_{j}\tau_{j}N_{j}}F^{\bullet}\in D_{M}\mspace{10mu} & \text{if all }\Im(\tau_{j})\gg0.\end{array}\end{equation}
As before, we set \begin{equation}\label{p23*} \tilde{B}(\sigma):=\left\{ \left.F^{\bullet}\in\check{D}_{M}\;\right|\; F^{\bullet}\text{ satisfies (i) and (ii)}\right\} , \end{equation}and then \begin{equation}\label{p23**} B(\sigma):=e^{\langle\sigma\rangle_{\CC}}\backslash \tilde{B}(\sigma) \end{equation}is the set of $\sigma$-nilpotent orbits. We assume $B(\sigma)\neq\emptyset$.
The \emph{rank} of $B(\sigma)$ is just $\dim\langle\sigma\rangle\, =:r$.

Assume $\tilde{B}(\sigma)\neq \emptyset$. A fundamental result of Cattani and Kaplan (\cite{CK}, Thm. 2.3(ii))
says that the weight filtration $W(N)_{\bullet}$ of $\lm$ is independent
of the choice of $N\in\sigma^{\circ}.$ Without loss of generality
we can put $N:=N_{1}+\cdots+N_{s}$, and write $W(\sigma)_{\bullet}\lm:=W(N)_{\bullet}\lm$,
$W_{\bullet}^{\sigma}M:=W_{\bullet}^{N}M$. It will also be convenient
to take elements $\hat{N}_{1},\ldots,\hat{N}_{r}\in\sigma^{\circ}$
giving a basis for $\langle\sigma\rangle$. We have the Lie algebras
\[
\lz(\sigma):=\bigcap_{j=1}^{s}\ker(\ad N_{j})=\bigcap_{i=1}^{r}\ker(\ad\hat{N}_{i})
\]
\[
\lm_{\sigma}:=\text{im}(\ad N)\cap\lz(\sigma)
\]
and their corresponding Lie groups $Z(\sigma),\, M_{\sigma}\leq M$.
Since $\lz(\sigma)\subset\lz(N)$, we have
\[
\begin{array}{ccc}
\lm_{\sigma} & = & \lm_{N}\cap\lz(\sigma)\\
 & = & \left(W(N)_{-1}\cap\lz(N)\right)\cap\lz(\sigma)\\
 & = & W(\sigma)_{-1}\cap\lz(\sigma)
\end{array}
\]
which is clearly independent of the choice of $N$.

As before we may interpret the elements $F^{\bullet}\in\tilde{B}(\sigma)$
as LMHS $(F^{\bullet},W(\sigma)_{\bullet})$ on $\lm$, and define the $\RR$-split loci\footnote{also nonempty by \cite[(2.3)]{CK}, \cite[Prop. 4.66]{CKS}, and the $e^{-i\delta}$ process from the end of $\S4$.} $\tilde{B}^{\RR}(\sigma)$, $B^{\RR}(\sigma)$ and the Mumford-Tate
group $M_{B(\sigma)}$, $M_{B^{\RR}(\sigma)}$. These MHS have a richer
structure than in the rank-one setting: according to \cite[Thm. 2.3(iii)]{CK},
the Hard Lefschetz isomorphisms and polarizations $B_{k}$ of Remark
\ref{rem ?}(c) are induced by any element of $\sigma^{\circ}$ (e.g.,
all the $\hat{N}_{i}$). (However, we shall put $B_{k}(v,w):=B(v,N^{k}w)$
with our choice of $N$ above.) Moreover, the $\{N_{j}\}$ all belong
to $I^{-1,-1}(\lm_{\CC})$ -- meaning that $\langle\sigma\rangle$
consists of Hodge tensors.

Accordingly, for a given $\vft\in\tilde{B}(\sigma)$, there is a reduction
in the information required to describe $\vf_{\text{split}}$ in \eqref{eqn*1},
where $\sigma$ replaces $N$ everywhere and 
\[
\tilde{Z}(\sigma):=\left\{ g\in M\;\left|\;\Ad(g)N_{j}=\psi(g)N_{j}\;(\forall j)\text{ for some character }\psi\right.\right\} .
\]
Namely, in order for the primitive spaces in Remark \ref{rem ??}(b),(b')
to match up with $Gr_{-k}^{W(\sigma)}\lz(\sigma)$ under $N^{k}$,
we must define%
\footnote{Note that $k\geq0$, and $\text{ad}(N_{j})\circ(\text{ad}(N))^{k}$
maps from $Gr_{k}^{W(\sigma)}\mathfrak{m}$ to $Gr_{-k-2}^{W(\sigma)}\mathfrak{m}$.%
}
\begin{equation}\label{eq*22}
P_{k}:=\bigcap_{j=1}^{s}\ker\left(\ad N_{j}\circ(\ad N)^{k}\right)\subset Gr_{k}^{W(\sigma)}\lm.
\end{equation}

\begin{prop}
(i) $Z(\sigma)/M_{\sigma}$ is reductive and acts faithfully on $\oplus_{k\geq0}P_{k}$.\\
(ii) The HS on $\oplus_{k}Gr_{k}^{W(\sigma)}\lm$ can be recovered
from that on $\oplus_{k\geq0}P_{k}$ by \begin{equation}\label{eqn**}Gr_{k}^{W(\sigma)}\lm=\bigoplus_{\ell\geq\max{\{0,-k\}}}\sum_{1\leq j_{1},\ldots,j_{\ell}\leq s}(\ad N_{j_{1}})\circ\cdots\circ(\ad N_{j_{\ell}})P_{k+2\ell}.\end{equation}\end{prop}
\begin{proof}
(i) is as before. For (ii), the mixed Lefschetz result of \cite[Thm. 4.3]{Ca1}
implies that
\[
Gr_{k}^{W(\sigma)}\lm=\ker\left(\ad \hat{N}_{i}\circ(\ad N)^{k}\right)\bigoplus\text{im}(\ad \hat{N}_{i})
\]
for each $i$. By the nondegeneracy of $Q(\cdot ,N^k \cdot )$, this gives
\[ \begin{array}{ccccc}
Gr_{k}^{W(\sigma)}\lm & = & \bigcap_{i=1}^{r} \ker \left( \text{ad}\hat{N}_i \circ (\text{ad} N)^k\right)  &\oplus & \left( \sum_{i=1}^{r} \text{im}(\text{ad}\hat{N}_i) \right) \\
& = &  P_{k}\oplus\left(\sum_{j=1}^{s}\text{im}(\ad N_{j})\right), \\
\end{array} \]
which inductively leads to \eqref{eqn**}.
\end{proof}
We are now ready for
\begin{thm}
\label{mainthm}For a nonempty Kato-Usui boundary component $B(\sigma)$
of $D_{M}$ associated to a nilpotent cone $\sigma\subset\lm_{\QQ}$,
the Mumford-Tate group $M_{B(\sigma)}$ satisfies the following:\vspace{2mm}

\emph{(A)} $M_{B(\sigma)}$ is contained in the centralizer $Z(\sigma)$
of the cone;\vspace{2mm}

\emph{(B)} $W_{-1}M_{B(\sigma)}=M_{\sigma}$ is its unipotent radical, with (abelian) group isomorphisms $ Gr^{W(\sigma)}_{-k}\mathfrak{m}_{\sigma,\mathbb{C}} \overset{\exp}{\underset{\cong}{\to}} Gr^W_{-k}M_{\sigma}(\mathbb{C})$ for each $k\geq 1$;
and\vspace{2mm}

\emph{(C)} $Gr_{0}^{W}M_{B(\sigma)}=\left.M_{B(\sigma)}\right/M_{\sigma}\;\left(\subset Z(\sigma)\left/M_{\sigma}\right.\right)$
is the $\QQ$-algebraic-group closure of the orbit $\left(Z_{\sigma}\left/M_{\sigma}\right.\right)(\RR).\vf_{\text{split}}$,
which may be regarded as the Mumford-Tate group of a set of polarized
HS%
\footnote{note that a ``polarized HS'' means a \emph{direct sum} of \emph{pure}
polarized HS%
} on $\oplus_{k\geq0}P_{k}$.\vspace{2mm}\\
The analogues of Definition \ref{def a}, Proposition \ref{prop a},
and Proposition \ref{prop b} all hold, with $N$ replaced by $\sigma$
everywhere.\end{thm}
\begin{proof}
Let $\left(F^{\bullet},W(\sigma)_{\bullet}\right)\in\tilde{B}^{\RR}(\sigma)$
with corresponding representation $\vft:\mathbb{S}\to\tilde{Z}(\sigma)$
(defined over $\RR$) and Deligne splitting $\ad Y$, $Y\in\lm_{\RR}$ (cf. the proof of Lemma \ref{lemma3}).
From the above discussion and \cite[(3.3)]{CK}, we have for each $i$
\[
\left\{ \begin{array}{c}
Y\in\text{im}(\ad\hat{N}_{i})\\
{}[Y,\hat{N}_{i}]=-2\hat{N}_{i}
\end{array}\right..
\]
From $\S1$ (and noting $\lm_{\sigma}=\cap_{i}\lm_{\hat{N}_{i}}$,
$M_{\sigma}=\cap_{i}M_{\hat{N}_{i}}$), we obtain that the subsets
of $\lm_{\RR}$ given by \begin{equation}\label{eqn!1}\left\{ \left.x\in\lm_{\RR}\;\right|\;[x,\hat{N}_{i}]=-2\hat{N}_{i},\, x\in\text{im}(\ad\hat{N}_{i})\;(\forall i)\right\} \end{equation}
\begin{equation}\label{eqn!2}\bigcap_{i}\left\{ Y+\lm_{\hat{N}_{i},\RR}\right\} =Y+\lm_{\sigma,\RR}\end{equation}
\begin{equation}\label{eqn!3}\bigcap_{i}\left\{ \left(\Ad M_{\hat{N}_{i}}(\RR)\right).Y\right\} =\left(\Ad M_{\sigma}(\RR)\right).Y\end{equation}
are equal. Clearly \eqref{eqn!1} is defined over $\QQ$, so the existence
of a $\QQ$-split LMHS goes through as in Lemma \ref{lemma3}, and
we henceforth assume $Y\in\lm_{\QQ}$ ($\implies\vft$ and $(F^{\bullet},W(\sigma)_{\bullet})$
are $\QQ$-split).

The formula \eqref{eqn becomes} for $\vft$ holds as before, with
the following difference: since the $\hat{N}_{i}$ all polarize $(F^{\bullet},W(\sigma)_{\bullet})$,
they produce $(-1,-1)$-morphisms of HS from each $E(k)$ to $E(k-2)$.
Hence they commute with $\phi$ and we have 
\[
\phi\in\left(\lz(\sigma)\cap E(0)\right)_{\RR}=\lg_{\sigma,\RR}\,\left(\overset{\cong}{\to}(\lz(\sigma)/\lm_{\sigma})_{\RR}\right).
\]
Using the equality of \eqref{eqn!2} and \eqref{eqn!3} we have 
\[
\Ad M_{\sigma}(\RR).\langle Y\rangle=\langle Y\rangle+\lm_{\sigma,\RR}
\]
\[
\Ad M_{\sigma}(\RR).\langle\phi\rangle\subset\lm_{\sigma,\RR}
\]
\[
\left\{ \ad\lg_{\sigma}\right\} \langle Y\rangle=\{0\}
\]
whereupon the remainder of the proof of Theorem \ref{thm1} goes through
with superficial changes.

Finally, the subquotients in (B) are abelian simply because the commutator of $W_{-k}M_{\sigma}$ and $W_{-\ell}M_{\sigma}$ belongs to $W_{-(k+\ell)}M_{\sigma}$ (due to compatibility of the grading and the bracket).  Since $W_{-k}M_{\sigma}$ is unipotent, $\exp : W(\sigma)_{-k}\lm_{\sigma}\to W_{-k}M_{\sigma}$ is an isomorphism of affine algebraic varieties, and its composition with $W_{-k}M_{\sigma} \twoheadrightarrow Gr^W_{-k}M_{\sigma}$ is visibly a homomorphism (as $Gr^W_{-k}M_{\sigma}$ is abelian) with kernel $W(\sigma)_{-(k+1)}\lm_{\sigma}$. 
\end{proof}

\section{Interlude on Kato-Usui spaces}

In the remainder of this paper we shall be concerned with the structure
of boundary components $B(\sigma)$ and certain quotients thereof.
The purpose of this section is to offer the reader a glimpse of how
these quotients fit into partial compactifications of quotients $\Gamma\backslash D_{M}$.
The material surveyed here is done only for period domains in \cite{KU},
but extends in a straightforward way to the more general Mumford-Tate
domain setting.

Let $\Sigma$ be a \emph{fan} in $\lm$, that is, a \emph{set} of
(finitely generated, convex) rational nilpotent cones in $\lm$ intersecting
in faces, which is \emph{closed} under the operation of taking faces.
We define
\[
D_{M,\Sigma}:=\coprod_{\sigma\in\Sigma}\left\{ \left.Z\subset\check{D}_{M}\;\right|\; Z\text{ is a }\sigma\text{-nilpotent orbit}\right\} 
\]
\[
=\coprod_{\sigma\in\Sigma}B(\sigma),\mspace{200mu}
\]
noting that this always contains $B(\{0\})=D_{M}$. In particular,
we shall write
\[
D_{M,\sigma}:=D_{M,\{\text{faces of }\sigma\}},
\]
which in the rank one case $\sigma=\QQ_{\geq0}\langle N\rangle$ is
nothing but $D_{M}\amalg B(N)$.

Next take $\Gamma\subset M(\ZZ)$ to denote a neat%
\footnote{Recall that this means the subgroup of $\CC^{*}$ generated by the
eigenvalues of the elements of $\Gamma$ acting on $\lm$ is torsion-free.
Using the Jordan decomposition for $M$, it follows that the action
on any tensor space $T^{a,b}\lm$ (and its subquotients) is also free
of torsion eigenvalues. Neat subgroups of finite index always exist
in $M(\ZZ)$ \cite[Prop. 17.6]{Bo}.} subgroup of finite index, and consider the monoid
\[
\Gamma(\sigma):=\Gamma\cap\exp(\sigma_{\RR})
\]
with group-theoretic closure $\Gamma(\sigma)^{\text{gp}}$. We assume
that $\Gamma$ is strongly compatible with $\Sigma$, i.e.
\begin{enumerate}
\item $\Ad(\gamma).\sigma\in\Sigma$ ($\forall$$\gamma\in\Gamma$,$\sigma\in\Sigma$)
\item $\sigma_{\RR}=\RR_{\geq0}\langle\log\Gamma(\sigma)\rangle$.
\end{enumerate}
By (1), $Z\mapsto\gamma.Z$ induces maps $B(\sigma)\mapsto B(\Ad(\gamma).\sigma)$
and so the quotient $\Gamma\backslash D_{\Sigma}$ makes sense on
the set-theoretic level. 

The essence of Theorem A of \cite{KU} can be transcribed as follows: 
\begin{thm}
\label{thm Ku}$\Gamma\backslash D_{M,\Sigma}$ admits the structure of a logarithmic manifold,
which is Hausdorff in the strong topology. For each $\sigma\in\Sigma$, the natural map
\[
\Gamma(\sigma)^{\text{gp}}\backslash D_{M,\sigma}\;\overset{\Upsilon_{\sigma}}{\longrightarrow}\;\Gamma\backslash D_{M,\Sigma}
\]
is open and locally an isomorphism.
\end{thm}

Let $\Gamma_{\sigma}\subseteq\Gamma$ denote the largest subgroup
stabilizing $\sigma$ (viz., $Ad(\gamma).\sigma\subseteq\sigma$).
Noting that $\Gamma(\sigma)^{\text{gp}}$ acts trivially on $B(\sigma)$,\footnote{Since $\Gamma(\sigma)^{gp}\leq (e^{\sigma_{\mathbb{R}}})^{gp}=e^{\langle\sigma\rangle_{\mathbb{R}}} \leq e^{\langle\sigma\rangle_{\mathbb{C}}}$, this is an immediate consequence of the definition \eqref{p23**} of $B(\sigma)$ ($e^{\langle\sigma\rangle_{\mathbb{C}}}$ already having been quotiented out).}
we have 
\[
B(\sigma)\;\subset\;\left(\Gamma(\sigma)^{\text{gp}}\backslash D_{M,\sigma}\right)
\]
with image \begin{equation}\label{eqn*}\overline{B(\sigma)} \; :=\; \Gamma _{\sigma} \backslash B(\sigma) \end{equation}
under $\Upsilon_{\sigma}$. It is these boundary component quotients
which will be of particular interest in $\S7$.

For the proof of Theorem \ref{thm Ku}, we prefer to drop (as before) the subscript $M$, and shall instead denote the objects pertaining
to the ambient period domain $\mathbf{D}=G(\mathbb{R})/H$ ($G=Aut(\lm,B)$) in boldface.  In \cite{KU}, Theorem 6.1 is proved for $\Gamma\backslash \mathbf{D}$,\footnote{Note that \cite{KU} imposes no finite-index assumption on $\Gamma$; in particular, we can take $\Gamma\leq M(\mathbb{Z}) \leq G(\mathbb{Z})$ (and $\Sigma$ to consist of cones $\sigma\subset \lm\subseteq \mathfrak{g}$).} using properties of the fundamental diagram which connects various enlargements of $\mathbf{D}$ to the Borel-Serre space attached to the set of all maximal compact subgroups of $G(\mathbb{R})$.  Though rehashing this theory for Mumford-Tate domains would undoubtedly be possible (for instance, it can be shown that the multivariable $SL_2$-orbit theory of \cite{CKS} extends to this setting), the proof we offer here requires constructing only $D_{\Sigma}^{\#}$ and the torsors $E_{\sigma}$.  (These torsors, which provide the ``glue'' between the M-T-domain and boundary-component quotients, will be valuable for extending Carayol's construction of Fourier coefficients \cite{C} in anticipated future work.)

While we lack the space here to introduce all the concepts and notation required in the proof, we have written it to be accessible after a reading of \cite[Chap. 0]{KU}, with a copy of that book at hand.  It is not used in the rest of the paper.

\begin{proof}[Proof of Theorem \ref{thm Ku}]
We start with a brief outline.  The (possibly nonsimplicial) cone $\sigma$ gives rise to a toric variety $\text{toric}_{\sigma}$. We will define $E_{\sigma}$ as a subset of $\text{toric}_{\sigma}\times \check{D}$, and show it is a $\langle \sigma \rangle_{\CC}$-torsor over $\Gamma(\sigma)^{gp}\backslash D_{\sigma}$. (The latter is the disjoint union of quotients of $D$ and the $B(\tau)$, as $\tau$ runs over the nontrivial faces of $\sigma$.)  We then use the strong topology of $E_{\sigma}$ in $\text{toric}_{\sigma}\times \check{D}$ to topologize $\Gamma(\sigma)^{gp}\backslash D_{\sigma}$, and the canonical log structure on $\text{toric}_{\sigma}$ to put the structure of a log manifold on $E_{\sigma}$ (and thus on  $\Gamma(\sigma)^{gp}\backslash D_{\sigma}$).  This involves checking that $E_{\sigma}$ is an open subset of the vanishing locus of a finite collection of logarithmic 1-forms.  To obtain the statements involving $D_{\Sigma}$, we pass to its real blow-up $D_{\Sigma}^{\#}$ and follow the topological arguments in \cite{KU}, avoiding some of their complexity by using the results already proved in [op. cit.] for the ambient period domain.

To begin the proof itself, the first step is to construct the torsors as \emph{sets}.  Associated to $q\in \text{toric}_{\sigma}:=\text{Spec}(\mathbb{C}[\Gamma (\sigma)^{\vee}])_{an}$ is a face $\sigma(q)$ of $\sigma$ \cite[3.3.2]{KU}, and a class $[q]\in\langle\sigma\rangle_{\CC} \slash \left( \langle \sigma (q)\rangle_{\CC} + \log (\Gamma(\sigma)^{gp})\right)$ \cite[3.3.5]{KU} any lift of which to $\langle \sigma \rangle_{\CC}$ we denote by $\log_{\sigma}(q)$.  (For example, if $q_0 = (q_0^{(1)},\ldots ,q_0^{(r)} )\in(\CC^*)^r\subset \text{toric}_{\sigma}$, then $\sigma (q_0)=\{0\}$; and if furthermore $\sigma = \QQ_{\geq 0}\langle N_1,\ldots ,N_r\rangle$ is simplicial and $\log(\Gamma(\sigma)^{gp})=\ZZ\langle N_1,\ldots ,N_r\rangle$, then we have $\log_{\sigma}(q)=\frac{1}{2\pi i}\sum_{j=1}^r \log(q_0^{(j)})N_j$.)  Now define\footnote{for the ambient period domain case one also has $\check{\mathbf{E}}_{\sigma}:=\text{toric}_{\sigma}\times \check{\mathbf{D}}$ and so forth, as in \cite{KU}.}
\begin{flalign*}
\; \; & \check{E}_{\sigma} := \text{toric}_{\sigma}\times \check{D} &\\
\; \; & \tilde{E}_{\sigma} := \left\{ \left. (q,F^{\bullet})\in\check{E}_{\sigma} \right| N(F^{\bullet})\subset F^{\bullet -1}\; \forall N\in \sigma(q) \right\} & \\
\; \; & E_{\sigma} := \left\{ \left. (q,F^{\bullet})\in \check{E}_{\sigma} \right| e^{\langle \sigma(q)\rangle_{\CC}} e^{\log_{\sigma}q} F^{\bullet} \;\text{is a}\; \sigma(q)\text{-nilpotent orbit}\right\} &
\end{flalign*}
and note that since $\sigma(q)$ runs over all faces of $\sigma$, the natural map
\[
\Theta_{\sigma}:\, E_{\sigma} \to \underset{\underset{\text{face}}{\tau\subseteq \sigma}}{\coprod}\Gamma(\sigma)^{gp}\backslash B(\tau) = \Gamma(\sigma)^{gp}\backslash D_{\sigma}
\]
is surjective.  In fact, the action of $\langle\sigma\rangle_{\CC}$ on $E_{\sigma}$ by $\left( a,(q,F^{\bullet})\right)\mapsto\left( e(a)\cdot q,e^{-a}F^{\bullet}\right)$ ($e(a)\in(\CC^*)^r$) is free (the proof of \cite[7.2.9]{KU} applies verbatim), and $\Theta_{\sigma}$ is the quotient by this action.

We shall also need the set
\[
D_{\Sigma}^{\#} :=\underset{\sigma \in \Sigma}{\coprod} e^{i\langle \sigma\rangle_{\RR}}\backslash \tilde{B}(\sigma) \; \left( \twoheadrightarrow D_{\Sigma} \right)
\]
of ``$\sigma$-nilpotent $i$-orbits'' in $\check{D}$, and the corresponding ``torsor''.  Writing $|\text{toric}|_{\sigma}$ for the analytic closure of $(\RR_{>0})^r$ in $\text{toric}_{\sigma}$ (and noting that for $q\in |\text{toric}|_{\sigma}$ we may take $\log_{\sigma}(q)\in i\langle \sigma \rangle_{\RR}$), we may define $\check{E}_{\sigma}^{\#} :=|\text{toric}|_{\sigma}\times \check{D}$ and (just as above) the quotient
\[
\Theta_{\sigma}^{\#} :\, E_{\sigma}^{\#} :=\check{E}_{\sigma}^{\#}\cap E_{\sigma} \twoheadrightarrow D_{\sigma}^{\#}
\]
by a free $i\langle \sigma\rangle_{\RR}$-action, which fits together with $\Theta_{\sigma}$ in a commuting square.

The next step is to introduce topologies on each space, and log structures (for which we refer to \cite[Chap. 2]{KU}) on some.  Given a subset $\mathcal{S}\overset{\imath}{\hookrightarrow}\mathrm{X}$ of a (complex) analytic space $\mathrm{X}$, the \emph{strong topology} of $\mathcal{S}$ in $\mathrm{X}$ is the finest topology on $\mathcal{S}$ for which any map $\lambda:\mathrm{Y}\to\mathcal{S}$ ($\mathrm{Y}$ analytic) with $\imath\circ\lambda$ analytic is continuous.  With the usual analytic topology on $\check{E}_{\sigma}$, we place the strong topology (in $\check{E}_{\sigma}$) on $\tilde{E}_{\sigma}$ and $E_{\sigma}$, and topologize $E_{\sigma}^{\#}$ as a subspace of $E_{\sigma}$ and $\Gamma(\sigma)^{gp}\backslash D_{\sigma}$ [resp. $D_{\sigma}^{\#}$] as a quotient of $E_{\sigma}$ [resp. $E_{\sigma}^{\#}$]. On $\Gamma \backslash D_{\Sigma}$ [resp. $D_{\Sigma}^{\#}$] we place the strongest topology for which the inclusions of $\Gamma(\sigma)^{gp}\backslash D_{\sigma}$ [resp. $D_{\sigma}^{\#}$] are continuous ($\forall\sigma$).  Since $E_{\sigma}^{\#}\to E_{\sigma}$ is continuous by definition, so is $D_{\Sigma}^{\#}\twoheadrightarrow \Gamma\backslash D_{\Sigma}$.

The toric variety $\text{toric}_{\sigma}$ has a canonical fs logarithmic structure associated to the divisor $\Delta_{\sigma}:=\text{toric}_{\sigma}\setminus (\CC^*)^r$.  We obtain (inverse image) log structures on $E_{\sigma}$, $\tilde{E}_{\sigma}$, $\check{E}_{\sigma}$ using $E_{\sigma}\subseteq \tilde{E}_{\sigma}\subseteq \check{E}_{\sigma} \twoheadrightarrow \text{toric}_{\sigma}$.  Now a logarithmic manifold \cite[3.5.7]{KU} is (locally) the vanishing locus\footnote{in the sense of pointwise pullback; the pullback of a log differential form (in the Kato-Usui sense) to a point in the support of the log structure (like $\Delta_{\sigma}\subset \text{toric}_{\sigma}$) \emph{need not be zero}.} of a set of logarithmic differential 1-forms $\omega^1$ on an fs log analytic space.  For the ambient period domain, \cite[3.5.10]{KU} shows that $\tilde{\mathbf{E}}_{\sigma}\subset \check{\mathbf{E}}_{\sigma}$ is (locally) cut out in this way by a finite number of sections of the subsheaf $\omega^1_{\text{toric}_{\sigma}} \otimes \mathcal{O}_{\check{\mathbf{E}}_{\sigma}}\langle \sigma \rangle_{\CC} \subset \omega^1_{\check{\mathbf{E}}_{\sigma}}$. Since the restriction of this subsheaf to $\check{E}_{\sigma}$ injects into $\omega^1_{\check{E}_{\sigma}}$, applying the proof of [loc. cit.] verbatim shows that $\tilde{E}_{\sigma} = \check{E}_{\sigma} \cap \tilde{\mathbf{E}}_{\sigma}$ is a logarithmic manifold.  (This is perhaps the crucial step.)

Now we shall apply repeatedly the following continuity criterion ({\bf CC}): given subsets $\mathcal{S}_i \overset{\imath_i}{\subset} \mathrm{X}_i$ ($i=1,2$) of analytic spaces, endowed with the strong topology, a map $\mathcal{S}_1 \overset{f}{\to} \mathcal{S}_2$ is continuous if for any $\lambda:\mathrm{Y}\to \mathcal{S}$ ($\mathrm{Y}$ analytic) with $\imath_1\circ \lambda$ analytic, we have  $\imath_2\circ f\circ\lambda$ analytic. So for example $\tilde{E}_{\sigma}\hookrightarrow \tilde{\mathbf{E}}_{\sigma}$ is continuous (by {\bf CC}); and since $\mathbf{E}_{\sigma}\subset \tilde{\mathbf{E}}_{\sigma}$ is open \cite[$\S$7.1]{KU}, $E_{\sigma}=\tilde{E}_{\sigma}\cap\mathbf{E}_{\sigma}$ is open in $\tilde{E}_{\sigma}$ (in the strong topology), hence a logarithmic manifold.  The action of $\langle\sigma\rangle_{\CC}$ on $E_{\sigma}$ is continuous (by {\bf CC}) and $E_{\sigma}\,(\subset \check{E}_{\sigma})$ is obviously Hausdorff, so by \cite[7.2.6(ii)]{KU} the properness of the action of $\langle\sigma\rangle_{\CC}$ on $\mathbf{E}_{\sigma}$ \cite[7.2.2(i)]{KU} implies the same for $E_{\sigma}$.  We conclude that
\begin{flalign}\label{6A}
\; \; & \Gamma(\sigma)^{gp}\backslash D_{\sigma}\, \left( = \langle\sigma\rangle_{\CC}\backslash E_{\sigma} \right)\;\text{is Hausdorff.}&
\end{flalign}
Following \cite[7.3.5,7]{KU} verbatim (in the M-T setting) yields:
\begin{flalign}\label{6B}
\;\; &\Gamma(\sigma)^{gp}\backslash D_{\sigma}\;\text{is a log manifold}&
\end{flalign}
(with $E_{\sigma}$ a $\langle\sigma\rangle_{\CC}$-torsor over it in the category of log manifolds); and
\begin{flalign}\label{6C}
\;\; &\left( \Gamma(\sigma)^{gp}\backslash D_{\sigma}\right)^{\log}\cong\Gamma(\sigma)^{gp}\backslash D_{\sigma}^{\#}.&
\end{flalign}
(In \eqref{6C}, $X^{\log}$ denotes a real blowup of a log manifold $X$ along the support of the log structure, cf. \cite[2.2.3]{KU}; the main point is that $X^{\log}\twoheadrightarrow X$ is proper.)  Moreover, the inclusion $D_{\Sigma}^{\#}\hookrightarrow \mathbf{D}_{\Sigma}^{\#}$ inherits continuity from $E_{\sigma}\hookrightarrow \mathbf{E}_{\sigma}$ (because of how we defined the topologies), and so $\mathbf{D}_{\Sigma}^{\#}$ Hausdorff $\implies$
\begin{flalign}\label{6D}
\;\; &D_{\Sigma}^{\#}\;\text{is Hausdorff.}&
\end{flalign}
Using the {\bf CC}, we easily check continuity of the maps $E_{\sigma}\underset{\cong}{\to}E_{(\text{ad}\gamma)\sigma}$ induced by $\gamma\in\Gamma$, hence of the action of $\Gamma$ on $D_{\Sigma}^{\#}$.  Applying \cite[7.4.2(ii)]{KU} (properness of $\Gamma$'s action on $\mathbf{D}_{\Sigma}^{\#}$), \cite[7.2.6(ii)]{KU}, and \eqref{6D} $\implies$ $\Gamma$ acts properly on $D_{\Sigma}^{\#}$ $\implies$
\begin{flalign}\label{6E}
\;\; &\Gamma\backslash D_{\Sigma}^{\#}\;\text{is Hausdorff.}&
\end{flalign}
Employing the arguments in \cite[7.2.5+7.4.4-5]{KU} in the M-T setting together with \eqref{6D} $\implies$
\begin{flalign}\label{6F}
\;\; &D_{\Sigma}^{\#} \twoheadrightarrow \Gamma\backslash D_{\Sigma}^{\#}\;\text{is a local homeomorphism;}&
\end{flalign}
while following \cite[7.4.7-8]{KU} together with \eqref{6C},\eqref{6F} $\implies$
\begin{flalign}\label{6G}
\;\; &\Gamma(\sigma)^{gp}\backslash D_{\sigma} \rightarrow \Gamma \backslash D_{\Sigma}\;\text{is a local homeomorphism}\;(\forall\sigma\in\Sigma ).&
\end{flalign}
Finally \eqref{6B} and \eqref{6G} imply that $\Gamma \backslash D_{\Sigma}$ is a log manifold; while \eqref{6C} and \eqref{6G} $\implies$ $(\Gamma \backslash D_{\Sigma} )^{\log} \cong \Gamma \backslash D_{\Sigma}^{\#}$ $\implies$ $\Gamma\backslash D_{\Sigma}^{\#} \twoheadrightarrow \Gamma \backslash D_{\Sigma}$ proper, which with \eqref{6E} and \cite[7.2.6(i)]{KU} shows $\Gamma \backslash D_{\Sigma}$ is Hausdorff.
\end{proof}

\begin{rem}
\label{rem ku}Kato and Usui do not offer any magic prescription for
how to choose $\Sigma$. (One of their key examples is simply to take
all rays generated by rational nilpotent elements of $\lm$.) Indeed,
it can be tedious to check that a given choice of $\sigma$ (or $N$)
satisfies criterion \eqref{eqn hatch}(ii) in the definition of a
($\sigma$-)nilpotent orbit. We have worked around this issue by assuming
$B(\sigma)\neq\emptyset$ above, but obviously must face it for the
examples in $\S8$. By \cite[(2.3)]{CK}, in order for $B(\sigma)$
to be nonempty, the \emph{necessary} conditions (on $\sigma$) are:\vspace{2mm}

(a) $(\ad N_{j})^{k+1}=0$ $(\forall j)$, where $k$ is the level%
\footnote{Recall that for a weight $n$ HS with Hodge numbers $\{h^{p,n-p}\}_{p\in\ZZ}$,
the \emph{level} is the difference $\max\{p|h^{p,n-p}\neq0\}-\min\{p|h^{p,n-p}\neq0\}$;
in our (adjoint HS) setting, $n=0$ and the level is even.%
} of the Hodge structures (on $\lm$) in $D_{M}$; and\vspace{2mm}

(b) all $N\in\sigma^{\circ}$ define the same $W(N)_{\bullet}$.\vspace{2mm}\\
It turns out (by \cite[Prop. 4.66]{CKS}; see the ``converse to (2.3)'' in \cite{CK}) that
a given $F^{\bullet}$ satisfying \eqref{eqn hatch}(i) gives a $\sigma$-nilpotent
orbit (i.e. \eqref{eqn hatch}(ii) holds) if one has (a),(b), and
\vspace{2mm}

(ii') the pure Hodge structures induced by $F^{\bullet}$ on the%
\footnote{These $\hat{P}_{k}$ (which contain our $P_{k}$) are the more conventional
primitive spaces.%
} 
\[
\hat{P}_{k}:=\ker\left((\ad N)^{k+1}\right)\subset Gr_{k}^{W(N)}\lm
\]
are polarized by the restriction of $B\left((\cdot),(\ad N)^{k}(\cdot)\right)$
(for some $N\in\sigma^{\circ}$).\vspace{2mm}\\
We expect, but have not proved in full, that this only needs to be
checked on the $\{P_{k}\}$ (as defined \emph{for cones} in \eqref{eq*22}). 
\end{rem}
{}
\begin{rem} 
\label{HSD remark}In contrast, consider the case where $D$ is Hermitian symmetric with $F^{-1}\lm_{\CC}=\lm_{\CC}$,\footnote{This encompasses many examples arising from VHS of higher weight, such as the CY variations of \cite{FL}, and includes cases where $\lm$ has type $\mathfrak{e}_6$ or $\mathfrak{e}_7$.} and $\Gamma\leq M(\ZZ)$ is a neat arithmetic subgroup.  Then $\Gamma \backslash D$ is a (connected) Shimura variety, and the work \cite{AMRT} produces a fan $\Sigma$ strongly compatible with $\Gamma$ and a corresponding smooth toroidal compactification $\overline{\Gamma\backslash D}$ of $\Gamma \backslash D$.  The identification of $\overline{\Gamma\backslash D}$ with $\Gamma \backslash D_{\Sigma}$ as constructed above may be seen as follow (with more details in \cite{KP2}).

Let $\mathbf{F}\subset cl(D)\setminus D$ be a Baily-Borel boundary component of $D$. Associated to $\mathbf{F}$ is an open cone $C(\mathbf{F})\subset \lm_{\RR}$ \cite[III.4.1]{AMRT} and a subset $D(\mathbf{F})\subset \check{D}$ \cite[III.4.5]{AMRT}. For any finitely generated rational nilpotent cone $\sigma = \QQ_{\geq 0}\langle N_1,\ldots N_s\rangle $ with $\sigma^{\circ}\subset C(\mathbf{F})$, it is shown in \cite[$\S$5.3]{KP2} that $D(\mathbf{F})=\tilde{B}(\sigma)$.  The strata in $\overline{\Gamma\backslash D}$, which have the structure of (connected) \emph{mixed} Shimura varieties \cite{Mi2,Pi}, are then given by $\Gamma_{\sigma} e^{\langle \sigma \rangle_{\CC}}\backslash D(\mathbf{F})$ ($\sigma \in \Sigma$) \cite[proof of III.5.2]{AMRT}, which is to say $\Gamma_{\sigma}\setminus B(\sigma)$. Hence they are precisely the strata appearing in $\Gamma\backslash D_{\Sigma}$.

Moreover, in this setting (by \cite{Mi2,Pi}) the $\Gamma_{\sigma}\backslash B(\sigma)$ always have canonical models over $\bar{\QQ}$. In the next two sections we shall investigate when this might be so even for $D$ non-Hermitian. 
\end{rem}

\section{Boundary component structure}

We are interested in when $\Gamma_{\sigma}\backslash B(\sigma)$ has
an arithmetic algebraic structure -- i.e., a canonical model over
$\bar{\QQ}$ -- especially in those cases when $\Gamma\backslash D_{M}$
does \emph{not}. In order to uncover this we have to understand the
fibration structure, starting with the prequotient $B(\sigma)$. We
make the slight notational change of letting $\vft_{0}$ (or $(F_{0}^{\bullet},W(\sigma)_{\bullet})$)
denote our $\QQ$-split point, and $\vft$ (or $(F^{\bullet},W(\sigma)_{\bullet})$)
an arbitrary point of $\tilde{B}(\sigma)$. We shall write $I_0^{p,q}(\lm_{\CC})$ [resp. $I^{p,q}(\lm_{\CC})$]
for the Deligne subspaces corresponding to $\tilde{\varphi}_0$ [resp. $\tilde{\varphi}$]. Let $n$ denote the level of the (weight $0$) Hodge structures parametrized by $D_{M}$.

By Proposition \ref{prop b}, together with the faithfulness of the
action of $M_{N}(\RR)$ on $Y_{0}$, we obtain:
\[
\tilde{B}^{\RR}(\sigma)=M_{B(\sigma)}(\RR).\vft_{0}=\left.\left(M_{\sigma}(\RR)\rtimes G_{B(\sigma)}(\RR)\right)\right/K_{\vft_{0}}
\]
where $K_{\vft_{0}}:=\left(G_{B(\sigma)}(\RR)\right)^{\vft_{0}}$;
and
\[
\tilde{B}(\sigma)=\left(M_{\sigma}(\CC)\rtimes G_{B(\sigma)}(\RR)\right).(F_{0}^{\bullet},W(\sigma)_{\bullet})
\]
\[
=\left.\left(M_{\sigma}(\CC)\rtimes G_{B(\sigma)}(\RR)\right)\right/K_{F_{0}^{\bullet}}
\]
where $K_{F_{0}^{\bullet}}=H_{F_{0}^{\bullet}}\rtimes K_{\vft}$ and
$H_{F_{0}^{\bullet}}=\left(M_{\sigma}(\CC)\right)_{F_{0}^{\bullet}}$
(stabilizer of $F_{0}^{\bullet}$). (The action on $W(\sigma)_{\bullet}$
is of course trivial.) Passing to nilpotent orbits, we have
\[
B(\sigma)\;=\; e^{\langle\sigma\rangle_{\CC}}\backslash\tilde{B}(\sigma)\;,\;\;\;\; B^{\RR}(\sigma)\;=\; e^{\langle\sigma\rangle_{\RR}}\backslash M_{B(\sigma)}(\RR)/K_{\vft_{0}}.
\]

The Mumford-Tate \emph{domain} of a generic $\vf_{\text{split}}$,
viz.
\[
D(\sigma):=G_{B(\sigma)}(\RR).(\vf_{0})_{\text{split}}\cong\left.G_{B(\sigma)}(\RR)\right/K_{\vft},
\]
is contained in the product of period domains%
\footnote{More precisely, for a generic $\vf_{\text{split}}$, $D(\sigma)$
surjects onto the MT domains of all the irreducible pure Hodge structures
into which $\vf_{\text{split}}$ decomposes. %
} for the $(P_{k},Q_{k})$. Sending $\vft$ resp. $e^{\langle\sigma\rangle_{\CC}}\vft$
to $\vf_{\text{split}}$ gives bundles \begin{equation}\label{eqn??}\xymatrix{\tilde{B}(\sigma) \ar @{->>} [rr] \ar @{->>} [rd]_{\tilde{\rho}_{\sigma}} & & B(\sigma) \ar @{->>} [ld]^{\rho_{\sigma}} \\ & D(\sigma) } \end{equation}
with fibers through $\vft$ 
\[
\tilde{\mathfrak{F}}_{\vft}=M_{\sigma}(\CC).\left(F^{\bullet},W(\sigma)_{\bullet}\right)\;,\;\;\;\mathfrak{F}_{\vft}=M_{\sigma}(\CC).e^{\langle\sigma\rangle_{\CC}}F^{\bullet}.
\]
To refine \eqref{eqn??}, recall that we have a filtration of the
Lie group
\[
\mathcal{M}:=M_{\sigma}(\CC)\rtimes G_{B(\sigma)}(\RR)
\]
by normal subgroups $W_{-k}\mathcal{M}=W_{-k}M_{\sigma}(\CC)$ ($k\geq1$).
Writing $\mathcal{K}:=K_{F_{0}^{\bullet}}$, set 
\[
\tilde{B}(\sigma)_{(k)}:=\left.\left(\frac{\mathcal{M}}{W_{-(k+1)}\mathcal{M}}\right)\right/\left(\frac{\mathcal{K}}{W_{-(k+1)}\mathcal{K}}\right)
\]
\[
B(\sigma)_{(k)}:=\left\{ \begin{array}{cc}
\tilde{B}(\sigma)_{(1)}\;\;\;,\;\;\; & k=1\\
e^{\langle\sigma\rangle_{\CC}}\backslash\tilde{B}(\sigma)_{(k)}, & k>1
\end{array}\right..
\]
The fibration $\tilde{\rho}_{\sigma}$ then factors
\[
\tilde{B}(\sigma)\twoheadrightarrow\cdots\sur\tilde{B}(\sigma)_{(k)}\underset{\tilde{\rho}_{\sigma}^{(k)}}{\sur}\tilde{B}(\sigma)_{(k-1)}\sur\cdots\sur\tilde{B}(\sigma)_{(1)}\underset{\tilde{\rho}_{\sigma}^{(1)}}{\sur}D(\sigma),
\]
with fibers (of $\tilde{\rho}_{\sigma}^{(k)}$, through $\vft$) \begin{equation}\label{eqn???}\tilde{\mathfrak{F}}_{\vft}^{(k)}=\frac{Gr_{-k}^{W}\mathcal{M}}{Gr_{-k}^{W}\mathcal{K}}=\frac{Gr_{-k}^{W}M_{\sigma}(\CC)}{Gr_{-k}^{W}H_{F^{\bullet}}}\overset{\cong}{\underset{\exp}{\longleftarrow}}\frac{Gr_{-k}^{W(\sigma)}(\lm_{\sigma,\CC})}{F^{0}Gr_{-k}^{W(\sigma)}(\lm_{\sigma,\CC})},\end{equation}
in which the right-hand side renders the abelian structure of $\tilde{\mathfrak{F}}_{\tilde{\varphi}}^{(k)}$ visible (cf. Theorem \ref{mainthm}(B)). The version with
tildes removed only differs from \eqref{eqn???} at $k=2$: \begin{equation}\label{eqn????}\mathfrak{F}_{\vft}^{(2)}=\frac{Gr_{-2}^{W(\sigma)}(\lm_{\sigma,\CC})}{F^{0}\{Gr_{-2}^{W(\sigma)}(\lm_{\sigma,\CC})\}+\langle\sigma\rangle_{\CC}}.\end{equation}
In \eqref{eqn???} and \eqref{eqn????}, $Gr_{-k}^{W(\sigma)}(\lm_{\sigma})$
is a Hodge structure of weight $-k$ and its $F^{0}$ depends on $\vf_{\text{split }}=\tilde{\rho}_{\sigma}(\vft)\in D(\sigma)$.
\begin{rem}
To make all of this more concrete in the rank one case, recall that
\begin{equation}\label{eqn hatch'}\lm=\bigoplus_{\ell=0}^{n}\left(\bigoplus_{j=0}^{\ell}N^{j}\tilde{P}_{\ell}\right).\end{equation}
We may view (for $k\geq1$) \begin{equation}\label{eqn hatch ?}\tilde{\mathfrak{F}}_{\vft}^{(k)}\subseteq\bigoplus_{\ell=0}^{n}\bigoplus_{j=0}^{\ell}\frac{\text{Hom}_{\CC}\left(P_{\ell},N^{j}P_{\ell-k+2j}\right)}{F^{0}\text{Hom}_{\CC}\left(P_{\ell},N^{j}P_{\ell-k+2j}\right)}\end{equation}
as being cut out by (among other equations) polarization conditions.
(For higher rank, $\lm=\bigoplus_{\ell=0}^{n}\left\{ \bigoplus_{j=0}^{\ell}\left(\sum_{i_{1}\leq\cdots\leq i_{j}}N_{i_{1}}\cdots N_{i_{j}}P_{\ell}\right)\right\} $
leads to a slightly more complicated formula.) For the fiber through
$\vft_{0}$, here is how this works: thinking of $g\in W_{-k}\mathcal{M}$
as an automorphism of \eqref{eqn hatch'}, we write $g=\text{id.}+\tilde{g}+\tilde{g}'$,
where
\[
\tilde{g}\in\oplus_{\ell}\text{Hom}_{\CC}\left(E(\ell),E(\ell-k)\right)
\]
\[
\tilde{g}'\in\oplus_{\ell}\text{Hom}_{\CC}\left(E(\ell),W(N)_{\ell-k-1}\lm\right)
\]
and $\tilde{g}$ is determined by its ``components''
\[
\tilde{g}_{(\ell,j)}\in\text{Hom}_{\CC}\left(P_{\ell},N^{j}P_{\ell-k+2j}\right).
\]
Given $\alpha\in\tilde{P}_{\ell}$, $\beta\in\tilde{P}_{\ell-k+2j}$,
and $q\in\ZZ_{\geq0}$ we have 
\[
B(\alpha,N^{q}\beta)=B(g\alpha,gN^{q}\beta)=B(g\alpha,N^{q}g\beta)
\]
\[
=\left\{ B(\alpha,N^{q}\beta)+B(\tilde{g}\alpha,N^{q}\tilde{g}\beta)\right\} +B(\tilde{g}\alpha,N^{q}\beta)+B(\alpha,N^{q}\tilde{g}\beta)
\]
\[
\mspace{200mu}+\left\{ \text{terms involving }\tilde{g}'\right\} .
\]
Choosing $q=\ell-k+j$ forces $B(\alpha,N^{q}\beta)$ and both bracketed
terms to be zero, so that
\[
B(\tilde{g}\alpha,N^{q}\beta)+B(\alpha,N^{q}\tilde{g}\beta)=0.
\]
Since $\lm$ is $N$-polarized (and $B$ pairs $E(k)$ with $E(-k)$),
we conclude that $\tilde{g}_{(\ell,j)}$ and $\tilde{g}_{(\ell-k+2j,k-j)}$
determine one another provided the subscripts are distinct. For $k$
even and $j=\frac{k}{2}$ they coincide, and $B$ directly imposes
conditions on $\tilde{g}_{(\ell,\frac{k}{2})}$:
\[
B(\tilde{g}b,N^{\ell-\frac{k}{2}}b')=-B(b,N^{\ell-\frac{k}{2}}\tilde{g}b')=(-1)^{\ell-\frac{k}{2}+1}B(N^{\ell-\frac{k}{2}}b,\tilde{g}b')
\]
\[
=(-1)^{\ell-\frac{k}{2}+1}B(\tilde{g}b',N^{\ell-\frac{k}{2}}b).
\]
The result is a symmetry (or skew-symmetry, depending on $\ell-\frac{k}{2}+1$)
condition on {[}the matrix entries of{]} $\tilde{g}_{(\ell,\frac{k}{2})}\in\text{Hom}_{\CC}\left(P_{\ell},N^{\frac{k}{2}}P_{\ell}\right).$

In our situation all these constraints (and others, including those
coming from the Lie bracket's status as a Hodge tensor, cf. Remark \ref{rem ?}(a)) are implicit
in, and computed by, the formula on the right-hand side of \eqref{eqn???}.
However, a computation like that above can be valuable for finding
dimensions of boundary components of ordinary period domains without
entering into Hodge structures on Lie algebras.
\end{rem}
We are now ready to consider the (left) quotient by $\Gamma_{\sigma}$
in the context of the iterated fibration above. 
\begin{lem}
\label{lemma a}$\Gamma_{\sigma}\leq M_{B(\sigma)}(\ZZ)$.\end{lem}
\begin{proof}
As a subgroup of $\Gamma,$ $\Gamma_{\sigma}$ is integral, unimodular,
and neat. Since $\Gamma_{\sigma}$ preserves $\sigma$ and its faces,
it acts on $\langle\sigma\rangle$ through a finite group, which by
neatness must be trivial. So $\Gamma_{\sigma}$ fixes $\sigma$ (hence
$N$ and all $P_{k}$) and belongs to $Z(\sigma)(\ZZ)$.

Now the polarizing form $B$ is preserved by the action on $\lm$
of $M(\ZZ)$, \emph{a fortiori} by that of $\Gamma_{\sigma}$. So
$\Gamma_{\sigma}$ preserves $(P_{k},Q_{k})$ for each $k$, and thus
acts through the (finite) integer points of an orthogonal group on
the spaces of Hodge tensors in each $P_{k}^{\otimes a}\otimes\check{P}_{k}^{\otimes b}$.
By neatness, these actions are also trivial,%
\footnote{Lest the reader doubt the assertion that $\Gamma$ can be chosen neat
and of finite index, we remark that by Chevalley's theorem $Gr_{0}^{W}M_{B(\sigma)}$
is ``cut out'' by finitely many Hodge tensors. So once triviality
is checked for finitely many $a$ and $b$, it follows for the rest.%
} and the conclusion follows.
\end{proof}
Clearly $\Gamma_{\sigma}$ preserves $W(\sigma)_{\bullet}$ and commutes
with $e^{\langle\sigma\rangle_{\CC}}$. We therefore have a fibration
tower
\[
\overline{B(\sigma)}\sur\cdots\sur\overline{B(\sigma)}_{(k)}\underset{\bar{\rho}_{\sigma}^{(k)}}{\sur}\overline{B(\sigma)}_{(k-1)}\sur\cdots\sur\overline{B(\sigma)}_{(1)}\underset{\bar{\rho}_{\sigma}^{(1)}}{\sur}\overline{D(\sigma)}
\]
with fibers $\bar{\mathfrak{F}}_{\vft}^{(k)}$, where:
\begin{itemize}
\item $\overline{D(\sigma)}\;:=\; Gr_{0}^{W}\Gamma_{\sigma}\backslash D(\sigma)$
is a standard Mumford-Tate domain quotient (i.e. by a neat subgroup
of $G_{B(\sigma)}(\ZZ)$)
\item $\overline{B(\sigma)}_{(k)}:=\left.\left(\frac{\Gamma_{\sigma}}{W_{-(k+1)}\Gamma_{\sigma}}\right)\right\backslash B(\sigma)_{(k)}$
\item $\bar{\mathfrak{F}}_{\vft}^{(k)}:=\left.Gr_{-k}^{W}\Gamma_{\sigma}\right\backslash \mathfrak{F}_{\vft}^{(k)}$
is isogenous to the quotient of $\mathfrak{F}_{\tilde{\varphi}}^{(k)}$ by $Gr_{-k}^{W(\sigma)}\lm_{\sigma,\ZZ}$,
i.e. to a (generalized if $k>1$) intermediate Jacobian.
\end{itemize}
In particular, the $\bar{\mathfrak{F}}_{\vft}^{(k)}$ are complex
tori for $k=1$ and complex semi-tori for $k>1$, with complex structure
depending on $[\vf_{\text{split}}]\in\overline{D(\sigma)}$.  Write $\bar{\mathfrak{F}}_{\tilde{\varphi}}$ for the ``total'' fiber of $\overline{B(\sigma)} \twoheadrightarrow \overline{D(\sigma)}$.
\begin{rem}
(i) In order that the $Gr_{0}^{W}\Gamma_{\sigma}$ not act on the
fibers, we need to know that it acts on $D(\sigma)$ without fixed
points. While this follows from the fact (cf. Theorem \ref{thm Ku})
that $\Gamma\backslash D_{M,\Sigma}$ is a log-manifold, a direct
argument can be given as follows. If $\vf_{\text{split }}\in D(\sigma)$
is fixed by $\bar{\gamma}\in Gr_{0}^{W}\Gamma_{\sigma}$, then its
$\QQ$-closure $\mathcal{M}$ (as a morphism) commutes with $\bar{\gamma}$.
The orbit $\mathcal{M}(\RR).\vf_{\text{split}}$ is then a MT subdomain
of $D(\sigma)$ fixed pointwise by $\bar{\gamma}$. Since every MT
domain for Hodge structures contains a CM point, $\bar{\gamma}$ fixes
a CM Hodge structure $\hat{\vf}_{\text{split}}$ and thus its irreducible
pure polarized CM HS components. These are all of the form described
in \cite[sec. V]{GGK}, constructed from a generalized CM type $(K,\Theta)$.
The integral points of their MT groups are contained in the elements
of $\mathcal{O}_{K}^{*}$ with modulus $1$ under all complex embeddings.
By a theorem of Kronecker \cite{Kr}, these are roots of unity. The neatness
assumption on $\Gamma$ (hence on $Gr_{0}^{W}\Gamma_{\sigma}$) now
implies that $\bar{\gamma}$ acts trivially on $\oplus_{\ell}Gr_{\ell}^{W(\sigma)}\lm$
and so is itself trivial.

(ii) That the $Gr_{-k}^{W}\Gamma_{\sigma}$ do not have fixed points
on $\bar{\mathfrak{F}}_{\vft}^{(k)}$ is a consequence of Lemma \ref{lemma a}.
\end{rem}
We can immediately characterize an important special case.
\begin{prop}
\label{prop hatch} (i) If $\lm_{\sigma}=\langle\sigma\rangle$ and
$\lg_{B(\sigma)}\subset I_{0}^{(-1,1)}+I_{0}^{(0,0)}+I_{0}^{(1,-1)}$,
then $\overline{B(\sigma)}$ is an irreducible component of a Shimura
variety.

(ii) If $W(\sigma)_{-2}\lm_{\sigma}=\langle\sigma\rangle$, $Gr_{-1}^{W(\sigma)}\lm_{\sigma}$
is of type $(-1,0)+(0,-1)$, $\lg_{B(\sigma)}$ is of type $(-1,1)+(0,0)+(1,-1)$,
and the map $\lg_{B(\sigma)}\to\text{End}(Gr_{-1}^{W(\sigma)}\lm_{\sigma})$
is injective, then $\overline{B(\sigma)}$ is the canonical abelian
fibration over (an irreducible component of) a Shimura variety of
Hodge type.%
\footnote{cf. \cite[Def. 7.1]{Mi} for the definition of the canonical abelian
fibration.%
}\end{prop}
\begin{proof}
That $\lg_{B(\sigma)}$ be of type $(-1,1)+(0,0)+(1,-1)$ is the criterion
for $D(\sigma)$ to be a Hermitian symmetric space (cf. \cite[sec. 1.D]{Ke}
or \cite{Mi}), and $Gr_{0}^{W}\Gamma_{\sigma}\subset G_{B(\sigma)}(\RR)^{+}$
(being of finite index in $G_{B(\sigma)}(\ZZ)$) is of congruence
type. Otherwise, this follows from the theory above.
\end{proof}
In either case, $B(\sigma)$ is a quasi-projective algebraic variety
with a model over $\bar{\QQ}$ \cite[sec. 5.B]{Ke}, \cite[Ch. 12-14]{Mi}.
\begin{rem}
In order to check the conditions of Proposition \ref{prop hatch}
and later results, it may be helpful to note that
\[
W_{-k}\lm_{\sigma}=\left(\cap_{j}\ker(\ad N_{j})\right)\cap\text{im}\left\{ (\ad N)^{k}\right\} .
\]

\end{rem}
Now it is natural to expect that the constraints required to make
$\overline{B(\sigma)}$ a CM abelian variety are even more stringent
than those in Proposition \ref{prop hatch}. As we shall see, this
is not the case. We first dig a bit further into the boundary component
structure.
\begin{lem}
\label{lemma b}The following are equivalent:

(i) the base $\overline{D(\sigma)}$ is a point;

(ii) the $Gr_{-k}^{W(\sigma)}\lz(\sigma)$ (and thus $P_{k},\, Gr_{k}^{W(\sigma)}\lm$)
are polarized CM Hodge structures, constant in $\vft\in\tilde{B}(\sigma)$;

(iii) $G_{B(\sigma)}$ \emph{{[}}resp. $\lg_{B(\sigma)}$\emph{{]}
is abelian;}

(iv) in the decomposition of $G_{\sigma}$ \emph{{[}}resp. $\lg_{\sigma}$\emph{{]}}
into $\QQ$-simple and abelian factors, the projection of $(\vf_{0})_{\text{split }}$
\emph{{[}}resp.%
\footnote{Here $\phi\in\lg_{\sigma,\RR}$ is the ``tangent'' to $(\phi_{0})_{\text{split}}$
is in $\S\S3-4$.%
} $\phi$\emph{{]}} onto the simple factors is trivial.\end{lem}
\begin{proof}
For $(i)\,\Longleftrightarrow\,(ii)\,\Longleftrightarrow\,(iii)$,
see \cite[Ch. V]{GGK}. $(iii)\,\iff\,(iv)$ is obvious, and is
essentially Prop. (IV.A.9) in \cite{GGK}.
\end{proof}
Recall the notation $\Lambda^{-1,-1}(\mathfrak{m}_{\CC}):=\oplus_{p<0,q<0}\mathfrak{m}^{p,q}$.
\begin{lem}
\label{lemma c}The following are equivalent:

(i) the fibers $\bar{\mathfrak{F}}_{\vft}$ are compact;

(ii) $B^{\RR}(\sigma)=B(\sigma)$;

(iii) $\Lambda^{-1,-1}(\lm_{\CC})=\langle\sigma\rangle_{\CC}$.\end{lem}
\begin{proof}
$\underline{(i)\Leftrightarrow(ii)}$: The point is that $\overline{B^{\RR}(\sigma)}$
also admits an iterated fibration, with compact $\bar{\mathfrak{F}}_{\vft}^{(k),\RR}\simeq\frac{Gr_{-k}^{W(\sigma)}\left((\lm_{\sigma}/\langle\sigma\rangle)_{\RR}\right)}{Gr_{-k}^{W(\sigma)}\left((\lm_{\sigma}/\langle\sigma\rangle)_{\ZZ}\right)}.$
Anything ``larger'' is not compact.

$\underline{(ii)\Leftrightarrow(iii)}$: This is because $\mathfrak{F}_{\vft}^{(k),\RR}=\mathfrak{F}_{\vft}^{(k)}$
if and only if 
\[
Gr_{-k}^{W(\sigma)}\left((\lm_{\sigma}/\langle\sigma\rangle)_{\RR}\right)=Gr_{-k}^{W(\sigma)}\left((\lm_{\sigma}/\langle\sigma\rangle)_{\CC}\right)/F^{0},
\]
i.e. $Gr^{W(\sigma)}_{-k}\left( (\lm_{\sigma}\slash\langle\sigma\rangle)_{\CC}\right)$ has no $(-k+1,-1)+\cdots +(-1,-k+1)$ part.
\end{proof}
\begin{rem}
Under the equivalent conditions of Lemma \ref{lemma c}, the fibers
are complex tori if, in addition, $\Gamma_{\sigma}$ is abelian. This
is ensured by assuming that the weight filtration on $\lm_{\sigma}/\langle\sigma\rangle$
is \emph{short}: for every $k$, if $Gr_{-k}^{W(\sigma)}(\lm_{\sigma}/\langle\sigma\rangle)\neq\{0\}$,
then $W(\sigma)_{-2k}(\lm_{\sigma}/\langle\sigma\rangle)=\{0\}$.
\end{rem}
Now assume $\overline{D(\sigma)}$ is a point, with single fiber $\bar{\mathfrak{F}}$
($=\overline{B(\sigma)}$) a complex torus. Writing $V_{k}$ for the
weight $(-k)$ Hodge structure $Gr_{-k}^{W(\sigma)}(\lm_{\sigma}/\langle\sigma\rangle)$,
we know that $\bar{\mathfrak{F}}$ is isogenous to
\[
\times_{k>0}\frac{V_{k,\CC}}{F^{0}(V_{k,\CC})+V_{k,\ZZ}}\;=:\;\times_{k>0}J(V_{k}).
\]
The $V_{k}$ are polarizable CM Hodge structures, and therefore each
have a decomposition of the form \begin{equation}\label{eqn (*)}\bigoplus_{i}\left(V_{(K_{i},\Theta_{i})}^{-k}\right)^{\oplus m_{i}}\end{equation}
where the $\{K_{i}\}$ are CM fields and $\Theta_{i}$ $(-k)$-orientations.
Referring to \cite[Ch. V]{GGK} for a fuller discussion, these are
partitions of the set of embeddings $\text{Hom}(K_{i},\CC)$ into
$(p,q)$-subsets ($p+q=-k$) in a manner consistent with complex conjugation.
This induces a decomposition of $K_{i}\otimes_{\QQ}\CC$ into $(p,q)$-subspaces,
putting a (necessarily polarizable, weight $(-k)$) HS on the $\QQ$-vector
space $K_{i}$.

By Lemma \ref{lemma c}(iii), in our present case either $(p,q)$
or its ``conjugate'' $(q,p)$ is always in $\ZZ_{\geq0}\times\ZZ$.
For each term in \eqref{eqn (*)}, define a CM-type $\Theta_{i}'$
by replacing the pairs in $\ZZ_{\geq0}\times\ZZ$ by $(0,-1)$ and
their ``conjugates'' by $(-1,0)$. The resulting weight $(-1)$
HS $V_{(K_{i},\Theta_{i}')}^{-1}$ is well-known to be polarizable,
with Jacobian a CM abelian variety. So while the HS
\[
V_{k}'\;:=\;\bigoplus_{i}\left(V_{(K_{i},\Theta_{i})}^{-1}\right)^{\oplus m_{i}}
\]
is of type $(-1,0)+(0,-1)$, it shares the same underlying integral
structure and $F^{0}$ as $V_{k}$, and we conclude that 
\[
J(V_{k})\cong J(V_{k}').
\]
This proves that $J(V_{k})$ hence $\bar{\mathfrak{F}}$ is a CM abelian
variety.%
\footnote{Note that we consider a point to be a zero-dimensional CM abelian
variety.%
}

Here, then, is the strongest result that can be stated without explicit
knowledge of $\phi$:
\begin{thm}
\label{thm structure}Assume $\overline{B(\sigma)}\subset(\Gamma\backslash D_{M,\Sigma})$
is nonempty, and that 

(a) $W(\sigma)_{\bullet}(\lm_{\sigma}/\langle\sigma\rangle)$ is short,

(b) $\Lambda^{-1,-1}(\lm_{\sigma}/\langle\sigma\rangle)=\{0\}$, and

(c) $\lz(\sigma)/\lm_{\sigma}$ is abelian.\\
Then $\overline{B(\sigma)}$ is a CM abelian variety (and has a model
over $\bar{\QQ}$).
\end{thm}
A useful shortcut is provided by the
\begin{cor}
\label{cor a}$\overline{B(\sigma)}\neq\emptyset$ is a CM abelian
variety if $W(\sigma)_{-2}\lm_{\sigma}=\langle\sigma\rangle$ and
$\lz(\sigma)/\lm_{\sigma}$ is abelian.
\end{cor}
In the rank one case, using Remark \ref{rem ku}, this yields the
most practical criterion:
\begin{cor}
\label{cor b}Suppose $D_{M}$ parametrizes (weight zero, $B$-polarized)
HS on $\lm$ of level $2\ell$. A nilpotent $N\in\lm_{\QQ}$ produces
a nonempty boundary component $B(N)$ iff:\vspace{2mm}

$\bullet$ \emph{$(\ad N)^{2\ell+1}=0$;}

$\bullet$ $N(F^{\bullet})\subset F^{\bullet-1}$, for some $F^{\bullet}\in\check{D}_{M}$;
and

$\bullet$ \emph{$\left(P_{k}(N),Gr_{k}^{W(N)}F^{\bullet},B(\cdot,(\ad N)^{k}\cdot)\right)$}
is a polarized HS for each $k$.\vspace{2mm}\\
If in addition \vspace{2mm}

$\bullet$\emph{ $\ker(\mathrm{ad}N)\cap\text{im}\{(\mathrm{ad}N)^{2}\}=\langle N\rangle$},
and 

$\bullet$ $[\ker(\mathrm{ad}N),\ker(\mathrm{ad}N)]\subseteq\mathrm{im}(\mathrm{ad}N)$,\vspace{2mm}\\
then its quotient $\overline{B(N)}$ is a CM abelian variety, of dimension\linebreak
$\frac{1}{2}\dim\left(Gr_{-1}^{W(N)}\lm\right)$.\end{cor}
\begin{rem}
\label{rem7.12} For purposes of directly generalizing the arguments
in \cite{C}, one requires not only that a boundary component quotient
be arithmetic, but to have $M$ of Hermitian type and $(\ad N)^{3}=0$.
The reader may have noticed that we ignored the (easy) case when $\overline{B(\sigma)}\cong(\mathbb{C}^{*})^{m}$,
although this is obviously defined over $\bar{\mathbb{Q}}$ and we
shall encounter such examples in $\S8$. This is because to have a
$(\mathbb{C}^{*})^{m}$ boundary component with $(\ad N)^{3}=0$,
$D_{M}$ would have to parametrize HS on $\mathfrak{m}$ of level
$2$; that is, it would already be classical.
\end{rem}

\section{Examples}

The authors expect, in a future work, to use Corollary \ref{cor b}
(and other results above) to treat systematically the MT domains for
simple Lie groups classified in \cite[sec. 4]{GGK}. Here we shall
restrict ourselves to a quick analysis of a few examples which motivated
this paper.\footnote{A more systematic treatment of one \emph{class} of examples has now been carried out in \cite[App. A]{KR}} The first is quite classical, and the first three just
compute (but very efficiently) the boundary components of some period
domains.

A slight difference to how the results are stated above,
is that we start with a Hodge representation\footnote{not necessarily of weight zero (recall the passage from $V$ to $\lm$ in the Introduction)} $V$ of $M$ on a smaller
vector space than $\lm$. This can be more convenient for checking
that $B(N)$ or $B(\sigma)$%
\footnote{The rank of $B(\sigma)$ ($:=\dim\langle\sigma\rangle$) is one in
all but the first example.%
} is nonempty, e.g. by producing an $F^{\bullet}\in \check{D}$ such that $(F^{\bullet},W(N)_{\bullet})$
is an $N$-polarized MHS. Having done this, we then content ourselves
with $I^{p,q}$ diagrams (\# of dots at the $(p,q)$-spot $=$ $\dim(I^{p,q})$)
for the LMHS on $V$ resp. $\lm$ that are parametrized by each boundary
component. These are easy to produce if one knows the $h^{p,q}$'s
for $\lm$ (which can be looked up in \cite[Ch. IV]{GGK} for simple
Lie algebras) corresponding to those chosen for $V$, and if one computes
the ranks of kernel and image of $(\ad N)^{k}$ acting on $\lm$ (left
to the reader). The structure theory of $M_{B(N)}$ and $\overline{B(N)}$
can then be read off these pictures from the results above.
\begin{example}
We begin with the familiar example where $D$ is the Siegel upper half-space, with
$M=Sp_{4}$, $\dim(V)=4$, $h^{1,0}=h^{0,1}=2$,\small 
\[
Q=\left(\begin{array}{cccc}
 &  & 1\\
 &  &  & 1\\
-1\\
 & -1
\end{array}\right),\; N_{1}=\left(\begin{array}{cccc}
 &  & 0\\
 &  &  & 0\\
1\\
 & 0
\end{array}\right),\; N_{2}=\left(\begin{array}{cccc}
 &  & 0\\
 &  &  & 0\\
0\\
 & 1
\end{array}\right),
\]
\normalsize 
\[
\text{and}\;N=N_{1}+N_{2},\;\sigma=\QQ_{\geq0}\langle N_{1},N_{2}\rangle .
\]
Since $\tilde{B}(sigma)\supset \tilde{B}(N_1)\supset D$ (cf. \cite{Ca2}), for $F^{\bullet}$ we may take anything in $D$. The standard $I^{p,q}$'s
(for $V$) are then: $$\includegraphics[scale=0.65]{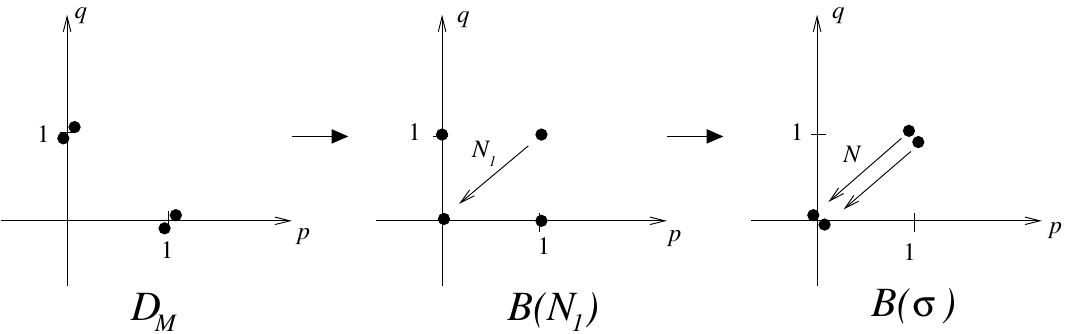}$$ while
the adjoint $I^{p,q}$'s are depicted in $$\includegraphics[scale=0.65]{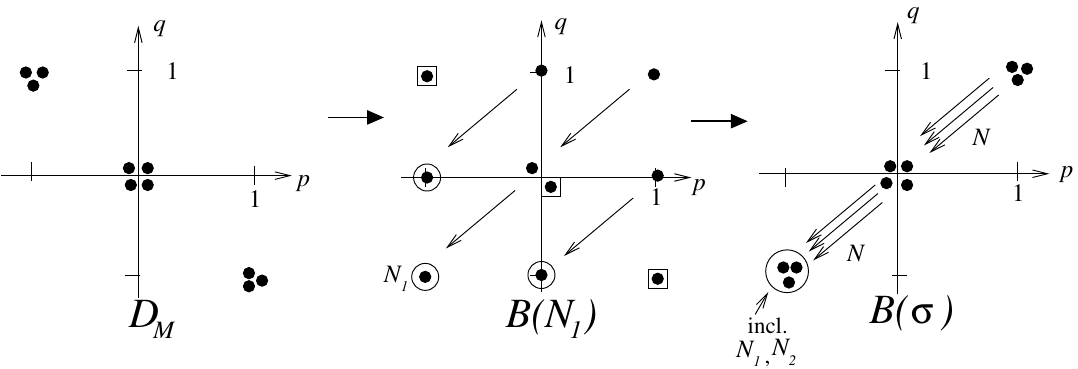}$$
with generators of $\lm_{\sigma}$ circled and generators of $\lz(\sigma)/\lm_{\sigma}\cong P_{0}$
boxed. (For $B(\sigma)$, $P_{0}=\{0\}$.)

Turning to boundary component structure, for $B(N_{1})$ we have 
\[
Gr_{0}^{W}M_{B(N_{1})}\cong SL_{2}
\]
over $\RR$ because the boxed stuff is $\sl$ and $\phi$ cannot be
$0$ (since the Hodge structures on the $P_{k}$ are nontrivial);
and 
\[
\left.e^{\langle N_{1}\rangle}\right\backslash W_{-1}M_{B(N_{1})}\,\cong\,\mathbb{G}_{a}^{\times2}.
\]
Consequently, $\overline{B(N_{1})}$ $\simeq$ a (noncompact) elliptic
modular surface. Geometrically, given a degeneration of genus 2 curves
$$\includegraphics[scale=0.8]{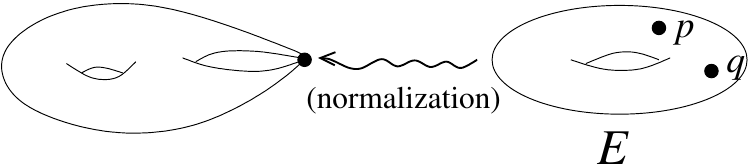}$$ $\overline{B(N_{1})}$ records
the isomorphism class (with $\Gamma$-level structure) of $E$ and
$AJ([p]-[q])\in J(E)$. 

For $B(\sigma)$, 
\[
\left\{ \begin{array}{c}
Gr_{0}^{W}M_{B(\sigma)}\text{ is trivial}\;\;\;\\
\left.e^{\langle\sigma\rangle}\right\backslash W_{-1}M_{B(\sigma)}\cong\mathbb{G}_{a}
\end{array}\right.
\]
implies $\overline{B(\sigma)}\cong\CC^{*}$. Given a degeneration
of genus 2 curves $$\includegraphics[scale=0.8]{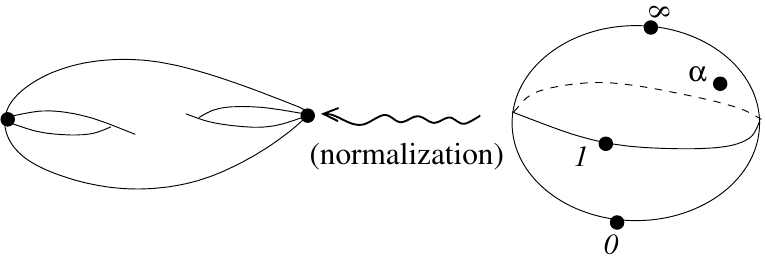}$$ the $\CC^{*}$
records $\alpha$ (i.e. the cross-ratio) \cite{CCK,Ca2}. The smooth
loci of the N\'eron $n$-gons in a minimal smooth compactification
of $\overline{B(N_{1})}$ can be thought of as unions of $\overline{B(\sigma)}$'s.

We can also consider $B(N)$, which has $I^{p,q}$ diagrams $$\includegraphics[scale=0.8]{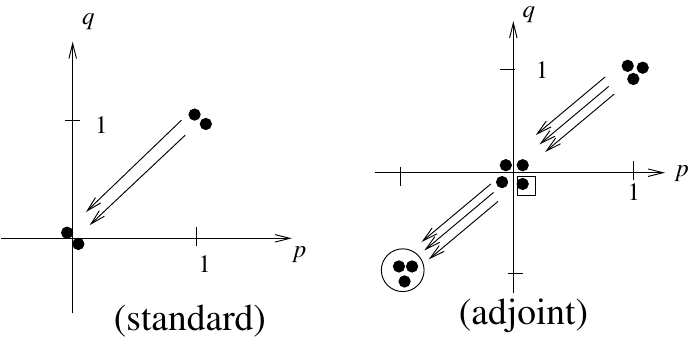}$$
While $P_{0}\neq\{0\}$, $\phi$ is trivial since the MHS is Hodge-Tate,
and so $Gr_{0}^{W}M_{B(N)}$ is once again trivial; on the other hand
\[
W_{-1}M_{B(N)}\cong\mathbb{G}_{a}^{\times3}.
\]
So in this example, $M_{B(N)}=M_{N}\subsetneq Z(N)$, and we have
$\overline{B(N)}\cong(\CC^{*})^{\times2}$. (In every other example
we consider here, it turns out that $M_{B(N)}=Z(N)$ {[}or $Z(\sigma)${]}.
This is because all the $\lg_{N}$'s which turn up are trivial, 1-dimensional,
or {[}over $\RR${]} $\sl$.)
\end{example}
\begin{example}
We now demonstrate the ease with which this method recovers the structural results on boundary components for the ``mirror quintic period domain'' in \cite{GGK4}.  Let $M=Sp_{4}$, $\dim(V)=4$, $h^{3,0}=h^{2,1}=h^{1,2}=h^{0,3}=1$, with rational basis $\gamma_3,\gamma_2,\gamma_1,\gamma_0$ in which\small
\[
Q=\left(\begin{array}{cccc}
 &  &  & 1\\
 &  & 1\\
 & -1\\
-1
\end{array}\right).
\]
\normalsize By \cite{GGK2}, up to a symplectic change of basis the
only possibilities for $N$ are (writing $A$ for a $2\times 2$ matrix, and $a>0>b$) \small
\[
N_{1}=\left(\begin{array}{cccc}
\\
a\\
c & b\\
d & c & -a & {}
\end{array}\right),\; N_{2}=\left(\begin{array}{cccc}
\\
\\
\\
a & {}\;{} & {}\;{} & {}\;{}
\end{array}\right),\;\begin{array}{c}
N_{3}=\left(\begin{array}{cc}
\\
A & {}\;{}
\end{array}\right)\\
\text{with }{}^{t}A=A>0
\end{array}.
\]
\normalsize To check that $B(N_k)\neq \emptyset$ in each case, take for $F^3\subset F^2\subset F^1$:
\begin{flalign*}
\;\; & \underline{k=1}:\; \langle \gamma_3 + \frac{d}{2a}\gamma_1\rangle\subset F^3+\langle\gamma_2+\frac{c}{a}\gamma_1 + \frac{d}{2a}\gamma_0\rangle \subset F^2+\langle \gamma_1\rangle&\\
\;\; & \underline{k=2}:\; \langle \gamma_2 - \sqrt{-1}\gamma_1\rangle \subset F^3 + \langle \gamma_3\rangle \subset F^2+\langle \gamma_0\rangle &\\
\;\; & \underline{k=3}:\; \text{see \cite[(I.C.19)]{GGK4}}.
\end{flalign*} 
(It is a general feature of the nonclassical setting that ``any $F^{\bullet}\in D$'' will not do.)

Now, for the period domain itself, we have $I^{p,q}$ diagrams $$\includegraphics[scale=0.8]{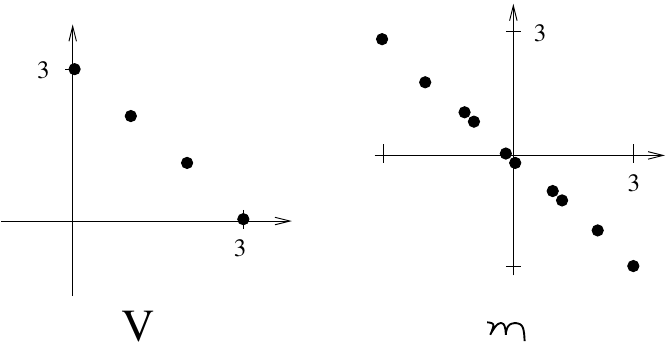}$$
Next, the LMHS in $B(N_{1})$ have $I^{p,q}$'s $$\includegraphics[scale=0.8]{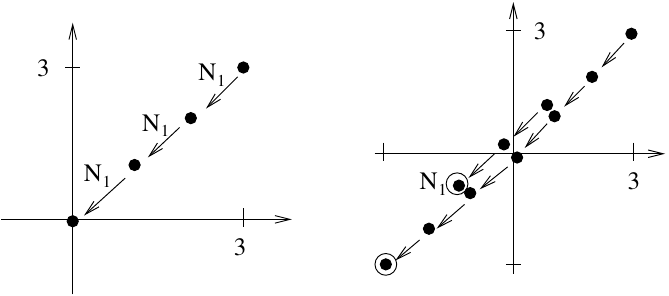}$$
and $M_{B(N_{1})}\cong M_{N_{1}}\cong\mathbb{G}_{a}^{\times2}$, $\overline{B(N_{1})}\cong\CC^{*}$.

Turning to $B(N_{2})$, one has $$\includegraphics[scale=0.8]{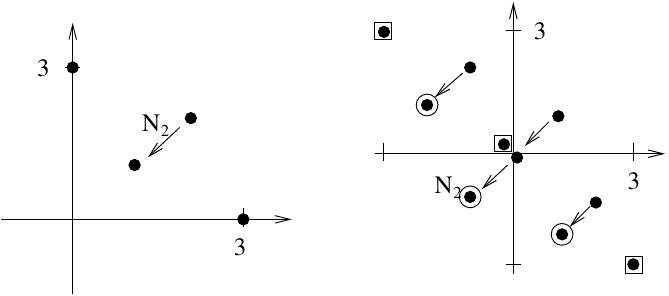}$$
and from 
\[
\left\{ \begin{array}{c}
Gr_{0}^{W}M_{B(N_{2})}\cong SL_{2}\;\text{(over }\RR\text{)}\\
e^{\langle N_{2}\rangle}\backslash W_{-1}M_{B(N_{2})}\cong\mathbb{G}_{a}^{\times2}
\end{array}\right.
\]
we conclude that $\overline{B(N_{2})}\;\simeq$ a noncompact%
\footnote{Of course, it admits an algebraic compactification, but this time
the added points do not correspond to LMHS (essentially by transversality
considerations).%
} elliptic modular surface (in spite of the fact that it does not fall
under Proposition \ref{prop hatch}).

Finally, for $B(N_{3})$, the $I^{p,q}$'s are $$\includegraphics[scale=0.8]{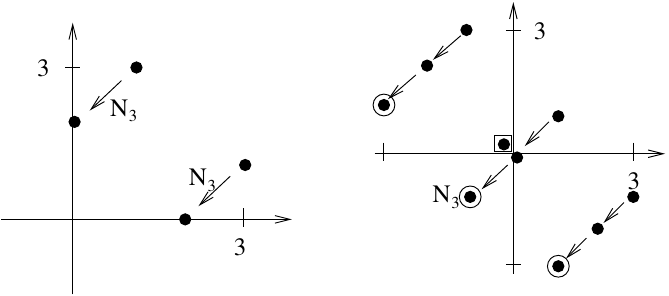}$$
so
\[
\left\{ \begin{array}{c}
Gr_{0}^{W}M_{B(N_{3})}\cong\mathbb{G}_{m}\;\;\;\;\;\\
e^{\langle N_{2}\rangle}\backslash W_{-1}M_{B(N_{3})}\cong\mathbb{G}_{a}^{\times3}
\end{array}\right.
\]
from which (by Theorem \ref{thm structure}) $\overline{B(N_{3})}$
is a CM elliptic curve.
\end{example}
\begin{example}
This is related to a running example in \cite[Chap. 0]{KU} and uncovers a structure not noticed there. We take $M=SO(4,1)$, $\dim(V)=5$, $h^{2,0}=h^{0,2}=2$, $h^{1,1}=1$,\small $$
Q=\left( \begin{array}{c|cccc} 1  \\ \hline & -1 \\ & & -1 \\ & & & -1 \\ & & & & -1 \end{array} \right)
\; , \;\;
N=\left( \begin{array}{c|cccc} & 1 \\ \hline 1 & & 1 \\ & -1 \\ & & & 0 \\ & & & & 0 \end{array} \right) . 
$$\normalsize Writing $u,v_{1},v_{2},v_{3},v_{4}$ for the basis of
$V$, we have $W_{0}V=\langle u-v_{2}\rangle$ and $W_{2}V=\langle v_{1},v_{3},v_{4},u-v_{2}\rangle$.
For an $F^{\bullet}\in\check{D}_{M}$ for which $e^{\langle N\rangle}F^{\bullet}$
is a nilpotent orbit, we can take $F^{2}=\langle u+v_{2},v_{3}+iv_{4}\rangle$
and $F^{1}=\langle v_{1},u+v_{2},v_{3}+iv_{4}\rangle$, which satisfies
$Q(F^{2},F^{1})=0$ and polarizes $Gr_{4}^{W}$ and $Gr_{2}^{W}\ker N$.

The $I^{p,q}$'s for the standard representation are $$\includegraphics[scale=0.8]{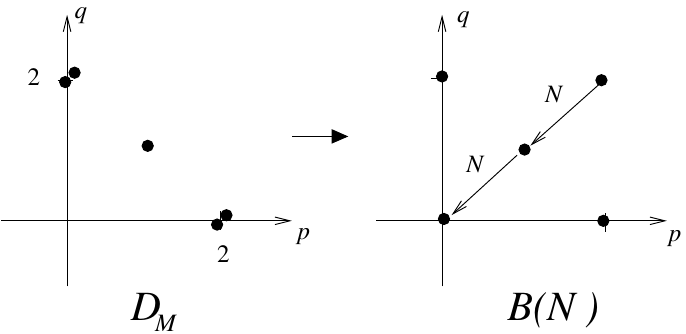}$$
From the adjoint representation we instead obtain $$\includegraphics[scale=0.8]{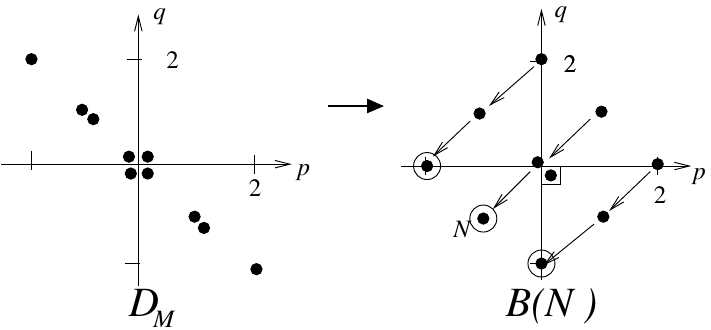}$$
As far as $I^{p,q}$'s go, $B(N)$ is essentially the only type of
boundary component for the period domain $D_{M}$. We have 
\[
\left\{ \begin{array}{c}
Gr_{0}^{W}M_{B(N)}\cong\mathbb{G}_{m}\;\;\;\;\;\;\;\;\;\;\\
W_{-1}M_{B(N)}\cong M_{N}\cong\mathbb{G}_{a}^{\times3}
\end{array}\right.
\]
and hence that $\overline{B(N)}$ is a CM elliptic curve.
\end{example}
\begin{example}\label{exa:(Carayol's-example-)}
Next is an analysis of Carayol's example \cite{C}, which inspired the writing of this paper.
This takes a bit more setting up, and (unlike the other examples)
involves distinct $'G$ and $M={}'G^{\ad}$: in this case,
\[
\left\{ \begin{array}{c}
'G(\RR)\cong U(2,1)\;\;\;\\
M(\RR)\cong SU(2,1)^{\ad}
\end{array}.\right.
\]
Let $V$ be a 6-dimensional vector space, $Q:V\times V\to\QQ$ an
alternating nondegenerate bilinear form, and $d$ be a square-free
integer. We fix a ring homomorphism 
\[
\mu:\mathbb{F}:=\QQ(\sqrt{-d})\hookrightarrow\text{End}_{\QQ}(V)
\]
such that in the decomposition $V_{\mathbb{F}}=V_{+}\oplus V_{-}$
into eigenspaces for the conjugate complex embeddings of $\mathbb{F}$,
$V_{+}$ is $Q$-isotropic. Consider a $Q$-polarized Hodge structure%
\footnote{The notation means in particular that $\mu(\mathbb{F})\subset\text{End}_{'\vf}(V)$.%
} 
\[
'\vf:\mathbb{U}\to\text{Aut}(V,Q,\mu)
\]
with $h_{+}^{3,0}=h_{+}^{2,1}=h_{+}^{1,2}=1$ ($\implies$ $h^{3,0}=1$
and $h^{2,1}=2$). Finally, set $\mathbf{h}(v,w):=-\sqrt{-d}Q(v,\bar{w})$
and
\[
'G:=\text{Aut}(V,Q,\mu)\cong\text{Aut}(V_{+},\mathbf{h}).
\]
Then, as in $\S1$, 
\[
'G(\RR).{}'\vf\cong M(\RR).(\Ad\circ{}'\vf)=D_{M}
\]
and $'G$ is the Mumford-Tate group of a generic Hodge structure (on
$V$) in the left-hand orbit.

More precisely, we may choose $Q$ and an $\mathbb{F}$-basis $\{\gamma_{1},\gamma_{2},\gamma_{3}\} \subseteq V_{+}$
so that
\[
[\mathbf{h}]_{\gamma}=\left(\begin{array}{ccc}
 &  & -1\\
 & 1\\
-1
\end{array}\right),
\]
and then $\lm_{\RR}$ identifies with elements of $\text{End}(V_{+,\CC})$
of the form 
\[
\left(\begin{array}{ccc}
A & B & C\\
D & E & \bar{B}\\
G & \bar{D} & -\bar{A}
\end{array}\right),\;\left\{ \begin{array}{c}
C,E,G\in i\RR\\
A,B,D\in\CC
\end{array}\right..
\]
(We can define a $'\vf$ by $F^{3}V_{+}=\langle\gamma_{2}\rangle$,
$F^{2}V_{+}=\langle\gamma_{1}+\gamma_{3},\gamma_{2}\rangle$; this
won't necessarily be the $F^{\bullet}$ giving the nilpotent orbits
below, which are written out in \cite[$\S$2.3]{C}.) Up to conjugation by $M(\RR)$, and under this identification,
the only nilpotents in $\lm$ giving nilpotent orbits are
\[
N_{1}=\left(\begin{array}{ccc}
0 & \alpha & ia\\
0 & 0 & \bar{\alpha}\\
0 & 0 & 0
\end{array}\right)\;,\;\; N_{2}^{\pm}=\left(\begin{array}{ccc}
0 & 0 & \pm ib\\
0 & 0 & 0\\
0 & 0 & 0
\end{array}\right)
\]
where $\alpha\in\mathbb{F}$, $a\in\QQ$, and $b\in\QQ_{+}$. We refer
to \cite[$\S$2.2]{C} for the proof.

Turning to $I^{p,q}$'s, the generic element of $D_{M}$ has $$\includegraphics[scale=0.8]{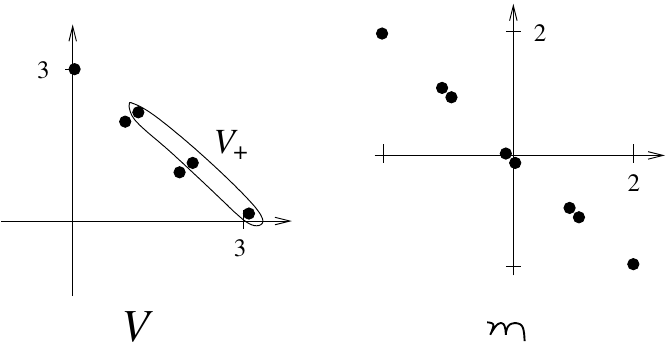}$$
For the three types of boundary components, we have $$\includegraphics[scale=0.8]{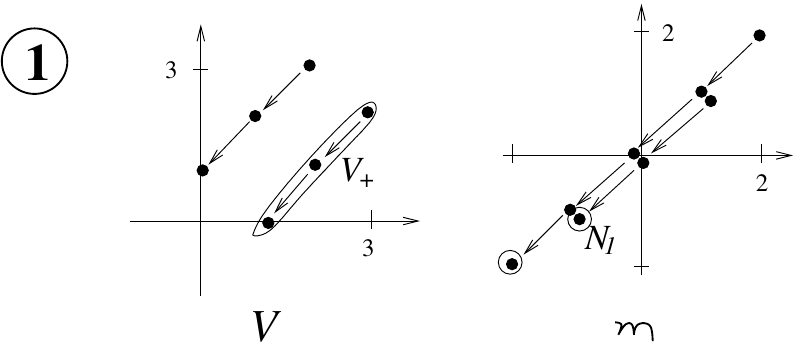}$$
which gives $M_{B(N_{1})}\cong M_{N_{1}}\cong\mathbb{G}_{a}^{\times2}$,
$\overline{B(N_{1})}\cong\mathbb{C}^{*}$; $$\includegraphics[scale=0.8]{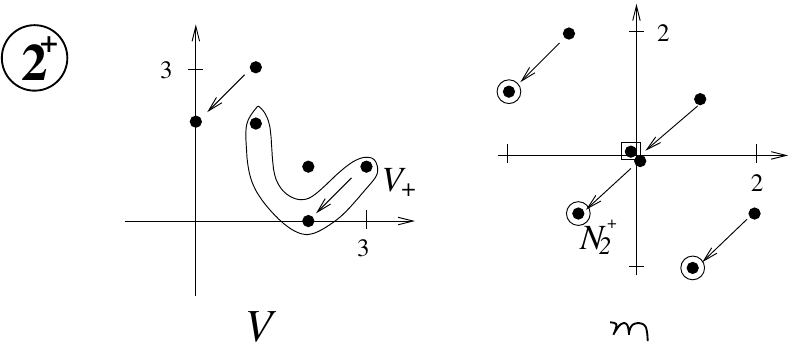}$$
which yields 
\[
\left\{ \begin{array}{c}
Gr_{0}^{W}M_{B(N_{2}^{+})}\cong\mathbb{G}_{m}\;\;\;\;\;\;\;\;\\
e^{\langle N_{2}^{+}\rangle}\backslash W_{-1}M_{B(N_{2}^{+})}\cong\mathbb{G}_{a}^{\times2}
\end{array}\right.
\]
 making $\overline{B(N_{2}^{+})}$ a CM elliptic curve; and $$\includegraphics[scale=0.8]{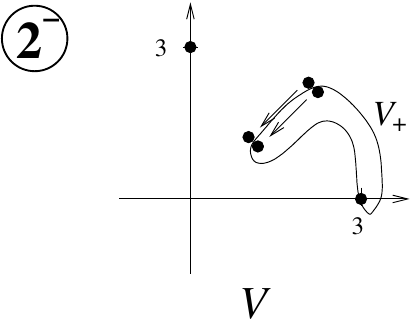}$$
with everything else the same as in the $N_{2}^{+}$ case.
\end{example}
\begin{example}\label{G2 ex}
Our last example treats an exceptional group, and sets the stage for the analysis of fundamental adjoint varieties in \cite{KR}; it should also be related to degenerations of the surfaces studied in \cite{Ka2}.  Put $M=G_{2}$, $\dim V=7$, $h^{2,0}=h^{0,2}=2$, $h^{1,1}=3$,\small 
$$
Q=\left( \begin{array}{c|c} I_3 \\ \hline & -I_4 \end{array} \right)\; ,\;\;
N=\left( \begin{array}{ccc|cccc} & & -1 & 1 \\ & 0 & & & 0 \\ 1 & & & & & -1 & 0 \\ \hline 1 & & & & & -1 & 0 \\ & 0 & & & 0 \\ & & -1 & 1 \\ & & 0 & 0 \end{array} \right)
$$\normalsize (note $N^{2}=0$). Inside $SO(3,4)$, $G_{2}$ is cut
out by preserving a certain $3$-tensor. In the present $u_{1},u_{2},u_{3},v_{1},v_{2},v_{3},v_{4}$
basis, $\lg_{2}$ consists of matrices of the form\small  $$
\left( \begin{array}{ccc|cccc} 0 & f-a & -b-e & A & B & C & H-F \\ a-f & 0 & d-c & D & E & F & C-G \\ b+e & c-d & 0 & G & H & -A-E & D-B \\ \hline A & D & G & 0 & -a & -b & -d \\ B & E & H & a & 0 & -c & -e \\ C & F & -A-E & b & c & 0 & -f \\ H-F & C-G & D-B & d & e & f & 0 \end{array} \right).
$$\normalsize Clearly $N$ is of this form.

On $V$, $N$ induces the weight filtration 
\[
\left\{ \begin{array}{ccc}
(W_{3} & = & V)\;\;\;\;\;\;\;\;\;\;\;\;\;\;\;\;\;\;\;\;\;\;\;\;\;\;\;\;\;\;\;\;\;\;\;\;\\
W_{2} & = & \langle u_{2},v_{2},v_{4},u_{1}+v_{3},u_{3}+v_{1}\rangle\\
W_{1} & = & \langle u_{1}+v_{3},u_{3}+v_{1}\rangle\;\;\;\;\;\;\;\;\;\;\;\;\;\;\;\;\\
(W_{0} & = & \{0\})\;\;\;\;\;\;\;\;\;\;\;\;\;\;\;\;\;\;\;\;\;\;\;\;\;\;\;\;\;\;\;\;\;\;\;\;
\end{array}\right.
\]
and it is easy to see that
\[
\left\{ \begin{array}{ccc}
F^{2} & := & \langle v_{1}+iv_{3},v_{2}-iv_{4}\rangle\;\;\;\;\;\;\;\;\;\;\;\;\;\;\;\;\\
F^{1} & := & \langle u_{1},u_{2},u_{3},v_{1}+iv_{3},v_{2}-iv_{4}\rangle
\end{array}\right.
\]
produces an $N$-polarized MHS and belongs to $\check{D}_{M}$. In
fact, this $F^{\bullet}$ defines a HS $\vf$ on $V$ polarized by
$Q$, and the differential of $\vf$ is a multiple (writing $\Delta=\text{diag}\{1,-1\}$)
of $$
\left( \begin{array}{c|cc} 0 & 0 & 0 \\ \hline 0 & 0 & \Delta \\ 0 & -\Delta & 0 \end{array} \right)
$$which is clearly in $\lg_{2}$. The $I^{p,q}$ pictures for $V$ are
$$\includegraphics[scale=0.8]{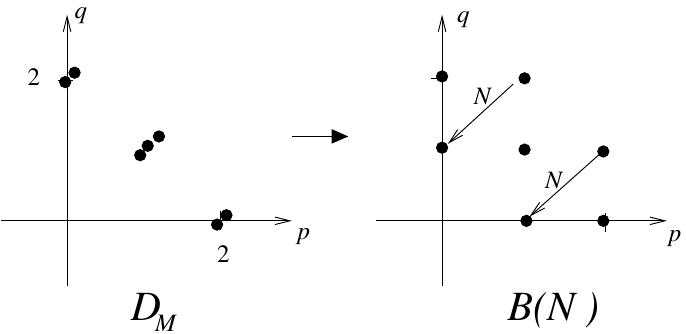}$$ As for the adjoint representation,
we have Hodge numbers $h^{-2,2}=1$, $h^{-1,1}=4=h^{0,0}$ and $I^{p,q}$
diagrams $$\includegraphics[scale=0.8]{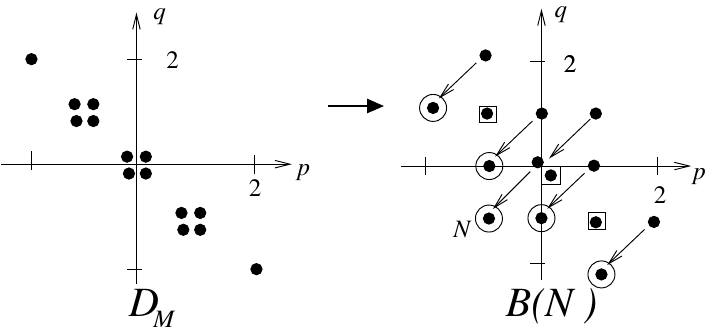}$$ Since $\phi$ is
nontrivial, $Gr_{0}^{W}M_{B(N)}\cong SL_{2}$ (over $\RR$); moreover,
$e^{\langle N\rangle}\backslash W_{-1}M_{B(N)}\cong\mathbb{G}_{a}^{\times4}$,
and we find that $\overline{B(N)}$ is isomorphic to a family of (compact)
complex 2-tori over a modular curve.

As far as $I^{p,q}$ types are concerned, it turns out that there
are two further types of rank one boundary components for this $G_{2}$-domain.
On $V$, the corresponding diagrams are $$\includegraphics[scale=0.8]{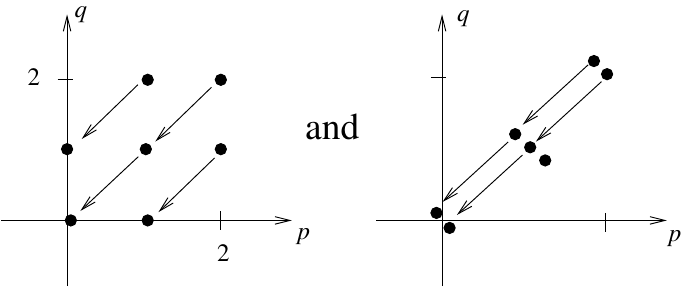}$$
We won't pursue a systematic treatment of this here.
\end{example}

\section{A $G_{2}$-variation of Hodge structure}

In \cite[Ch. IV]{GGK} it was shown that, up to Tate twists, there
were three possible gap-free collections of Hodge numbers for effective
rank 7 Hodge structures with Mumford-Tate group $G_{2}$:

$(a)$ $(2,3,2)$ in weight 2 (cf. Example \ref{G2 ex});

$(b)$ $(1,2,1,2,1)$ in weight 4; and

$(c)$ $(1,1,1,1,1,1,1)$ in weight 6.\\
For weight 6, the complete list of possibilities is: $(c)$; twists
of $(a)$ and $(b)$; and $(2,0,0,3,0,0,2)$.

It was also asked in {[}\emph{op. cit.}{]} whether $G_{2}$ arises
as the Mumford-Tate group of a \emph{motivic} Hodge structure. This
is the transcendental question most closely related to Serre's famous
problem on motivic Galois groups \cite{Se}, and has received less
attention than the corresponding one for $\ell$-adic monodromy. It
turns out that recent work of Dettweiler and Reiter \cite{DR} on
Serre's problem produces a family of quasi-projective varieties over
$\mathcal{S}:=\mathbb{P}^{1}\backslash\{0,1,\infty\}$, the lowest
weight part of whose degree-$6$ cohomology carries a VHS of type
$(c)$ and MT group $G_{2}$. By \cite[Lemma 4]{An} (also cf. \cite[7.5]{De3}),
this is also the MT group of $W_{6}H^{6}$ of a very general fiber.

Consider the family\begin{equation}\xymatrix{
\mathcal{X} \ar @/^1pc/ [rr]^{\pi} \ar @{^(->} [r] & \mathcal{Y} \ar [r] & \mathcal{S} 
}\end{equation} of 6-folds defined by
\[
Y_{s}:=\left.\left(\PP^{1}\backslash\{0,1,\infty\}\right)^{\times6}\right\backslash \left\{ \left(\prod_{i=1}^{5}(x_{i}-x_{i+1})\right)(x_{6}-s)=0\right\} 
\]
\[
X_{s}:=\left\{ y^{2}=s(x_{6}-s)\prod_{i=1}^{5}(x_{i}-x_{i+1})\prod_{i=1,3,5}x_{i}\prod_{i=1,2,4,6}(x_{i}-1)\right\} \subset\CC^{*}\times Y_{s}
\]
with involution
\[
\sigma:\mathcal{X}\to\mathcal{X}
\]
sending $y\mapsto-y$. Denoting by $\left(\cdot\right)^{-}$ the $\sigma$-$(-1)$-eigenspace,
the local system
\[
\mathbb{V}:=\left(Gr_{6}^{W}R^{6}\pi_{!}\QQ\right)^{-}
\]
(with stalks $V_{s}$) induces a monodromy representation 
\[
\rho:\pi_{1}(\mathcal{S}_{\CC,s_{0}}^{an})\to Aut\left(V_{s_{0}}\right).
\]
with associated (geometric) VHS $\mathcal{V}$ and stalks $V_{s}$.
The \emph{geometric monodromy group} $\Pi$ of $\mathbb{V}$ is the
identity connected component of the ($\QQ$-)Zariski closure of the
image of $\rho$. Write $\mathcal{V}$ for the associated VHS and
$M_{\mathcal{V}}$ for its MT group; we recall that this identifies
with the MT groups of fibers $\mathcal{V}_{s}$ outside a countable
union of analytic subvarieties. Denoting unipotent Jordan blocks of
length $n$ by $J(n)$, we have the 
\begin{thm}
\cite{DR} (a) The monodromy types of $\mathbb{V}$ about $0$, $1$,
$\infty$ are $(\mathbf{-}\mathbf{1})^{\oplus4}\oplus\mathbf{1}^{\oplus3}$,
$J(2)^{\oplus2}\oplus J(3)$, and $J(7)$, respectively.

(b) $\Pi=G_{2}$.%
\footnote{In fact, they state this for the $\bar{\QQ}_{\ell}$-closure of $\text{im}(\rho)$
(cf. the proof of Theorem 3.3.1 in \cite{DR}), but since $\rho$
is defined rationally it makes no difference. %
}\end{thm}
\begin{cor}
(i) $M_{\mathcal{V}}=G_{2}$; and (ii) $\mathcal{V}$ has Hodge type
(c).\end{cor}
\begin{proof}
By the Theorem of the Fixed Part \cite{Sc}, $\Pi$ is normalized
by $M_{\mathcal{V}}$ (cf. \cite{An}), and so $G_{2}\trianglelefteq M_{\mathcal{V}}\leq SL(V_{s})$.
Using the fact that $N_{SL_{7}}(G_{2})=G_{2}$, $(i)$ follows at
once.

For $(ii)$, one can argue in two ways. On the one hand, the presence
of the $J(7)$ block means that $\mathcal{V}$ is not isotrivial ($\implies$
$(2,0,0,3,0,0,2)$ is impossible), and that $N^{6}\neq0$, which is
impossible for a VHS of level $<6$. Alternatively, one can show that
$\frac{dx_{1}\wedge\cdots\wedge dx_{6}}{y}$ extends to an anti-invariant
holomorphic 6-form on a $\sigma$-compatible good compactification
of $X_{s}$. Noting that $\mathbb{V}^{\vee}\cong\left(W_{6}R^{6}\pi_{*}\QQ\right)^{-}$,
this shows that
\[
\{0\}\neq\left(\mathcal{V}^{\vee}\right)^{6,0}\cong\left(\mathcal{V}^{0,6}\right)^{\vee}
\]
which at least rules out Hodge types \emph{(a)} and \emph{(b)}.
\end{proof}
Of course, the Theorem and Corollary are valid for $\mathbb{V}^{\vee}$
and $\mathcal{V}^{\vee}$ as well.

Each of the Hodge types $(a)$-$(c)$ corresponds to a projection%
\footnote{see \cite{GGK}, sections VI.B and IV.F.%
} on the weight diagram for the standard 7-dimensional irrep of $G_{2}$:
$$\includegraphics[scale=0.6]{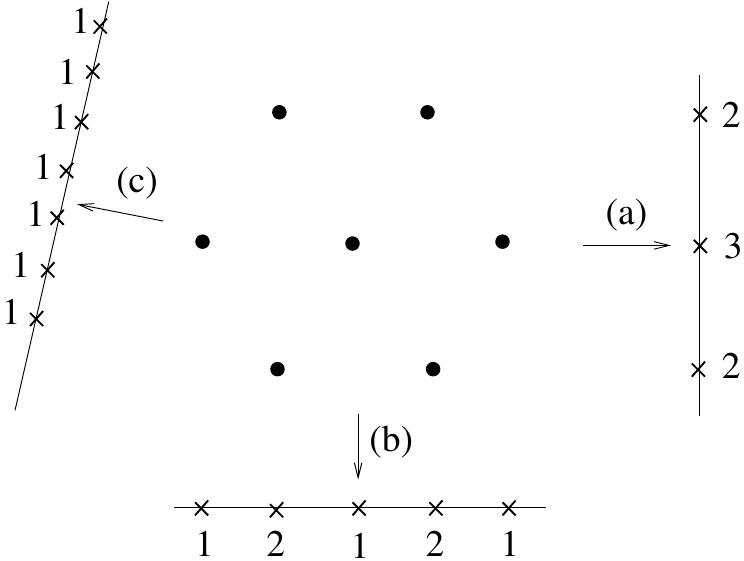}$$ Fix a Hodge type. Reasoning
heuristically, from this picture together with the root diagram for
$\mathfrak{g}_{2}$, $$\includegraphics[scale=0.6]{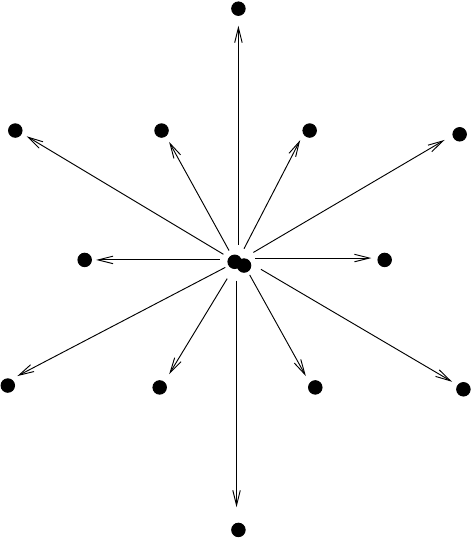}$$ one
may read off enough information about the possibilities for nilpotent
$N\in\mathfrak{g}_{2}$ satisfying $N(F^{\bullet})\subset F^{\bullet-1}$,
to classify the $I^{p,q}$-types of their associated limit mixed Hodge
structures (or boundary components).%
\footnote{We shall have more to say about this approach in the future work.%
} For type $(c)$, these are $$\includegraphics[scale=0.65]{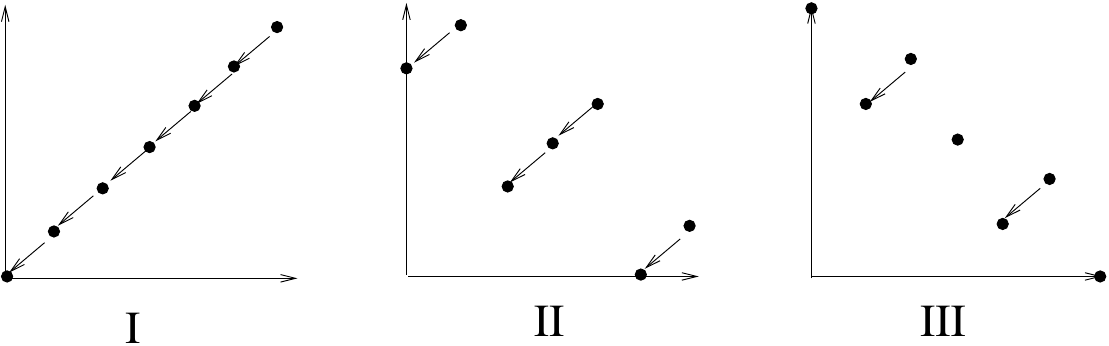}$$
the first two of which match the degenerations of $\mathcal{V}$ (or
$\mathcal{V}^{\vee}$) at $\infty$ and $1$ (resp.). It would be
very interesting to compute the LMHS at these points. In particular,
the adjoint $I^{p,q}$ diagram for type I shows that the corresponding
$\overline{B(N)}$ is a $\mathbb{C}^{*}$ which classifies the extension
of $\ZZ(-5)$ by $\ZZ(0)$ (or dually, $\ZZ(-6)$ by $\ZZ(-1)$).
One could also try to classify all type $(c)$ VHS over $\mathcal{S}_{\CC}^{an}$
with a type I and type II monodromy point, in the spirit of Doran
and Morgan \cite{DM}. We put these problems into a broader context
in the next section.

Finally, we would be remiss not to point out that the M-T domains
for type $(b)$ and $(c)$ both have boundary components isomorphic
to a (noncompact) elliptic modular surface. See $\S$6.1.3 in the follow-up paper
\cite{KP2}, where these correspond to the boundary strata labeled $\mathbf{o}^{\{e\}}_2$.

\section{Rigidity and boundary values}

To conclude this article, we wish to highlight how MT domains and
their boundary components provide a convenient structure in which
to think about rigidity and finiteness (Arakelov-type) results for
VHS. We fix once and for all a complex algebraic manifold $\mathcal{S}$
and a point $s_{0}\in\mathcal{S}$. In view of the result of \cite{De2},
that there are only finitely many $\QQ$-local systems $\mathbb{V}\to\mathcal{S}$
of given rank underlying an integral polarizable variation of Hodge
structure, we shall fix $\mathbb{V}$ as well. Writing $V:=\mathbb{V}_{s_{0}}$
($=\mathbb{Q}$-vector space), we have as in $\S9$ the monodromy
representation $\rho$ with image $\Gamma\subset GL(V)$ and geometric
monodromy group $\Pi\subset GL(V)$.

Let $\left(\mathcal{O}_{\mathcal{S}},\mathcal{F}^{\bullet},Q,\mathbb{V}\right)$
be weight $n$ PVHS over $\mathcal{S}$ (abbreviated ``$\mathcal{V}$''),
with MT group $M_{\mathcal{V}}\leq GL(V)$. As $\mathbb{V}$ is fixed,
$\mathcal{V}$ may be recovered from the associated period map%
\footnote{For each $s\in\mathcal{S}$, we may consider $\vf_{\mathcal{V},s}:=\Phi_{\mathcal{V}}(s)$
as a morphism $\mathbb{U}\to Aut(V,Q)$ (defined up to the action
of $\Gamma$).%
} \begin{equation}\label{eqn8.1} \Phi_{\mathcal{V}}:\mathcal{S}\to \Gamma \backslash D,\end{equation}
where $D$ is a (connected component of a) MT domain parametrizing
$Q$-polarized HS on $V$ with the same Hodge numbers $\underline{h}$
as $\mathcal{V}$ and MT group $M_{\mathcal{V}}$. We define the \emph{MT-class}
$[\mathcal{V}]$ of $\mathcal{V}$ (or $\Phi_{\mathcal{V}}$) to be
the \emph{set} of all ($\QQ$-)VHS on $\mathcal{S}$ with local system
$\mathbb{V}$, polarization $Q$, Hodge numbers $\underline{h}$,
and MT group $\leq M_{\mathcal{V}}\leq GL(V)$.%
\footnote{In some ways it is more natural to allow the MT group to be any conjugate
$gM_{\mathcal{V}}g^{-1}$ with $g\in Aut(V,Q)^{\Gamma}=:\mathcal{G}_{\QQ}$;
to get a finiteness result one then has to go modulo the equivalence
relation on such VHS induced by the action of $\mathcal{G}_{\QQ}$. %
}

Following \cite[sec. III.A]{GGK}, there is an almost-direct-product
decomposition $M_{\mathcal{V}}=M_{1}\cdots M_{\ell}\cdot A$, with
$\{M_{i}\}_{i=1}^{\ell}$ $\QQ$-simple, $A$ abelian, and $\Pi=M_{1}\cdots M_{k}=M_{var}$
for some $k\leq\ell$, write $M_{fix}=M_{k+1}\cdots M_{\ell}\cdot A$.
According to the Structure Theorem for VHS {[}op.cit.{]}, by passing
to a finite cover $\tilde{\mathcal{S}}$ (of $\mathcal{S}$) we may
replace \eqref{eqn8.1} by a period map of the form \begin{equation}\label{eqn8.2}\Phi_{\tilde{\mathcal{V}}}:\tilde{\mathcal{S}}\to \tilde{\Gamma}\backslash D_{var}\times D_{fix} \cong \left( \tilde{\Gamma} \backslash M_{var}(\RR)/H_{var}\right) \times M_{fix}(\RR)/H_{fix}\end{equation}
whose projection to $D_{fix}$ is \emph{constant}. (Here $\tilde{\mathcal{V}}$
is the pullback variation, and $\tilde{\Gamma}\leq\Gamma$.) Hodge-theoretically,
\eqref{eqn8.2} simply reflects the splitting \begin{equation}\label{eqn8.3}\lm_{\mathcal{V}}=\lm_{var}\oplus \lm_{fix}\end{equation}
of the weight-zero variation induced by $\tilde{\mathcal{V}}$. In
particular, we note that \begin{equation}\label{eqn8.4}\left(\lm \right)^{\tilde{\Gamma}} = \lm_{fix}.\end{equation}

By a \emph{horizontal local deformation} of $\Phi_{\mathcal{V}}$
in its MT-class, we shall mean an extension of $\mathcal{V}$ (from
$\mathcal{S}\times\{0\}$) to a PVHS over $\mathcal{S}\times\Delta$
with MT group $M_{\mathcal{V}}$. (Here $\Delta$ is the unit disk,
with coordinate $t$.) Define $\mathcal{V}$ (or $\Phi_{\mathcal{V}}$)
to be \emph{MT-rigid} if all such deformations are constant in $t$.

Write $\hat{\Phi}$ for the extension of \eqref{eqn8.2} induced by
such a deformation. Its derivative at $t=0$ gives a section \begin{equation}\label{eqn8.5}\nu:=\left. \frac{\partial \hat{\Phi}}{\partial t} \right|_{\tilde{\mathcal{S}}\times{0}} \in\Gamma\left( \tilde{\mathcal{S}},(\tilde{\mathcal{V}}^{\vee}\otimes\tilde{\mathcal{V}})^{-1,1}\right)\end{equation}
which, by a curvature argument \cite[Thm. 3.2(ii)]{Pe}, is \emph{flat}.
It is therefore a \emph{constant} section of the fixed part, and determined
by its value $\nu(s_{0})\in End(V)^{\Pi}$. In fact, since our deformation
was restricted to the MT domain, using \eqref{eqn8.4} \begin{equation}\label{eqn8.6}\nu\equiv\nu(s_0)\in\lm^{-1,1}_{fix}.\end{equation}Here,
the Hodge decomposition of $\lm_{fix}$ is constant over $\mathcal{S}$
by the Theorem of the Fixed Part \cite{Sc}, and induced by the projection
to $\lm_{fix}$ of $\text{Ad}(\vf_{\mathcal{V},s})$ (for any $s$,
say $s_{0}$).
\begin{prop}
If $\lm_{fix}^{-1,1}=\{0\}$, then $\mathcal{V}$ is MT-rigid.
\end{prop}
Turning to the cardinality of $[\mathcal{V}]$, instead of complex
variations (as in \cite{Pe},\cite{De2}) we shall make use of the
Structure Theorem and the following simple uniqueness result.
\begin{lem}
\label{lem10.2}Let $\mathcal{H}$ and $\mathcal{H}'$ be two polarizable
weight $n$ VHS over $\mathcal{S}$ with the same underlying local
system $\mathbb{H}$. If $\mathbb{H}_{\mathbb{C}}$ is irreducible
then $\mathcal{H}=\mathcal{H}'$.\end{lem}
\begin{proof}
According to the Theorem of the Fixed Part, $End_{\mathcal{S}}(\mathbb{H},\mathbb{H})$
must underlie a subvariation of $\mathcal{H}^{\vee}\otimes\mathcal{H}'$.
The hypothesis implies $End_{\mathcal{S}}(\mathbb{H},\mathbb{H})\cong\QQ\langle\text{id}_{\mathbb{H}}\rangle$
by Schur's lemma. Therefore $\text{id}_{\mathbb{H}}\in(\mathcal{H}^{\vee}\otimes\mathcal{H}')^{(0,0)}$
and is a morphism of VHS.
\end{proof}
Note that in the Lemma, we need \emph{not} assume that $\mathcal{H}$
and $\mathcal{H}'$ have the same Hodge numbers or polarization.
\begin{example}
\label{10ex}If $\mathcal{S}$ is a curve and $\mathbb{H}$ has maximal
unipotent monodromy of order equal to its rank about a point of $\bar{\mathcal{S}}\backslash\mathcal{S}$,
there is at most one polarizable VHS on $\mathbb{H}$. This is the
case both for the ``Calabi-Yau variations'' investigated by Doran
and Morgan \cite{DM} and for the $G_{2}$-variation in $\S9$.
\end{example}
In the period domain for $Q$-polarized HS on $V$ with Hodge numbers
$\underline{h}$, the locus $NL(M_{\mathcal{V}})$ of HS with MT group
contained in $M_{\mathcal{V}}$ is a finite union of (connected) MT
domains. Within each such domain, a HS $\vf$ is determined by the
weight-zero structure it induces on $\lm_{\mathcal{V}}$, which necessarily
splits into $\lm_{1}\oplus\cdots\oplus\lm_{\ell}\oplus\mathfrak{a}$
as these summands are closed under the adjoint action of $M_{\mathcal{V}}$.
Further, since $\text{Ad}\circ\vf$ acts trivially on the abelian
part $\mathfrak{a}$, we have $\mathfrak{a}_{\CC}=\mathfrak{a}^{0,0}$.

Write $\underline{m}=\underline{m}_{1}\oplus\cdots\oplus\underline{m}_{\ell}\oplus\underline{a}\subset\mathbb{V}^{\vee}\otimes\mathbb{V}$
for the $\QQ$-local system arising from $\rho^{\vee}\otimes\rho|_{\lm_{\mathcal{V}}}$.
Let $\mathsf{m}_{\mathcal{V}}=\mathsf{m}_{\mathcal{V},1}\oplus\cdots\oplus\mathsf{m}_{\mathcal{V},\ell}\oplus\mathsf{a}$
and $\mathsf{m}_{\mathcal{W}}=\mathsf{m}_{\mathcal{W},1}\oplus\cdots\oplus\mathsf{m}_{\mathcal{V},\ell}\oplus\mathsf{a}_{\mathcal{W}}$
denote the VHS on $\underline{m}$ induced by $\mathcal{V}$ and some
other variation $\mathcal{W}\in[\mathcal{V}]$ with period mapping
$\Phi_{\mathcal{W}}:\mathcal{S}\to\Gamma\backslash D$ into the same
MT domain. Assuming that the $\{M_{i}\}_{1\leq i\leq k}$ are $\CC$-simple,
the $\{\underline{m}_{i}\}_{1\leq i\leq k}$ are absolutely irreducible
and Lemma \ref{lem10.2} immediately implies that $\mathsf{m}_{\mathcal{V},i}=\mathsf{m}_{\mathcal{W},i}$
($1\leq i\leq k$). Noting that $\mathsf{a}_{\mathcal{V}}=\mathsf{a}_{\mathcal{W}}$
(as both are trivial) and $DM:=[M,M]=M_{1}\cdots M_{\ell}$, we have
proved the 
\begin{thm}
If $\Pi=DM_{\mathcal{V}}$ and the $\QQ$-simple factors $M_{i}$
of $M_{\mathcal{V}}$ are absolutely simple, then $[\mathcal{V}]$
is finite; more precisely, we have $|[\mathcal{V}]|\leq|\pi_{0}(NL(M_{\mathcal{V}}))|\,(<\infty)$.\end{thm}
\begin{rem}
\label{10rm}(i) Hypotheses on $\mathcal{V}$ which ensure $\Pi=DM_{\mathcal{V}}$
are the presence of a CM point \cite[sec. 6]{An} or a graded-CM LMHS
\cite[sec. 8]{KP}.

(ii) An approach to computing $|\pi_{0}(NL(M_{\mathcal{V}}))|$ is
described in \cite[Ch. 6]{GGK}.
\end{rem}
Beyond their relation to the hypotheses in Example \ref{10ex} and
Remark \ref{10rm}(i), boundary components provide (together with
the choice of $D\in\pi_{0}(NL(M_{\mathcal{V}}))$) a means of \emph{parametrizing}
$[\mathcal{V}]$. It is well-known (cf. \cite[Cor. 12]{PS}) that
any $\mathcal{W}\in[\mathcal{V}]$ is determined by its restriction
to $s_{0}$; in the same spirit, one has the next
\begin{prop}
Suppose $\mathcal{S}$ is a curve, and $x_{0}\in\bar{\mathcal{S}}\backslash\mathcal{S}$
a point. Then any $\mathcal{W}\in[\mathcal{V}]$ is determined by
its LMHS $\psi_{x_{0}}\mathcal{W}$ at $x_{0}$.\end{prop}
\begin{proof}
Let $\mathcal{W},\mathcal{W}'\in[\mathcal{V}]$, and put $\mathcal{E}:=\mathcal{W}^{\vee}\otimes\mathcal{W}'$
(with underlying local system $\mathbb{V}^{\vee}\otimes\mathbb{V}$).
Since the fixed part $E_{fix}\otimes\mathcal{O}_{\mathcal{S}}\cong\mathcal{E}_{fix}\underset{\alpha}{\hookrightarrow}\mathcal{E}$
is a (constant) sub-VHS, 
\[
E_{fix}\cong\psi_{x_{0}}\mathcal{E}_{fix}\underset{\alpha_{lim}}{\hookrightarrow}\psi_{x_{0}}\mathcal{E}\cong(\psi_{x_{0}}\mathcal{W})^{\vee}\otimes(\psi_{x_{0}}\mathcal{W}')
\]
is a sub-MHS. If $\psi_{x_{0}}\mathcal{W}\cong\psi_{x_{0}}\mathcal{W}'$,
then $\alpha_{lim}(\text{id}_{V})$ is Hodge of type $(0,0)$ in $\psi_{x_{0}}\mathcal{E}$.
By strictness of morphisms of MHS, $\text{id}_{V}$ is Hodge in $E_{fix}$.
Hence, $\text{id}_{V}$ is Hodge $(0,0)$ in $\mathcal{E}_{fix}$
\emph{a fortiori} in $\mathcal{E}$, and $\mathcal{W}=\mathcal{W}'$.
\end{proof}
When $\mathcal{V}$ is rigid in $[\mathcal{V}]$ (let alone $|[\mathcal{V}]|<\infty$),
it is natural to expect that the LMHS at $x_{0}$ has arithmetic significance
-- particularly if $\mathcal{V}$ is motivic and $\mathcal{S}$ defined
over $\bar{\mathbb{Q}}$ (cf. \cite[Conj. (III.B.5)]{GGK2}). A classic
example of this is the class of the LMHS of the mirror quintic $(1,1,1,1)$-VHS
at the maximal unipotent monodromy point, which is given by $-200\zeta(3)\in\CC/\QQ(3)\cong Ext_{\text{MHS}}^{1}(\QQ(0),\QQ(3))$
{[}op. cit., sec. III.A{]}. We expect that the Picard-Fuchs equations
identified in \cite{DM} will allow one to examine how this invariant
changes with the choice of local system. Moreover, it seems likely
that the corresponding invariant (at the type I boundary component)
for the $G_{2}$ example of $\S9$ is a rational multiple of $\zeta(5)$.
See \cite[$\S$5]{DKP} for some progress in this direction.

\address{\noun{Department of Mathematics, Washington University in St. Louis,
Campus Box 1146, One Brookings Drive, St. Louis, MO} \noun{63130-4899,
USA}\\
\emph{E-mail address}: matkerr@math.wustl.edu}

${}$

\address{\noun{Mathematics Department, Mail stop 3368, Texas A\&M University, College Station, TX 77843, USA}\\
\emph{E-mail address}: gpearl@math.tamu.edu}
\end{document}